\documentclass[tocnosub,noragright,centerchapter,12pt]{uiucecethesis09}

\makeatletter

\usepackage{setspace}
\usepackage{epsfig}  
\usepackage{graphicx}  
\usepackage{subfigure}  
\usepackage{amsmath}  
\usepackage{amssymb}  
\usepackage{url}  
\usepackage{lscape}  
\usepackage[justification=raggedright]{caption}	

\usepackage{cite}
\usepackage{amsthm}
\usepackage{enumitem}
\phdthesis

\newtheorem{theorem}{Theorem}[chapter]

\newtheorem{conjecture}{Conjecture}[chapter]
\numberwithin{equation}{chapter}
\newcommand{\R}{\mathbb{R}} 
\newcommand{\C}{\mathbb{C}} 
 
\newcommand{\Z}{\mathbb{Z}} 
\newcommand{\N}{\mathbb{N}} 
\newcommand{\F}{\mathbb{F}} 
\newcommand{\Mat}{\text{Mat}}
\renewcommand{\vec}{\text{vec}}

\renewcommand{\P}{\mathbb{P}}
\newcommand{\E}{\mathbb{E}}

\newcommand{\Cov}{\text{Cov}}

\title{Cyclicity Analysis of the Ornstein-Uhlenbeck Process}
\author{Vivek Kaushik}
\department{Mathematics}
\degreeyear{2024}

\advisor{Yuliy Baryshnikov}

\committee{Professor Vadim Zharnitsky, Chair\\
	Professor Yuliy Baryshnikov, Director of Research \\
       Professor Richard Sowers \\ Professor Lee DeVille}

\begin{document}

%

\maketitle

\parindent 1em%

\frontmatter

%
\begin{abstract} 
In this thesis, we consider an $N$-dimensional Ornstein-Uhlenbeck (OU) process $\left \lbrace \mathbf x(t) \right \rbrace_{t \ge 0}$ satisfying the linear stochastic differential equation $d \mathbf x(t) = - \mathbf B \ \mathbf x(t) \ dt + \boldsymbol \Sigma \  d \mathbf w(t).$ Here, $\mathbf B$ is a fixed $N \times N$ circulant friction matrix whose eigenvalues have positive real parts, $\boldsymbol \Sigma$ is a fixed $N \times M$ matrix for some $M \in \N,$ and $\left \lbrace \mathbf w(t) \right \rbrace_{t \ge 0}$ is the standard $M$-dimensional Wiener process. 

We consider a signal propagation model governed by this OU process. In this model, an underlying signal propagates throughout a network consisting of $N$ linked sensors located in space. For each $t \ge 0,$ we interpret $x_n(t),$ the $n$-th component of the OU process at time $t,$ as the measurement of the propagating effect made by the $n$-th sensor. The matrix $\mathbf B$ represents the sensor network structure: if $\mathbf B$ has first row $(b_1 \ , \ \dots \ , \ b_N),$ where $b_1>0$ and $b_2 \ , \ \dots \ ,\ b_N \le 0,$ then the magnitude of $b_p$ quantifies how receptive the $n$-th sensor is to activity within the $(n+p-1)$-th sensor, where $n+p-1$ is indexed mod $N.$ Finally, the $(m,n)$-th entry of the matrix $\mathbf D = \frac{\boldsymbol \Sigma  \boldsymbol \Sigma^\text T}{2}$ is the covariance of the component noises injected into the $m$-th and $n$-th sensors.

For different choices of $\mathbf B$ and $\boldsymbol \Sigma,$ we investigate whether Cyclicity Analysis   enables us to recover the structure of network. Roughly speaking, Cyclicity Analysis studies the lead-lag dynamics pertaining to the components of a multivariate signal. We specifically consider an $N \times N$ skew-symmetric matrix $\mathbf Q,$ known as the lead matrix, in which the sign of its $(m,n)$-th entry captures the lead-lag relationship between the $m$-th and $n$-th component OU processes. We investigate whether the structure of the leading eigenvector of $\mathbf Q,$ the eigenvector corresponding to the largest eigenvalue of $\mathbf Q$ in modulus, reflects the network structure induced by $\mathbf B.$

\end{abstract}

%

%
\begin{acknowledgments}
Many people have supported me throughout my Ph.D. journey. I would like to acknowledge them here.

Firstly, I would like to thank my advisor Professor Baryshnikov for his guidance. Professor Baryshnikov introduced me to the topic of Cyclicity Analysis and strongly encouraged me to explore its applications on real-world data. Later on, we formulated the topic for this thesis. Throughout my Ph.D. journey, Professor Baryshnikov instilled the meaning of a Ph.D. scholar: a scholar who is able to conduct independent research. Via his meticulous comments and feedback on my dissertation drafts, he taught me that being a good researcher does not just involve doing research but also involves articulating my findings in a clear and concise manner. 

Next, I would like to thank all the other thesis committee members.  During my undergraduate career, Professor Sowers supervised my first data science project, which initiated my curiosity for the field. He also taught me Stochastic Calculus, a course that was essential to my thesis research. Professor DeVille strongly encouraged me to pursue a Ph.D. at the University of Illinois and taught me Dynamical Systems, another course that was essential to my thesis research.  Although I did not take courses or do research with him, Professor Zharnitsky and I would communicate frequently especially in Russian. I thank him for kindly lending me a book on the Russian language.

Furthermore, I would like to thank Professor Reznick and Professor D'Angelo from the University of Illinois and Professor Ritelli from the University of Bologna. All professors played a significant role during my early research career. I enrolled as an undergraduate freshman in Professor Reznick's senior level research seminar, during which I learned how to conduct and share research. The work done in that seminar would later turn into joint work with Professor Ritelli during my undergraduate career. Professor Ritelli guided me in publishing my first undergraduate paper in the AMS Quarterly of Applied Mathematics. Lastly, Professor D'Angelo has always been supportive of my research and regularly encouraged me to share my research findings with him. He also shared with me his insights on the math department and on issues related to the broader mathematical community.

Moreover, I thank all of my colleagues at eBay. During my Ph.D. journey, I had the opportunity to work at eBay as a data scientist and collaborate with many experts in the field. eBay gave me a unique platform to showcase my skills in solving business related problems.  My colleagues were very supportive of me and helped me grow as a data scientist. The work I did at eBay helped shaped my decision to pursue a Ph.D. 

Finally, I would like to thank my loving family. Admittedly, I faced various challenges pertaining to the completion of the thesis. But my family has always been extremely supportive of me nonetheless. 

\end{acknowledgments}

%
\tableofcontents

%

%

%


%

\mainmatter

%

\chapter{Introduction}\label{chap: Introduction}

Throughout this chapter, we fix $N \in \N$ and let $T$ be an interval in $\R,$ which we regard as an interval of time. We let $\mathbf x: T \rightarrow \R^N$ be a continuous function. Note that we can write $\mathbf x = (x_1 \ , \ \dots \ , \ x_N),$ in which $x_n: T \rightarrow \R$ is a real-valued continuous function. We refer to $x_n$ as the $n$-th \textit{component} of $\mathbf x.$

In many practical applications, $\mathbf x$ is a \textit{signal (wave/pulse)} representing some observed phenomenon, in which we interpret $\mathbf x(t)$ as a measurement of the observed phenomenon recorded at the fixed time $t \in T.$  For example, if we are observing the behavior of the entire stock market, then we could consider an $N$-dimensional signal $\mathbf x: T \rightarrow \R^N,$ where $N$ is the total number of companies in the market; here, the component $x_n(t)$ could be the stock price for the $n$-th company at the time $t \in T.$

\section{Leader-Follower Dynamics Questions}\label{sec: Leader-Follower Dynamics Questions}
In this section, we pose two very general questions pertaining to the leader-follower dynamics amongst the $N$ components of $\mathbf x$ relative to their evolution throughout time. These questions are central to the thesis as a whole.

First, fix two specific indices $1 \le m,n \le N$ and assume that the corresponding component signals $x_m$ and $x_n$ evolve very similarly up to a time shift. Does $x_n$ \textit{lead (precede)} or \textit{follow (lag)} $x_m$ throughout time ? To illustrate the meaning of this question, we consider the signals $f,g: \R \rightarrow \R$ defined by $f(t)=\exp(-\pi t^2)$ and $g(t)=f(t-1).$ We plot these signals in \textbf{Figure \ref{fig: SampleLeaderFollowerTimeSeriesExample}}.
Note that $g$ is a horizontal translation of $f$ to the right by $1$ unit of time, and the number $1$ is the value of the time shift. As a result, throughout time, we can describe the evolution of $f$ and $g$ in relation to one another:  when $f$ increases, $g$ later increases, and when $f$ decreases, $g$ later decreases. Thus, we say that $f$ leads $g,$ or equivalently, $g$ follows $f.$

\begin{figure}[h!]
	\centering
	\includegraphics[scale=0.8]{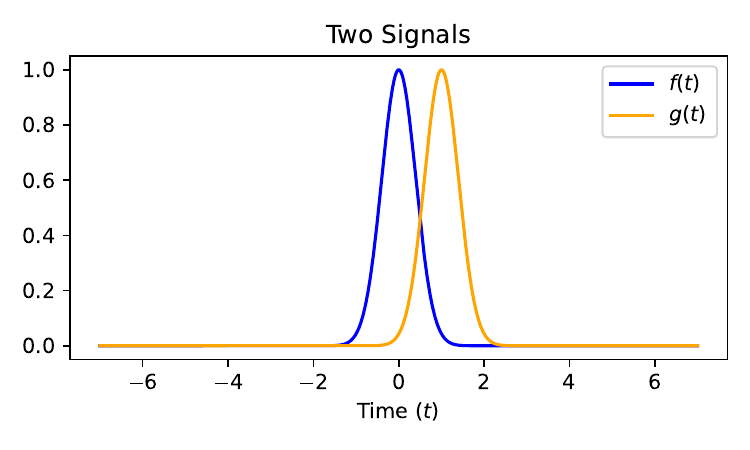}
	\caption{An illustration of two signals exhibiting a leader follower relationship. Here, we plot $f(t)=e^{- \pi t^2}$ and $g(t)=f(t-1).$ We interpret $f$ to lead $g.$  }\label{fig: SampleLeaderFollowerTimeSeriesExample}
\end{figure}

Next, assume all of the $N$ component signals of $\mathbf x$ trace some underlying real-valued signal up to scaling constants and time shifts. What is the order in which the component signals evolve throughout time ? More specifically, what is the order of the time shifts ? To illustrate the meaning of this question, let $T$ be the time interval $[0,1],$ and consider the sinusoidal signal $\mathbf x: [0,1] \rightarrow \R^N,$ whose $n$-th component signal is defined by
	\begin{align}\label{eq: Sinusoidal Signal}
		x_n(t) &= \sin \left(2 \pi \left(t - \frac{n-1}{N} \right) \right)
	\end{align}
for each $1 \le n \le N.$  In \textbf{Figure \ref{fig: SamplePeriodicCOOMExample}}, we fix $N=20$ and plot some of these component sinusoidal signals. Each $x_n$ is a time-shifted copy of the signal $\phi(t)=\sin(2 \pi t)$ with $\frac{n-1}{N}$ being the value of the time shift. Since $\frac{n-1}{N} < \frac{n}N$ for each $1 \le n \le N,$ we interpret the $n$-th component signal as leading the $(n+1)$-th component signal. Thus, the
order in which the $N$ component signals evolve throughout time is $x_1 \ , \ \dots \ , \ x_{N}.$ 

\begin{figure}[h]
	\centering
	\includegraphics[scale=0.6]{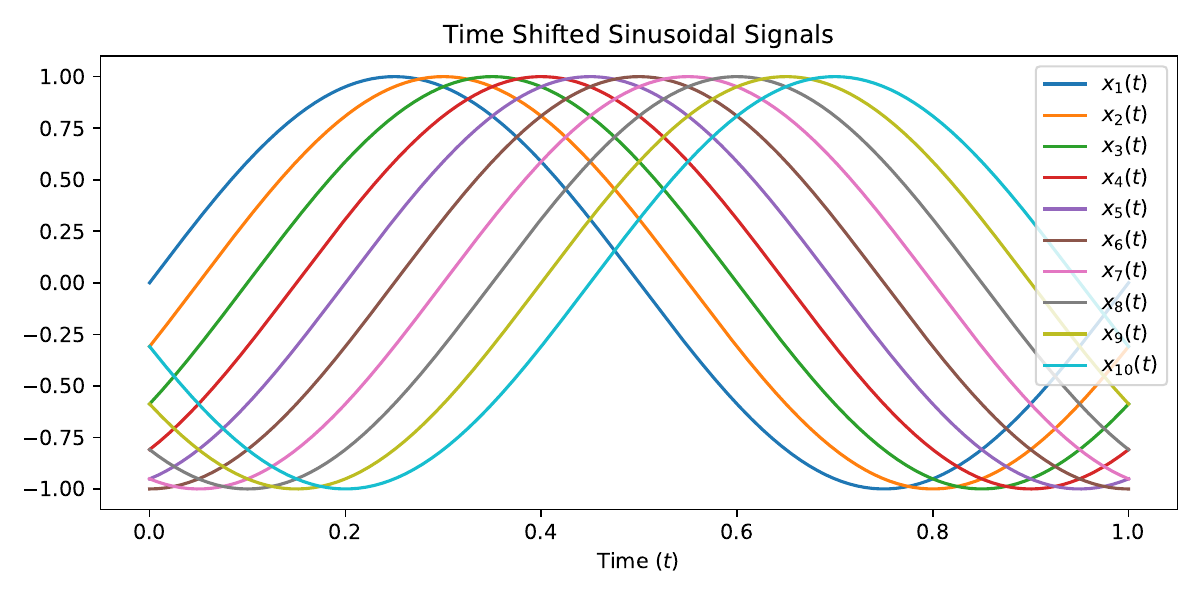}
	\caption{An illustration of multiple signals tracing the common signal $\phi(t)=\sin(2 \pi t)$. Here,  $x_n(t)=\sin \left( 2 \pi \left(t- \frac{n-1}{20} \right) \right)$ for each $1 \le n \le 10.$ We plot the $10$ signals in different colors.  } \label{fig: SamplePeriodicCOOMExample}
\end{figure}

\section{Analysis of Periodic Signals}\label{sec: Analysis of Periodic Signals}

In this section, we focus on analyzing periodic signals. Recall that a signal $\mathbf x: \R \rightarrow \R^N$ is \textit{periodic} if there exists a minimal constant $P>0$ such that $\mathbf x(t+P)=\mathbf x(t)$ for all $t \in \R.$ The constant $P$ is known as the \textit{period} of $\mathbf x.$ We specifically say $\mathbf x$ is $P$-periodic in order to emphasize $\mathbf x$ has period $P$. Equivalently, we can view a $P$-periodic signal as a well-defined map between $\R/(P \Z)$ and $\R^N,$ where $\R/(P \Z)$ is the additive group $[0,P)$ equipped with addition modulo $P.$ 

To analyze periodic signals in general, one uses tools from harmonic analysis, which include \textit{Fourier series} \cite[Chapter 2]{Stein2011} and the \textit{Fourier transform} \cite[Chapter 5]{Stein2011}. Another common tool that is suitable for analyzing periodic signals is the cross-correlation function \cite[Chapter 12]{Papoulis1967}.
Recall if $f,g: \R \rightarrow \R$ are two $P$-periodic signals, then the \textit{(periodic) cross-correlation} of $f$ and $g$ is the $P$-periodic signal $f \star g: \R \rightarrow \R$ defined by
\begin{align}\label{eq: Periodic Cross-Correlation Function}
	(f \star g)(\tau) &= \int_0^P f(t-\tau) \  g(t) \ dt.
\end{align}
The cross-correlation function measures the similarity between the input signals $f$ and $g$ as a function of the displacement variable $\tau.$ The displacement at which the cross-correlation attains its absolute maximum reveals information about $f$ and $g$. Up to an integer multiple of the period, the maximum displacement coincides with the time shift value such that shifting $f$ by this value would result in the signals $f$ and $g$ being most aligned. 

For example, let $f,g : \R \rightarrow \R$ be the periodic signals $f(t)=\sin(2 \pi t)$ and $g(t)= \cos(2 \pi t).$ Explicitly computing the cross-correlation of $f$ and $g$ using \eqref{eq: Periodic Cross-Correlation Function}, we obtain 
\begin{align}
	(f \star g)(\tau) &= \int_0^1 \sin(2 \pi (t- \tau)) \cos(2 \pi t) \ dt \nonumber  \\
	&=  \int_0^1 \left(\sin(2 \pi t) \cos(2 \pi \tau) - \sin(2 \pi \tau) \cos(2 \pi t)\right) \cos(2 \pi t) \ dt  \nonumber \\
	&=  \int_0^1 \sin(2 \pi t) \cos(2 \pi t) \cos(2 \pi \tau) - \sin(2 \pi \tau) \cos^2(2 \pi t) \ dt  \nonumber \\
	&=  \int_0^1 \frac{\sin(4 \pi t)}{2} \cos(2 \pi \tau) - \sin(2 \pi \tau) \left(\frac{1-\cos(4 \pi t)}{2} \right) \ dt \nonumber  \\
	&= - \frac{\sin(2 \pi \tau)}{2} \nonumber.
\end{align}
We plot this cross-correlation function in \textbf{Figure \ref{fig: SamplePeriodicCrossCorrelationExample}}.  On one hand, note that $g$ is a time-shifted copy of $f$ to the right by $\frac{3}{4}$ units of time. On the other hand, we observe the cross-correlation function $f \star g$ attains its maximum at $\tau = \frac{3}{4}+k$ for any $k \in \Z.$ This means the displacement at which the cross-correlation function attains its maximum coincides with the time shift between $f$ and $g$ up to an integer.

\begin{figure}[h]
	\centering
	\includegraphics[scale=0.6]{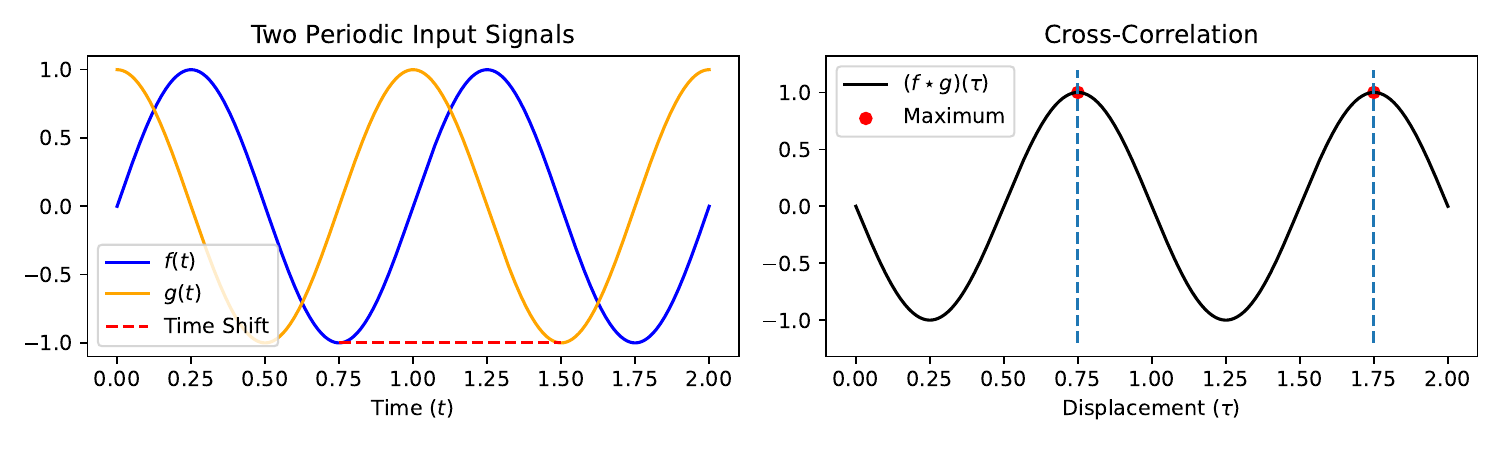}
	\caption{An example of cross-correlation. Let $f,g : \R \rightarrow \R$ be the signals $f(t)=\sin(2 \pi t)$ and $g(t)=\cos(2 \pi t).$ We plot $f$ and $g$ on the left, and we use a dashed red line to indicate the time shift needed to align $f$ with $g,$ namely $\frac{3}{4}$ units of time. We plot the cross-correlation $f \star g$ on the right and indicate its absolute maxima with red dots. Observe the time shift needed to align $f$ with $g$ coincides with the displacement at which $f \star g$ takes on its maximum up to an integer. } \label{fig: SamplePeriodicCrossCorrelationExample}
\end{figure}

\subsection{The Chain of Offsets Model}\label{subsec: Chain Of Offsets Model}
We answer the leader follower dynamics questions that we posed in \textbf{Section \ref{sec: Leader-Follower Dynamics Questions}} for periodic signals. We assume that $\mathbf x:\R \rightarrow \R^N$ obeys a \textit{Chain of Offsets model (COOM)}, which was first introduced in \cite{BaryshnikovSchlafly2016}. This model stipulates the existence of a $P$-periodic signal $\phi: \R \rightarrow \R,$ scaling constants $c_1 \ , \ \dots \ , \ c_N >0,$ and \textit{offsets} $\alpha_1 \ , \ \dots \ , \ \alpha_N \in \R/(P \Z)$ such that 
	\begin{align}\label{eq: Periodic COOM}
		x_n(t) &= c_n \phi(t-\alpha_n)
	\end{align}
for each $1 \le n \le N.$ In other words, under COOM, all component signals of $\mathbf x$ trace an underlying real-valued periodic signal up to scaling constants and offsets. 

Answering our posed leader follower dynamics questions amounts to determining the order of the offsets $\alpha_1 \ , \ \dots \ , \ \alpha_N$ under COOM. However, we emphasize this order of offsets is a \textit{cyclic order}; it corresponds to the order in which the points $e^{2 \pi i \alpha_1} \ , \ \dots \ , \ e^{2 \pi i \alpha_N}$ are traversed on the unit circle in the counterclockwise direction. We represent the cyclic order as a permutation $\sigma \in S_N$ of the indices $\left \lbrace 1\ , \ \dots \ , \ N \right \rbrace$ satisfying the condition that $e^{2 \pi i \alpha_{\sigma(n)}}$ comes before $e^{2 \pi i \alpha_{\sigma(n+1)}}$ in the counterclockwise direction. 
In terms of the signals themselves, we interpret $x_{\sigma(n)}$ as (cyclically) preceding $x_{\sigma(n+1)}$ for each $1 \le n <N.$ Furthermore, $\sigma$ is unique up to cyclic shifts of indices. This means if $\widetilde{\sigma} \in S_N$ is another permutation representing the cyclic order of the offsets, then there exists an integer $d$ such that 	
 	\begin{align} \nonumber
 		\widetilde{\sigma}(n)=\sigma(d+n-1)
 	\end{align}
for all $1 \le n \le N,$ in which $d+n-1$ is indexed mod $N.$ 
 
We give an example of how a permutation representing the cyclic order is unique up to cyclic shifts. Reconsider the sinsoidal signal $\mathbf x$ with $n$-th component signal defined in \eqref{eq: Sinusoidal Signal} for each $1 \le n \le N.$ Note that $\mathbf x$ satisfies COOM in \eqref{eq: Periodic COOM} if we set $\phi(t)=\sin(2 \pi t)$ and $c_n=1$ and $\alpha_n=\frac{n-1}{N}$ for each $1 \le n \le N.$ On one hand, note that $0 = \alpha_1 < \ \dots \ < \alpha_N <1,$ which means the points $e^{2 \pi i \alpha_1} \ , \ \dots \ , \ e^{2 \pi i \alpha_N}$ are traversed in that order on the unit circle in the counterclockwise direction beginning from the positive real axis. Hence, the identity permutation $\text{Id} \in S_N$ is one permutation that represents the cyclic order of the offsets. In fact, any permutation $\widetilde{\sigma} \in S_N$ defined by $\widetilde{\sigma}(n)=n+d,$ where $d \in \N$ is a fixed integer and $n+d$ is indexed mod $N,$ is a valid permutation representing the cyclic order of the offsets. Nevertheless, observe that 
    \begin{align}
        \text{Id}(n) &= \widetilde{\sigma}(n+N-d),
    \end{align}
which means the identity permutation is equal to $\widetilde{\sigma}$ up to $N-d$ cyclic shifts.

\subsection{Recovery of the Cyclic Order under COOM}\label{subsec: Recovery of Cyclic Order under COOM}

Given the $P$-periodic signal $\mathbf x,$ we outline a procedure to recover the cyclic order of the offsets $\alpha_1 \ , \ \dots \ , \ \alpha_N$ under COOM. This procedure involves the utilization of the cross-correlation function, which we defined previously in \eqref{eq: Cross-Correlation Function}.

First, we write the underlying $P$-periodic signal $\phi$ in COOM as a Fourier series:	
	\begin{align}\label{eq: Periodic COOM Primary Function Fourier Series}
		\phi(t) &= \sum_{k \in \Z} \hat{\phi}_k \  e^{\frac{2 \pi i k t}{P}},
	\end{align}
where $\hat{\phi}_k$ is the $k$-th Fourier coefficient of $\phi$ for each $k \in \Z.$ Then, for each $1 \le m,n \le N,$ the cross-correlation of the component signals $x_m$ and $x_n$ is explicitly    
	\begin{align}
		(x_m \star x_n)(\tau) &= \int_0^P x_m(t -\tau) \ x_n(t) \ dt \nonumber  \\
		&= \int_0^{P}  \sum_{k,\ell \in \Z} c_m c_n  \  \hat{\phi}_{k} \  \hat{\phi_\ell} \  e^{\frac{2 \pi i k(t-\tau -\alpha_m)+ 2 \pi i \ell (t-\alpha_n)}{P}} \  dt \nonumber \\
		&= c_m c_n \sum_{k,\ell \in \Z}  \hat{\phi}_{k}  \ \hat{\phi_\ell}  \ e^{\frac{-2 \pi i k (\tau + \alpha_m) - 2 \pi i \ell \alpha_n }{P}} \int_0^{P}  e^{\frac{2 \pi i (k+\ell) t}{P}} dt \nonumber  \\
		&= c_m c_n \sum_{k,\ell \in \Z}  \hat{\phi}_{k}  \ \hat{\phi_\ell}  \ e^{\frac{-2 \pi i k (\tau + \alpha_m) - 2 \pi i \ell \alpha_n }{P}} \ P \  \delta_{k,-\ell}  \nonumber \\
		&= P  c_m c_n \sum_{k \in \Z}  \left |\hat{\phi}_{k} \right|^2   \ e^{\frac{2 \pi i k ( \alpha_n-\alpha_m-\tau) }{P}} \nonumber \\
		&=  P  c_m c_n \left|\hat{\phi}_{0} \right|^2  + 2P  c_m c_n \sum_{k \in \N }  \left |\hat{\phi}_{k} \right|^2 \cos \left(\frac{2 \pi k (\alpha_n-\alpha_m - \tau)}{P} \right) \nonumber,
	\end{align}
where $\delta_{\cdot ,\cdot}$ is the Kronecker delta function. Differentiating the cross-correlation function with respect to $\tau$, we obtain
	\begin{align}
		(x_m \star x_n)'(\tau) &= -4 \pi   c_m c_n \sum_{k \in \N}  k \left |\hat{\phi}_{k} \right|^2 \sin \left(\frac{2 \pi k (\alpha_n-\alpha_m - \tau)}{P} \right) \nonumber.
	\end{align}
Equating the derivative to $0,$ we see that any critical point of $x_m \star x_n$ is of the form
	\begin{align}
		\tau &= \alpha_n-\alpha_m + k P, \nonumber
	\end{align}
where $k \in \Z.$ This means that the critical points of the $x_m \star x_n$ are equal to the difference between the offsets $\alpha_n$ and $\alpha_m$ up to integer multiples of the period $P.$

Now, we are ready to construct a permutation of the indices $\left \lbrace 1 \ , \ \dots \ , \ N \right \rbrace$ representing the cyclic order of the offsets under COOM. For each $1 \le m,n \le N,$ let 
	\begin{align}
		\tau_{m,n} &= \alpha_n-\alpha_m \nonumber,
	\end{align} 
in which $\alpha_n-\alpha_m$ is specifically reduced modulo $P.$ More explicitly, we choose an integer $k \in \Z$ such that $\alpha_n-\alpha_m + kP \in [0,P)$ and set $\tau_{m,n}$ equal to $\alpha_n-\alpha_m + kP$ for this choice of $k.$ Define the permutation $\sigma \in S_N$ be a permutation as follows: let $\sigma(1)$ be arbitrary, and for each $1 < n \le N,$ let $\sigma(n)$ be the smallest index not in $\left \lbrace \sigma(1) \ , \ \dots  \  , \ \sigma(n-1) \right \rbrace$ such that $\tau_{\sigma(n-1) \ , \ \sigma(n)}$ is minimal i.e. we have $\tau_{\sigma(n-1) \ , \ \sigma(n)} \le \tau_{\sigma(n-1) \ , \ m}$ for all $m \notin \left \lbrace \sigma(1) \ , \ \dots  \  , \ \sigma(n-1) \right \rbrace.$ Observe $\sigma$ ensures that $e^{2 \pi i \alpha_{\sigma(n)}}$ is the closest point on the unit circle to the point $e^{2 \pi i \alpha_{\sigma(n-1)}}$ in the counterclockwise direction. Therefore, $\sigma$ is a candidate permutation representing the cyclic order.

It remains to show that our constructed permutation $\sigma$ from the previous paragraph is unique up to cyclic shifts. To this end, let $\widetilde{\sigma} \in S_N$ be another permutation constructed according to the procedure in the previous paragraph. Note that there exists a unique index $1 \le d \le N$ such that $\widetilde{\sigma}(1) =\sigma(d),$ which follows from the fact that the composition $\widetilde{\sigma} \circ \sigma^{-1}$ is itself a permutation in $S_N.$ We claim that for this chosen index $d,$ we have 
	\begin{align}\label{eq: Periodic COOM Cyclic Order Uniqueness}
		\widetilde{\sigma}(n) &= \sigma(d+n-1)
	\end{align}
for all $1 \le n \le N,$ in which $d+n-1$ is indexed mod $N.$ We prove this by induction. The statement \eqref{eq: Periodic COOM Cyclic Order Uniqueness} already holds for $n=1$ by the property of $d.$ Now, suppose the statement \eqref{eq: Periodic COOM Cyclic Order Uniqueness} holds for $n=k$ for some $1<k<N. $ We show it must also hold for $n=k+1.$ By the construction of $\widetilde{\sigma},$ we have that $\widetilde{\sigma}(k+1)$ is the smallest index not in $\left \lbrace \widetilde{\sigma}(1) \ , \ \dots \ , \ \widetilde{\sigma}(k) \right \rbrace$ such that $\tau_{\widetilde{\sigma}(k) \ , \ \widetilde{\sigma}(k+1)}$ is minimal. On the other hand, by the induction hypothesis, $\widetilde{\sigma}(k+1)$ is the smallest index not in $\left \lbrace \sigma(1) \ , \ \dots \ , \ \sigma(d+k-1) \right \rbrace$ such that $\tau_{\sigma(d+k-1) \ , \  \widetilde{\sigma}(k+1)}$ is minimal. By the construction of $\sigma,$ we conclude $\widetilde{\sigma}(k+1)=\sigma(d+k).$ Therefore, the statement \eqref{eq: Periodic COOM Cyclic Order Uniqueness} holds for $n=k+1.$ Thus, we have shown our candidate permutation $\sigma$ is unique up to cyclic shifts.

\section{Analysis of Non-Periodic Signals}\label{sec: Analysis of Non-Periodic Signals}

In the previous section, we utilized the cross-correlation function to answer our posed leader follower dynamics questions in \textbf{Section \ref{sec: Leader-Follower Dynamics Questions}} for periodic signals. In this section, we investigate whether cross-correlation can be used to answer our posed leader follower dynamics questions for signals that are not necessarily periodic. The \textit{cross-correlation} of two generic signals $f,g: \R \rightarrow \R$ is defined as the map $f \star g: \R \rightarrow \R,$ where 
	\begin{align}\label{eq: Cross-Correlation Function}
		(f \star g)(\tau) &= \int_{\R} f(t-\tau) g(t) \ dt,
	\end{align}
provided the integral on the right hand side exists.

\subsection{Time Shifted Model}\label{subsec: Time Shifted Model}
Suppose $f: \R \rightarrow \R$ is a  signal that is compactly supported in a closed interval i.e. the closure of the set of all points $t \in \R$ with $f(t) \ne 0$ is a closed interval. Let $g: \R \rightarrow \R$ be the map defined by  
	\begin{align} \label{eq: Time Shifted Model}
		g(t) &= f(t-\tau_0).
	\end{align}
for some time shift $\tau_0 \in \R.$ We refer to the model \eqref{eq: Time Shifted Model} as a \textit{time shifted model}.  Determining the leader follower relationship between $f$ and $g$ under the time shifted model amounts to determining the sign of the time shift $\tau_0.$  If $\tau_0>0,$ then we interpret $f$ to lead $g.$ If $\tau_0<0,$ then we interpret $f$ to follow $g.$ 

Given the signals $f$ and $g,$ we determine the value of the time shift $\tau_0$ under the time shifted model \eqref{eq: Time Shifted Model}. In order to do this, we utilize the cross-correlation of $f$ and $g$ defined in \eqref{eq: Cross-Correlation Function}. For each $\tau \in \R,$ we have  
	\begin{align}
		(f \star g)(\tau) &= \int_{\R} f(t - \tau ) \ g(t) \ dt \nonumber  \\
		&= \int_{\R} \frac{(f(t-\tau))^2 + (g(t))^2 - (f(t-\tau) - g(t))^2}{2} \ dt \nonumber  \\
		&= \int_{\R} \frac{(f(t-\tau))^2 + (f(t-\tau_0))^2}{2} \ dt - \int_{\R} \frac{\left(f(t-\tau) - f(t-\tau_0) \right)^2}{2} \ dt \nonumber \\
		&= \int_{\R} \frac{(f(t-\tau))^2}{2} \ dt + \int_{\R} \frac{(f(t-\tau_0))^2}{2} \ dt - \int_{\R} \frac{\left(f(t-\tau) - f(t-\tau_0) \right)^2}{2} \ dt
		  \label{eq: Time Shifted Model Cross Correlation Function 1} \\
		&=(f \star f)(0) - \int_{\R} \frac{\left(f(t-\tau) - f(t-\tau_0) \right)^2}{2} \ dt \label{eq: Time Shifted Model Cross Correlation Function 2},
	\end{align}
in which the first term of \eqref{eq: Time Shifted Model Cross Correlation Function 2} follows from observing the first two integral terms in \eqref{eq: Time Shifted Model Cross Correlation Function 1} are both equal to $\frac{(f \star f)(0)}{2}$ via the respective changes of variables $t \mapsto t-\tau$ and $t \mapsto t-\tau_0.$

Note that the integral after the minus sign in \eqref{eq: Time Shifted Model Cross Correlation Function 1} is a non-negative quantity because it is the integral of a non-negative function. This implies the cross-correlation function has an upper bound
	\begin{align} \label{eq: Time Shifted Model Cross Correlation Upper Bound}
		(f \star g)(\tau) \le (f \star f)(0).
	\end{align}
for all $\tau \in \R.$ Equality in \eqref{eq: Time Shifted Model Cross Correlation Upper Bound} holds if and only if the integral term in \eqref{eq: Time Shifted Model Cross Correlation Function 2} is $0$ for some $\tau \in \R.$ Because $f$ is continuous and compactly supported in a closed interval, equality in \eqref{eq: Time Shifted Model Cross Correlation Upper Bound} holds if and only if the integrand of the integral term in \eqref{eq: Time Shifted Model Cross Correlation Function 2} is $0$ i.e.
	\begin{align}
		f(t-\tau) &= f(t-\tau_0) \nonumber.
	\end{align} 
for all $t \in \R.$ But since $f$ is not periodic, we have $(f \star g)(\tau)=(f \star f)(0)$ if and only if  $\tau=\tau_0.$ 
Therefore, the unique displacement at which the cross-correlation $f \star g$ attains its maximum is precisely the time shift $\tau_0$ under the model \eqref{eq: Time Shifted Model}. Having recovered the time shift $\tau_0,$ we can determine the leader follower relationship between the signals $f$ and $g$ upon inspecting the sign of $\tau_0.$

\subsection{Time Reparameterization Invariance}\label{subsec: Time Reparameterization Invariance}

Based on our analysis so far, it seems that the cross-correlation function is a useful signal processing tool that enables us to determine the leader follower relationship between any two input signals, provided these signals have nice enough properties. We wonder whether there are drawbacks to using the cross-correlation. 

We focus on one drawback central to the thesis. The cross-correlation function suffers under time reparameterization.  A \textit{time reparameterization} of the time interval $T$ is a strictly increasing homeomorphism between $T$ and itself. This means that if we use the cross-correlation as a tool to analyze signals, then the results given by cross-correlation may not be the same across all time reparameterizations. In other words, the results we obtain via cross-correlation may vary according to how we observe the signals throughout time. 

To illustrate with an example, let $T=[-3,2],$ and consider the signals $f: [-3,2] \rightarrow \R,$ where
	\begin{align}
		f(t) &= \begin{cases} \sin^2(2 \pi t) & t \in [-1,1] \\ 0 & \text{else} \end{cases}
	\end{align}
and $g(t)= f \left(t-\frac{1}{4} \right).$ Consider the map $r: [-3,2] \rightarrow [-3,2]$ defined by 
	\begin{align}
	r(t)=\frac{t^3+6}{7}.
	\end{align} 
Note that $r$ is a valid time reparameterization of $[-3,2].$ Indeed, we have $r(-3)=-3$ and $r(2)=2,$ and since $t \mapsto t^3$ is strictly increasing on $\R$, the map $t \mapsto r(t)$ is also strictly increasing on $\R$ and thus strictly increasing in $[-3,2].$ Therefore, $r([-3,2])=[-3,2].$ Finally, the inverse map $r^{-1}(y)=(7y-6)^{\frac{1}{3}}$ is continuous on all of $\R$ and is thus continuous in $[-3,2].$ Therefore, $r$ is a strictly increasing homeomorphism, which means it is a valid time reparameterization. However, in \textbf{Figure \ref{fig: SampleReparameterizationVariance}}, we plot the correlation functions $f \star g$ and $(f \circ r) \star (g \circ r).$ We see that the displacement maximum for the original cross-correlation $f \star g$ and the displacement maximum for the reparameterized cross-correlation $(f \circ r) \star (g \circ r)$ differ in sign. Hence, the cross-correlation function does not allow us to conclude whether $f$ leads $g$ or vice versa.

\begin{figure}
	\centering
	\includegraphics[scale=0.65]{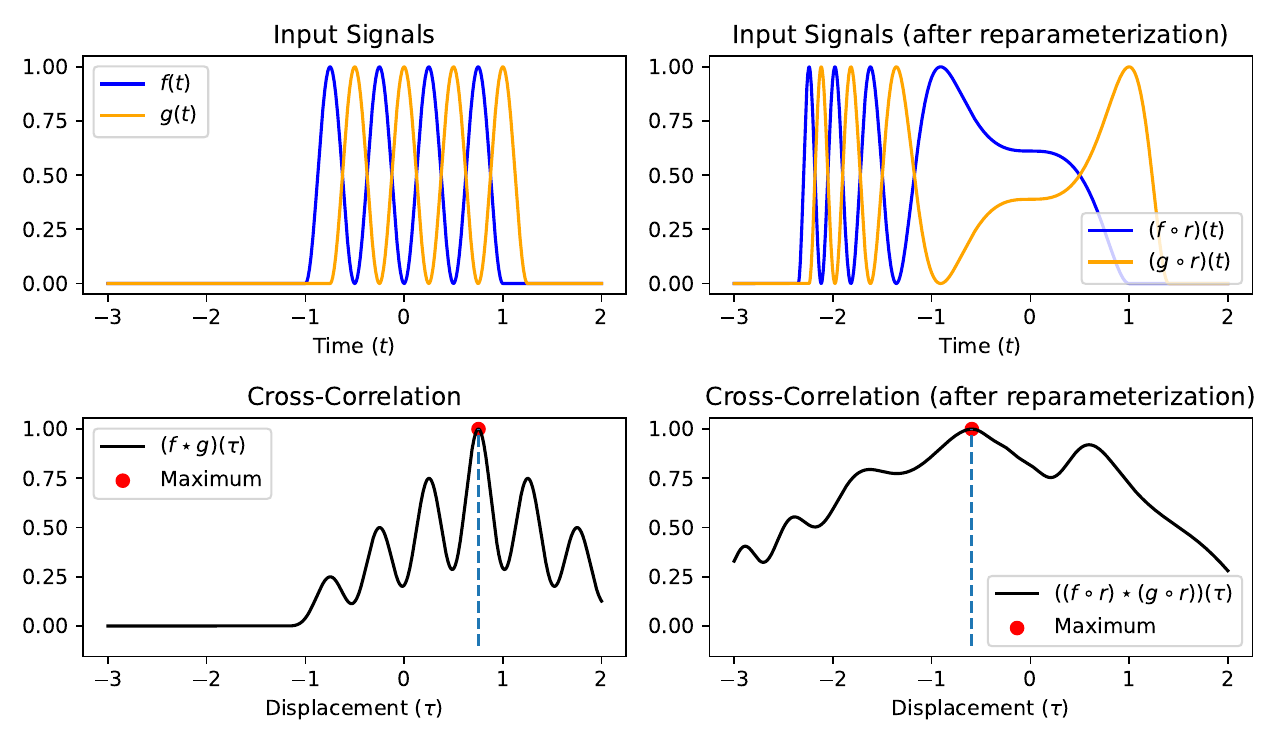}
	\caption{We illustrate how time reparameterization can impact the results of cross-correlation. Consider the signal $f: \R \rightarrow \R$ defined by $f(t)= \sin^2(2 \pi t)$ for $t \in [-1,1]$ and $0$ otherwise. Let $g(t)=f(t-0.25).$ In the first row, we plot the original signals $f$ and $g$ on the left, and we plot the reparameterized signals $f \circ r$ and $g \circ r$ on the right. In the second row, we plot the cross-correlation $f \star g$ on the left and the cross-correlation $(f \circ r) \star (g \circ r)$ on the right. Observe the displacement at which $(f \star g)$ takes on its maximum differs in sign from the displacement at which $(f \circ r) \star (g \circ r)$ takes on its maximum.} \label{fig: SampleReparameterizationVariance}
\end{figure}

Thus, in order to perform any meaningful analysis of signals, we need to consider using tools that are time reparameterization invariant.

\section{Cyclicity Analysis}\label{sec: Cyclicity Analysis}
In this section, we consider cyclic signals. A signal $\mathbf x : \R \rightarrow \R^N$ is \textbf{cyclic} if $\mathbf x \circ r$ is periodic for some time reparameterization $r: \R \rightarrow \R.$ Cyclic signals generalize periodic signals. Roughly speaking, cyclic signals have repeating temporal patterns, but such patterns do not repeat in a predictable fashion, which makes them different from periodic signals. Many real world phenomena are cyclic in nature. For example, the economy undergoes expansion and recession, but not predictably due to various external factors that control the economy. Other examples include cardiac cycles and fMRI signals \cite{AbrahamShahsavaraniZimmermanHusainBaryshnikov2021}.

We illustrate the difference between a periodic and a cyclic signal with a theoretical example. Consider the signals $f(t)= \sin(2 \pi t)$ and $g(t)= f(t^3)$ on $\R.$ In \textbf{Figure \ref{fig: SampleCyclicSignals}}, we plot $f$ and $g$ and the resulting signals we observe after shifting the time axis by $1$ unit. Since $f$ is $1$-periodic, shifting the time axis by $1$ unit results in us observing the exact same signal that we started with. On the other hand, $g$ is not $1$-periodic; shifting the time axis by $1$ unit results in us observing a completely different signal than the one we started with. In fact, $g$ is not $P$-periodic for any constant $P>0.$ However, $g$ is cyclic: if $r: \R \rightarrow \R$ is defined by $r(t)=t^{\frac{1}3},$ then note that $r$ is a valid time reparameterization satisfying $f= g \circ r.$

\begin{figure}
	\centering
	\includegraphics[scale=0.6]{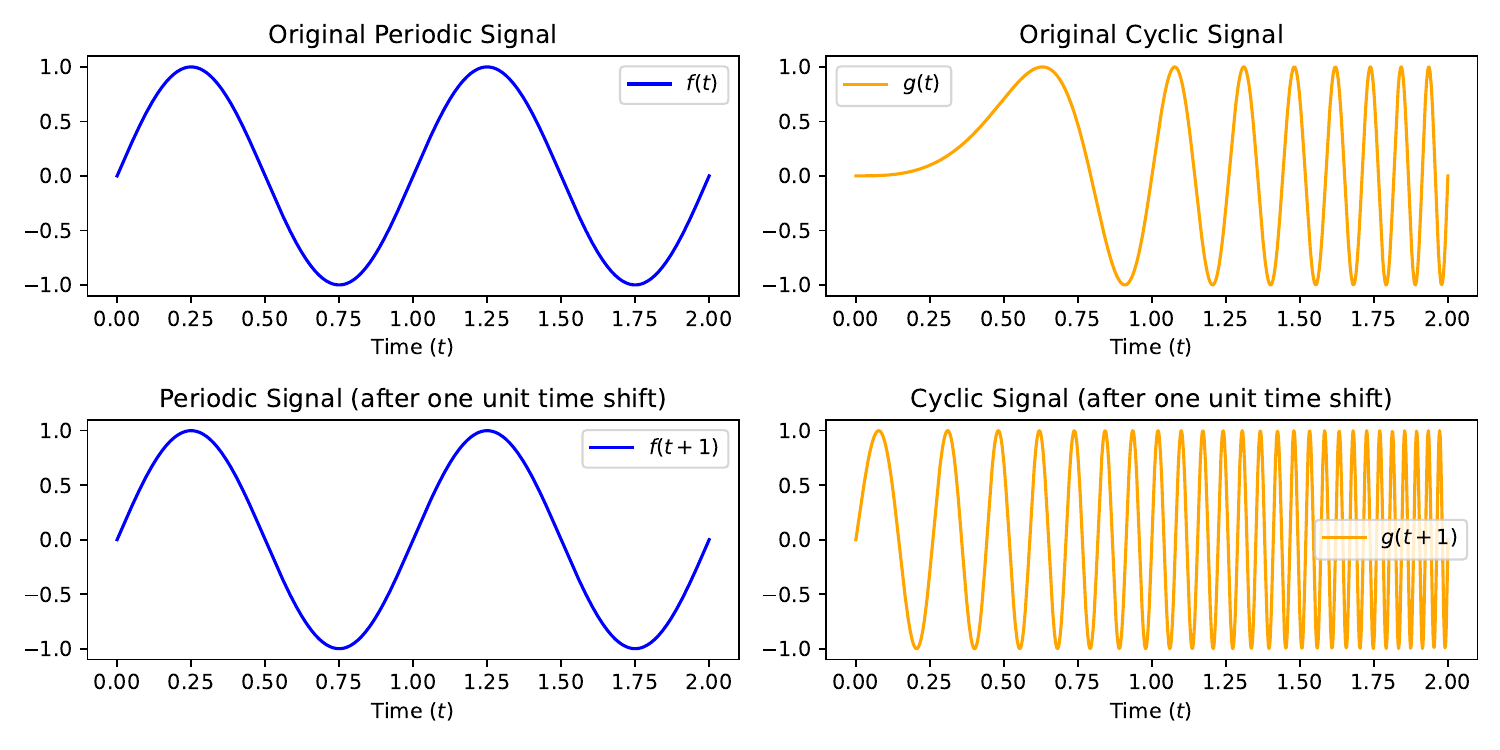}
	\caption{We illustrate the difference between a cyclic and a periodic signal. Let $f(t)=\sin (2 \pi t)$ and $g(t)=f(t^3).$ On the first row, we plot $f(t)$ on the left and $g(t)$ on the right. On the second row, we plot $f(t+1)$ on the left and $g(t+1)$ on the right. Note $f$ is unaltered by the time shift of $1$ unit, but $g$ is altered.} \label{fig: SampleCyclicSignals}
\end{figure}

In this section, we describe the procedure of \textit{Cyclicity Analysis}, first introduced in \cite{BaryshnikovSchlafly2016}, which is the study of the leader-follower dynamics pertaining to cyclic signals.  Cyclicity Analysis directly answers both of our posed leader follower questions in \textbf{Section \ref{sec: Leader-Follower Dynamics Questions}} in a time reparameterization invariant fashion.

\subsection{Iterated Path Integrals}
First, we assume $\mathbf x: T \rightarrow \R^N$ is a generic signal in which $T=[a,b]$ is a closed interval. We recall the first question we posed in \textbf{Section \ref{sec: Leader-Follower Dynamics Questions}}, which is to determine the pairwise leader follower relationship between the $m$-th component and $n$-th component signals of $\mathbf x$ for fixed indices $1 \le m,n \le N,$ assuming such component signals evolve similarly throughout time. 

In order to answer this question, 
Cyclicity Analysis invokes the theory of iterated path integrals \cite{Chen1958,Chen1971,Chen1973,Chen1977}. For each $K \in \N$, a $K$-th order \textit{iterated path integral} of $\mathbf x$ is a functional $\text{IP}_{i_1 \ , \ \dots \ , \ i_K}$ of the form
	\begin{align}\label{eq: Iterated Path Integral}
		\text{IP}_{i_1 \ , \ \dots \ , \ i_K}(\mathbf x) &= \frac{1}{K!} \int_{a \le t_1 \le \ \dots \ \le t_K \le b} x'_{i_1}(t_1) \ \dots \ x'_{i_K}(t_K) \ dt_K \ \dots \ dt_1,
	\end{align}
where $1 \le i_1 \le \ \dots \ \le i_K \le N.$ For example, a first order iterated path integral of $\mathbf x$ is explicitly of the form
	\begin{align}
		\text{IP}_n(\mathbf x) &= \int_a^b x'_{n}(t) \ dt \nonumber   \\
		&= x_{n}(b)- x_{n}(a) \nonumber ,
	\end{align}
which is the \textit{increment} of the trajectory parameterized by $x_n:T \rightarrow \R.$ 

Iterated path integrals date back to the times of Riemann and Picard. They were initially used as tools to construct and analyze solutions to ordinary differential equations. Later on, K.T. Chen \cite{Chen1958,Chen1971,Chen1973,Chen1977} discovered the relevance of iterated path integrals to algebraic topology, while Terry Lyons \cite{Lyons1998,Lyons2007} discovered the relevance to the analysis of \textit{rough paths}, which are Holder trajectories.

Iterated path integrals of the signal $\mathbf x$ obey several invariant properties \cite{BaryshnikovSchlafly2016}, which make them a suitable tool for data analysis. Firstly, they are indeed time reparameterization invariant:  for any time reparameterization $r: T \rightarrow T,$ we have
	\begin{align}
			\text{IP}_{i_1 \ , \ \dots \ , \ i_K}(\mathbf x)  &= 	\text{IP}_{i_1 \ , \ \dots \ , \ i_K}(\mathbf x \circ r) \nonumber,
	\end{align}
which follows directly from performing a change of variables $t \mapsto r(t)$ in \eqref{eq: Iterated Path Integral}. Next, the iterated path integrals of $\mathbf x$ are invariant under parallel translation: if $\mathbf a \in \R^N$ is fixed and if $\mathbf y : T \rightarrow \R^N$ is the signal defined by $\mathbf y(t)=\mathbf x(t) + \mathbf a,$ then 
	\begin{align}
		\text{IP}_{i_1 \ , \ \dots \ , \ i_K}(\mathbf x) &= \text{IP}_{i_1 \ , \ \dots \ , \ i_K}(\mathbf y) \nonumber.
	\end{align}
Moreover, the iterated path integrals are invariant under \textit{detours}. If $C_1 \ , \  \dots \ , \ C_L$ are oriented curves in $\R^N$ such that the ending point of $C_\ell$ is the starting point of $C_{\ell+1},$ their  concatenation possesses a \textit{detour}  if for some $1 \le \ell < L,$ the curves $C_{\ell+1}$ and $C_{\ell}$ are the same but with opposite orientations. In other words, if we traverse $C_1 \ , \ \dots \ , \ C_L$ in that order, then our traversal will involve backtracking. If $\mathbf x$ and $\mathbf y$ parameterize the same curve, in which $\mathbf x$ traverses it with detours but $\mathbf y$ traverses it without detours, then 
	\begin{align}
		\text{IP}_{i_1 \ , \ \dots \ , \ i_K}(\mathbf x) &= \text{IP}_{i_1 \ , \ \dots \ , \ i_K}(\mathbf y) \nonumber.
	\end{align}
Finally, K.T. Chen \cite{Chen1958,Chen1971,Chen1973,Chen1977} proved a stronger statement: two signals $\mathbf x, \mathbf y: T \rightarrow \R^N$ that parameterize their respective curves without detours are the same up to time reparameterization and parallel translation \textit{if and only if} all of their corresponding iterated path integrals agree.

\subsection{Oriented Areas}
Now, assume $\mathbf x: T \rightarrow \R^N$ parameterizes a closed curve i.e. $\mathbf x(a)=\mathbf x(b).$ We consider the second order iterated path integrals of $\mathbf x.$ For each $1 \le m,n \le N,$ we substitute $K=2$ and $i_1=m$ and $i_2=n$ into \eqref{eq: Iterated Path Integral} to obtain the explicit form of the second order iterated path integral:
	\begin{align}
		\text{IP}_{m,n}(\mathbf x) &= \frac{1}2 \int_{a \le s \le t \le b} x_{m}'(s) \ x_n'(t) \ ds \ dt \nonumber \\
		&=  \frac{1}2\int_a^b \int_a^t x_{m}'(s) \ x_n'(t) \ ds \ dt \nonumber \\
		&=  \frac{1}2 \int_a^b (x_m(t) - x_m(a)) \ x_n'(t) \ dt \nonumber \\
		&= \frac{1}2 \int_a^b x_m(t) \ x_n'(t)  \ dt -  \frac{\left(x_n(b) - x_n(a) \right) \ x_m(a)}2 \label{eq: Second Order Iterated Path Integral 1} \\
		&=  \frac{1}2 \int_a^b x_m(t) \ x_n'(t)  \ dt \nonumber,
	\end{align}
in which the second term of \eqref{eq: Second Order Iterated Path Integral 1} vanishes due to the assumption $\mathbf x$ parameterizes a closed curve.

For all $1 \le m,n \le N,$ we define
	\begin{align}
		 A_{m,n}(\mathbf x) &= \text{IP}_{m,n}(\mathbf x) - \text{IP}_{n,m}(\mathbf x)  \nonumber \\
		&= \frac{1}2 \int_a^b x_m(t) \ x_n'(t) - x_n(t) \ x_m'(t) \ dt.   \label{eq: Oriented Area} 
	\end{align}
By Green's Theorem, the integral \eqref{eq: Oriented Area} is a line integral that is equal to the area of the region enclosed by the curve parameterized by the map $(x_m, x_n): T \rightarrow \R^2.$  We note, however, this area is actually an \textit{oriented (signed) area}.  In particular, $A_{m,n}(\mathbf x)$ is positive if $(x_m, x_n): T \rightarrow \R^2$ traverses the aforementioned curve in the counterclockwise direction and is negative otherwise. We heuristically interpret the oriented area as an indicator of the lead-lag relationship between the signals $x_m$ and $x_n$. We say that $x_m$ \textit{leads (precedes)} $x_n$ if $A_{m,n}(\mathbf x)>0.$ We say that $x_m$ \textit{follows (lags)} $x_n$ if $A_{m,n}(\mathbf x)<0.$ Otherwise, we say $x_m$ and $x_n$ are \textit{in sync}.  

We illustrate the oriented area with an example. Reconsider the signals $f,g : [-\infty, \infty] \rightarrow \R$ with $f(t)= e^{-\pi t^2}$ and $g(t)=f(t-1)$ with $f(\pm \infty)=g(\pm \infty)=0.$ The curve $(f,g): [-\infty, \infty] \rightarrow \R^2$ is a closed curve oriented in the counterclockwise direction, which we plot in \textbf{Figure \ref{fig: SampleOrientedArea}}. The oriented area of the region enclosed by this curve is explicitly
	\begin{align}
		\frac{1}{2} \int_\R f(t) \ g'(t) - f'(t)  \ g(t) \ dt &= \frac{1}{2} \int_{\R} - 2 \pi (t-1) e^{- \pi t^2} e^{- \pi (t-1)^2} + 2 \pi t e^{- \pi t^2} e^{- \pi (t-1)^2} \ dt \nonumber  \\
		&= \int_\R \pi e^{-\pi t^2 - \pi (t-1)^2} \ dt = \frac{\pi e^{-\frac{\pi}{2}}}{\sqrt 2} \nonumber,
	\end{align} 
which is positive. Therefore, we interpret $f$ to lead $g$ based on the oriented area.

\begin{figure}[h]
	\centering
	\includegraphics[scale=0.65]{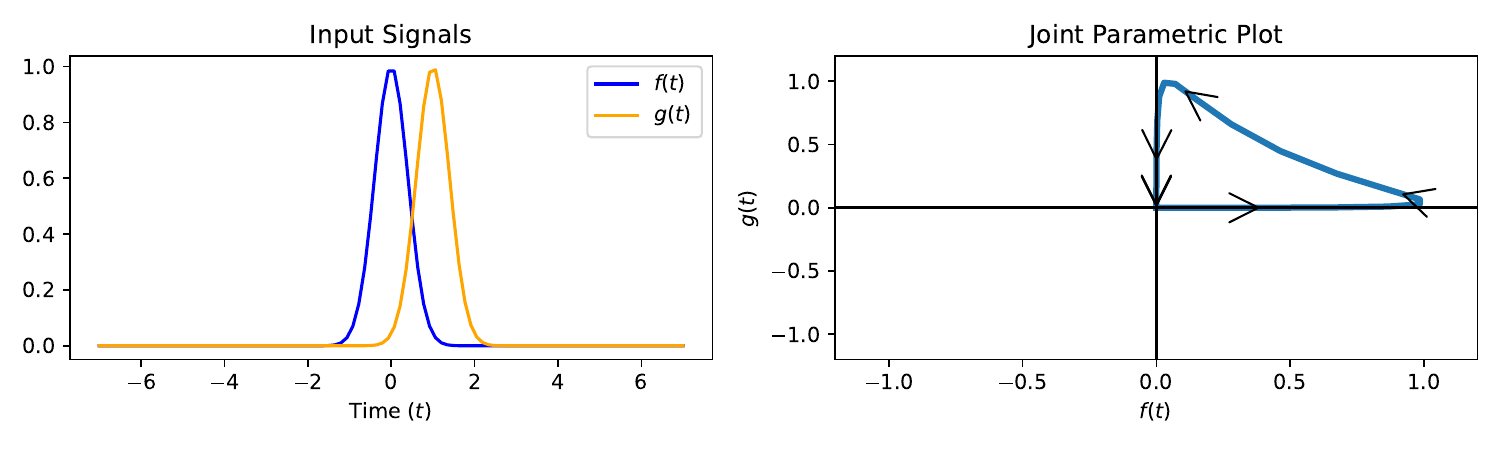}
	\caption{An illustration of the oriented area. Let $f,g: \R \rightarrow \R$ be the signals $f(t)=e^{- \pi t^2}$ and $g(t)=f(t-1).$ We plot $f$ and $g$ on the left and plot the curve parameterized by $(f,g): [-\infty,\infty] \rightarrow \R$ on the right. We indicate the curve's orientation via black arrows. Because the curve is oriented counterclockwise, we interpret $f$ to lead $g.$ } \label{fig: SampleOrientedArea}
\end{figure}

Recall our assumption $\mathbf x$ parameterizes a closed curve. We can modify the definition of the oriented area in the case $\mathbf x$ does not parameterize a closed curve. In this situation, we may consider any map $\widetilde{\mathbf x}: T \rightarrow \R^N$ that parameterizes a closed curve formed by concatenating the curve parameterized by the original signal $\mathbf x$ and any oriented curve in $\R^N$ with starting point $\mathbf x(b)$ and ending point $\mathbf x(a).$ Different choices of  $\widetilde{\mathbf x}$ yield different values of $A_{m,n}(\widetilde{\mathbf x}).$ For our purposes, however, we will choose a signal $\widetilde{\mathbf x}: T \rightarrow \R^N$ such that
	\begin{align} \label{eq: Oriented Area Preservation of Line Integral}
		A_{m,n}(\widetilde{\mathbf x}) &= \frac{1}2 \int_a^b x_m(t) \ x_n'(t) - x_n(t) \ x_m'(t) \ dt.
	\end{align}
i.e. the $(m,n)$-th oriented area of $\widetilde{\mathbf x}$ is the same as the $(m,n)$-th line integral of $\frac{x_m x_n' - x_n x_m'}{2}$ over the curve parameterized by $(x_m, x_n):T \rightarrow \R^2.$

Consider the signal $\widetilde{\mathbf x}:T \rightarrow \R^N$ defined by 
	\begin{align}\label{eq: Linear Interpolation Signal}
		\widetilde{\mathbf x}(t) &= \begin{cases}
			\mathbf x( 2t - a ) & t \in \left [a , \frac{a+b}{2} \right ] \\
			\mathbf x(b) + \left(\mathbf x(a) - \mathbf x(b)\right) \left( \frac{a+b-2t}{a-b} \right) & t \in \left [ \frac{a+b}{2} , b\right ].
		\end{cases}
	\end{align}
This signal parameterizes the curve formed by concatenating the curve parameterized by the original signal $\mathbf x$ and the oriented line segment with starting point $\mathbf x(b)$ and ending point $\mathbf x(a).$ We show this signal satisfies our requirement \eqref{eq: Oriented Area Preservation of Line Integral}. Explicitly computing $A_{m,n}(\widetilde{\mathbf x}),$ we have 
	\begin{align}
		A_{m,n}(\widetilde{\mathbf x}) &= \frac{1}{2} \int_a^b \widetilde{x}_m(t) \  \widetilde{x}_n'(t) - \widetilde{x}_n(t) \  \widetilde{x}_m'(t) \  dt \nonumber \\
		&= \frac{1}{2} \int_a^{\frac{a+b}{2}} \widetilde{x}_m(t) \  \widetilde{x}_n'(t) - \widetilde{x}_n(t) \  \widetilde{x}_m'(t)  \ dt  \nonumber \\
  &\qquad \qquad +  \frac{1}{2} \int_{\frac{a+b}{2}}^b \widetilde{x}_m(t) \  \widetilde{x}_n'(t) - \widetilde{x}_n(t) \  \widetilde{x}_m'(t) \ dt \nonumber \\
		&= \int_a^{\frac{a+b}{2}} x_m(2t-a) \  x_n'(2t-a) - x_n(2t-a) \  x_m'(2t-a) \ dt \nonumber \\
		&\qquad - \frac{x_n(a) - x_n(b)}{b-a} \int_{\frac{a+b}{2}}^b x_m(b) + \left(x_m(a) - x_m(b) \right) \left(\frac{a+b-2t}{a-b} \right) \ dt \nonumber \\
		&\qquad \qquad + \frac{x_m(a) - x_m(b)}{b-a} \int_{\frac{a+b}{2}}^b x_n(b) + \left(x_n(a) - x_n(b) \right) \left(\frac{a+b-2t}{a-b} \right) \ dt \nonumber \\
		&= \int_a^{\frac{a+b}{2}} x_m(2t-a) \  x_n'(2t-a) - x_n(2t-a) \  x_m'(2t-a) \ dt  \nonumber  \\
		&\qquad - \frac{x_m(b) (x_n(a)-x_n(b))}{2} - \frac{(x_m(a) - x_m(b)) (x_n(a)-x_n(b))}{4}  \nonumber \\
		&\qquad \quad +  \frac{x_n(b) (x_m(a)-x_m(b))}{2} + \frac{(x_n(a) - x_n(b)) (x_m(a)-x_m(b))}{4} \nonumber \\
		&= \int_a^{\frac{a+b}{2}} x_m(2t-a) \  x_n'(2t-a) - x_n(2t-a) \  x_m'(2t-a) \  dt \label{eq: Linear Interpolation Oriented Area 1} \\
		&= \frac{1}{2} \int_a^{b} x_m(t) \  x_n'(t) - x_n(t) \  x_m'(t) \ dt \label{eq: Linear Interpolation Oriented Area 2},
	\end{align}
in which \eqref{eq: Linear Interpolation Oriented Area 2} follows from performing the change of variables $t \mapsto 2t-a$ in \eqref{eq: Linear Interpolation Oriented Area 1}. Thus, our choice of $\widetilde{\mathbf x}$ satisfies our specific requirement in \eqref{eq: Oriented Area Preservation of Line Integral}.

So from henceforth, if the signal $\mathbf x$ does not parameterize a closed curve, then we define $A_{m,n}(\mathbf x)=A_{m,n}(\widetilde{\mathbf x}),$ where $\widetilde{\mathbf x}$ is the signal defined in \eqref{eq: Linear Interpolation Signal}.  

\subsection{The Lead Matrix}

If $\mathbf x: T \rightarrow \R^N$ is a signal with $T=[a,b],$ then we will write $A_{m,n}$ as the oriented area in place of the original notation $A_{m,n}(\mathbf x).$  In isolation, the oriented area is simply a number that heuristically indicates the pairwise leader follower relationship between any two specific components of $\mathbf x.$ For future purposes, however, we would like to consider all of these $N^2$ total oriented areas together. So we will arrange all the oriented areas into an $N \times N$ matrix. Let $\mathbf A$ be the $N \times N$ matrix whose $(m,n)$-th entry is the oriented area $A_{m,n}$ for each $1 \le m,n \le N.$ We refer to $\mathbf A$ as the \textit{lead(-lag) matrix} corresponding to $\mathbf x$.  More succinctly, we can rewrite $\mathbf A$ as an integral of a matrix-valued function:
	\begin{align}\label{eq: Lead Matrix}
		\mathbf A &= \frac{1}{2} \int_a^b \mathbf x(t) \  \dot{\mathbf x}^{\text T}(t) - \dot{\mathbf x}(t) \  \mathbf x^{\text T}(t) \ dt,
	\end{align}
where $\dot{\mathbf x}(t)=\left(x_1'(t) \ , \ \dots \ , \ x_N'(t) \right )$ is the derivative of $\mathbf x$ with respect to $t.$

We state several properties pertaining to the lead matrix. Firstly, the matrix $\mathbf A$ is \textit{skew-symmetric} i.e. $\mathbf A^\text T=-\mathbf A.$ Equivalently, $A_{m,n}=-A_{n,m}$ for all $1 \le m,n \le N,$ which follows immediately upon swapping the indices $m$ and $n$ in the equation \eqref{eq: Oriented Area}. Next, each eigenvalue of $\mathbf A$ is either $0$ or purely imaginary  \cite[Exercise 7B.4]{Axler2024}. These purely imaginary eigenvalues of $\mathbf A$ come in complex conjugate pairs: if $i \lambda$ is an eigenvalue of $\mathbf A$ for some nonzero $\lambda \in \R,$ then $-i \lambda$ is also an eigenvalue of $\mathbf A.$ Lastly, $\mathbf A$ is diagonalizable \cite[Exercise 7B.4]{Axler2024}. In particular, there is an orthonormal basis of $\C^N$ consisting of the eigenvectors of $\mathbf A.$

For each $1 \le n \le N,$ we let $\lambda_n(\mathbf A)$ be the $n$-th largest eigenvalue of $\mathbf A$ in modulus and consider the corresponding orthonormal basis of eigenvectors $\left \lbrace \mathbf v_n(\mathbf A) \right \rbrace_{n=1}^N,$ in which $\mathbf v_n(\mathbf A)$ is the associated unit eigenvector with $\lambda_n(\mathbf A)$ for each $1 \le n \le N.$  We refer to the eigenvector $\mathbf v_1(\mathbf A)$ as the \textit{leading (dominant) eigenvector} of $\mathbf A.$ If $\lambda_1(\mathbf A)$  is nonzero, then our notation for the largest eigenvalue and leading eigenvector is actually ambiguous because $\lambda_2(\mathbf A)$ would be the complex conjugate of $\lambda_1(\mathbf A)$ and thus satisfy $|\lambda_1(\mathbf A)|= |\lambda_2(\mathbf A)|.$ For our purposes, however, this ambiguity does not matter. By \cite[Exercise 7B.4]{Axler2024}, we have $\mathbf v_2(\mathbf A)=\overline{\mathbf v_1(\mathbf A)},$ in which $\overline{\mathbf v_1(\mathbf A)}$ is the vector whose $n$-th component is equal to the complex conjugate of the $n$-th component of $\mathbf v_1(\mathbf A).$ Therefore, regardless of whether $\lambda_1(\mathbf A)$ is assigned to have a positive or negative imaginary part, the leading eigenvector $\mathbf v_1(\mathbf A)$ is unique up to the complex conjugation of its components.

Finally, we mention how to construct a lead matrix so that it is suitable for practical applications. In the real world, one does not directly observe a continuous signal. Rather, one observes a signal at finitely many different times. Given a signal $\mathbf x: T \rightarrow \R^N$ with $T=[a,b],$ a \textit{time series} is a finite collection of the form $\left \lbrace \mathbf x_{t_k} \right \rbrace_{k=1}^K,$ where $K \in \N$ is fixed, $\left \lbrace t_k \right \rbrace_{k=1}^K$ is an increasing sequence in $T$ with $t_1=a$ and $t_K=b,$ and $\mathbf x_{t_k}=\mathbf x(t_k).$   Typically, $K$ is a very large integer. We define the lead matrix for the time series $\left \lbrace \mathbf x_{t_k} \right \rbrace_{k=1}^K$ to be the lead matrix $\mathbf A$ for the signal $\widetilde{\mathbf x}: T \rightarrow \R^N$ parameterizing the oriented curve formed by concatenating the oriented line segments $L_1 \ , \ \dots \ , \ L_K,$ in which $L_k$ has starting point $\mathbf x_{t_k}$ and ending point $\mathbf x_{t_{k+1}}$ for each $1 \le k <K$ and $L_K$ has starting point $\mathbf x_{t_K}$ and ending point $\mathbf x_{t_1}.$  Explicitly, 
	\begin{align} \label{eq: Discrete Lead Matrix}
		\mathbf A &= \frac{1}2 \sum_{k=1}^{K-1} \mathbf x_{t_{k}} \mathbf x_{t_{k+1}}^\text T - \mathbf x_{t_{k+1}} \mathbf x_{t_{k}}^ \text T.
	\end{align}
If we write out $\mathbf x_{t_k}=(x_{t_k, 1} \ , \ \dots \ ,\ x_{t_k, N}),$ then the $(m,n)$-th entry of $\mathbf A$ is the oriented area of the convex polygon with vertices $\left(x_{t_1,m} \ , \ x_{t_1,n} \right) \ , \   \dots \ ,\ \left(x_{t_K,m} \ , \ x_{t_K,n} \right)$ computed via the \textit{Shoelace Formula} \cite{AllgowerSchmidt1986}.

\subsection{Chain of Offsets Model}

Now, we assume that $\mathbf x: \R \rightarrow \R^N$ is a cyclic signal. We recall our second leader-follower dynamics question that we posed in \textbf{Section \ref{sec: Leader-Follower Dynamics Questions}}, which is to determine the order in which the $N$ component signals of $\mathbf x$ evolve throughout time assuming all these signals trace some underlying signal up to scaling constants and time shifts. Formally, we assume that the cyclic signal $\mathbf x$ satisfies the following  \textit{Chain of Offsets Model (COOM)} \cite{BaryshnikovSchlafly2016}, which stipulates the existence of an underlying real-valued $P$-periodic signal $\phi,$ scaling constants $c_1 \ , \ \dots \ , \ c_N >0,$ offsets $\alpha_1 \ , \ \dots \ ,\ \alpha_N \in \R/(P \Z),$ and a time reparameterization $r: \R\rightarrow \R$ such that 
	\begin{align}\label{eq: Cyclic COOM}
		x_n(t) &= c_n \phi(r(t) - \alpha_n).
	\end{align}
for each $1 \le n \le N$ and each $t \in \R.$

Given the cyclic signal $\mathbf x: \R \rightarrow \R^N,$ we determine the cyclic order of the offsets under COOM. But this time, we will utilize the lead matrix to do so. Without loss of generality, we may assume that $r$ is the identity map, which reduces the COOM in \eqref{eq: Cyclic COOM} to the COOM defined earlier in \eqref{eq: Periodic COOM}. If $r$ is not the identity map to begin with, then we can the consider the COOM in  \eqref{eq: Periodic COOM} for the $N$-dimensional signal $\widetilde{\mathbf x} = \mathbf x \circ r^{-1}.$ 

As before, we begin by writing the underlying signal $\phi$ as the Fourier series in \eqref{eq: Periodic COOM Primary Function Fourier Series}. Let $\mathbf A$ be the lead matrix for $\mathbf x$ over one period under COOM, that is, the lead matrix for the restricted signal $\mathbf x|_{[0,P]}.$ Then, for each $1 \le m,n \le N,$ the $(m,n)$-th entry of $\mathbf A$ is explicitly 
	\begin{align}
		A_{m,n} &= \frac{1}{2} \int_0^P x_m(t) \ x_n'(t) - x_n(t) \ x_m'(t)  \ dt \nonumber  \\
		&= \frac{1}{2} \int_0^P \sum_{k,\ell \in \Z} c_n c_m \left(\frac{2 \pi i \ell}{P} \right) \hat{\phi}_k  \ \hat{\phi}_\ell \  e^{\frac{2 \pi i k (t-\alpha_m)}{P}} \ e^{\frac{2 \pi i \ell (t-\alpha_n)}{P}}  \ dt \nonumber  \\
		&\qquad - \frac{1}2 \int_0^P \sum_{k,\ell \in \Z} c_n c_m \left(\frac{2 \pi i k}{P} \right) \hat{\phi}_k  \ \hat{\phi}_\ell \  e^{\frac{2 \pi i k (t-\alpha_m)}{P}} \ e^{\frac{2 \pi i \ell (t-\alpha_n)}{P}} \ dt \nonumber  \\
		&= \pi i \ c_n c_m  \  \sum_{k,\ell \in \Z} \left(\ell - k \right) \ \delta_{k,-\ell} \ \hat{\phi}_k  \ \hat{\phi}_\ell \  e^{-\frac{2 \pi i k \alpha_m + 2 \pi i \ell \alpha_n}{ P}} \nonumber \\
		&= -2 \pi i \ c_n c_m  \  \sum_{k \in \Z} k  \left| \hat{\phi}_k \right|^2 e^{\frac{2 \pi i k \left(\alpha_n - \alpha_m \right)}{ P}} \nonumber  \\
		&= 2 \pi c_m c_n \sum_{k \in \N} k  \left| \hat{\phi}_k \right|^2 \sin \left( \frac{2 \pi k(\alpha_m - \alpha_n)}P\right) \label{eq: Cyclic COOM Lead Matrix}.
	\end{align}

\subsection{One Harmonic}

Assume that the underlying $P$-periodic signal $\phi$ in \eqref{eq: Cyclic COOM} has only one harmonic. This means the Fourier Coefficient $\hat{\phi}_k$ is nonzero for only one index $k \in \N.$  We show that we can recover the cyclic order of the offsets via the leading eigenvector of the lead matrix $\mathbf A$ defined in \eqref{eq: Cyclic COOM Lead Matrix}. Note for each $1 \le m,n \le N,$ the $(m,n)$-th entry of the lead matrix $\mathbf A$ is explicitly
	\begin{align}
		A_{m,n} &=2 \pi c_m c_n  k  \left| \hat{\phi}_k \right|^2 \sin \left( \frac{2 \pi k(\alpha_m - \alpha_n)}P\right) \nonumber \\
		&= 2 \pi c_m c_n  k  \left| \hat{\phi}_k \right|^2 \left( \sin \left( \frac{2 \pi k \alpha_m}{P} \right) \ \cos \left( \frac{2 \pi k \alpha_n}{P} \right) - \cos \left( \frac{2 \pi k \alpha_m}{P} \right) \ \sin \left( \frac{2 \pi k \alpha_n}{P} \right) \right)  \label{eq: Rank 2 COOM Lead Matrix Entry}.
	\end{align}

We show that if the lead matrix $\mathbf A$ is nonzero, then it must have rank two. Recall if $\mathbf a, \mathbf b \in \C^N,$ then the \textit{outer product} of $\mathbf a$ and $\mathbf b$ is the complex $N \times N$ matrix $\mathbf a \otimes \mathbf b = \mathbf a \mathbf b^*,$ where $\mathbf b^*=\overline{\mathbf b^\text T}$ is the conjugate transpose of $\mathbf b.$ A statement made in \cite{BaryshnikovSchlafly2016} is that a complex skew-symmetric matrix has rank two if and only if it is of the form $\mathbf a \otimes \mathbf b - \mathbf b \otimes \mathbf a$ for some linearly independent $\mathbf a, \mathbf b \in \C^N.$ So we verify that our lead matrix $\mathbf A$ indeed satisfies this condition stated in \cite{BaryshnikovSchlafly2016}. Define the vectors
	\begin{align}
		\mathbf a &=  \sqrt{2 \pi k } |\hat{\phi}_k| \left(c_1 \cos \left(\frac{2 \pi k \alpha_1}{P} \right)  \ , \ \dots \ , \ c_N \cos \left(\frac{2 \pi k \alpha_N}{P} \right) \right) \nonumber \\
		\mathbf b &=  \sqrt{2 \pi k } |\hat{\phi}_k| \left(c_1 \sin \left(\frac{2 \pi k \alpha_1}{P} \right)  \ , \ \dots \ , \ c_N \sin \left(\frac{2 \pi k \alpha_N}{P} \right) \right) \nonumber.
	\end{align}
Then, by \eqref{eq: Rank 2 COOM Lead Matrix Entry} immediately implies
	\begin{align}
		\mathbf A &= \mathbf a \otimes \mathbf b - \mathbf b \otimes \mathbf a.
	\end{align}
Now, we need to show that $\mathbf a$ and $\mathbf b$ are linearly independent in $\C^N.$ 
Consider the $N \times 2$ matrix $\mathbf X$ whose first row is $\mathbf a$ and second row is $\mathbf b.$ The condition that $\mathbf a$ and $\mathbf b$ are linearly independent is equivalent to the condition $\mathbf X$ has rank $2.$ The latter condition is equivalent to the condition there exists a $2 \times 2$ submatrix of $\mathbf X$ with a nonzero determinant. Recall our assumption $\mathbf A$ is nonzero. This means $A_{m,n}$ has a nonzero entry for some two indices $1 \le m < n \le N.$ On the other hand,  consider the submatrix $\widetilde{\mathbf X}$ formed by extracting the $m$-th and $n$-th rows of $\mathbf X$  for the same choices of $m$ and $n,$ namely the $2 \times 2$ matrix
	\begin{align}
		\widetilde{\mathbf X}_{m,n} &= \begin{bmatrix} \sqrt{2 \pi k} | \widehat{\phi}_k| c_m \cos \left( \frac{2 \pi k \alpha_m}{P} \right) & \sqrt{2 \pi k} | \widehat{\phi}_k| c_m \sin \left( \frac{2 \pi k \alpha_m}{P} \right) \\
			 \sqrt{2 \pi k} | \widehat{\phi}_k| c_n \cos \left( \frac{2 \pi k \alpha_n}{P} \right) & \sqrt{2 \pi k} | \widehat{\phi}_k| c_n \sin \left( \frac{2 \pi k \alpha_n}{P} \right)
		\end{bmatrix}. \nonumber
	\end{align}
By \eqref{eq: Rank 2 COOM Lead Matrix Entry}, the determinant of $\widetilde{\mathbf X}_{m,n}$ is precisely $A_{m,n},$ which shows $\mathbf A$ has rank two.

We now determine the leading eigenvector of $\mathbf A.$ Since $\mathbf A$ has rank two, it has exactly two nonzero eigenvalues coming in a purely imaginary pair, namely $\lambda_1(\mathbf A)$ and $\lambda_2(\mathbf A),$ in which we recall $\lambda_n(\mathbf A)$ is the $n$-th largest eigenvalue of $\mathbf A$ in modulus. By \cite{BaryshnikovSchlafly2016}, 
	\begin{align}
		\lambda_1(\mathbf A) &=i \sin(\theta) \| \mathbf a\| \|\mathbf b\| \nonumber
	\end{align}
is the largest eigenvalue of $\mathbf A$ with associated eigenvector
	\begin{align}
		\mathbf v_1(\mathbf A) &= -e^{-i \theta} \frac{\mathbf a}{\| \mathbf a\|} + \frac{\mathbf b}{\| \mathbf b\|} \nonumber,
	\end{align}
where $\theta$ is the angle between the vectors $\mathbf a$ and $\mathbf b,$ and $\|\mathbf a\|$ is the Euclidean $2$-norm of $\mathbf a.$  In particular, if $v_{1,n}(\mathbf A)$ is the $n$-th component of $\mathbf v_1(\mathbf A),$ then upon extracting real and imaginary parts of $v_{1,n}(\mathbf A),$ we obtain
	\begin{align}
		\Re(v_{1,n}(\mathbf A)) &=- \frac{a_n \cos(\theta)}{\| \mathbf a\|} + \frac{b_n}{\|\mathbf b\|} \label{eq: Rank 2 COOM Leading Eigenvector Real Part} \\ 
		\Im(v_{1,n}(\mathbf A)) &= \frac{a_n \sin(\theta)}{\| \mathbf a\|} \label{eq: Rank 2 COOM Leading Eigenvector Imaginary Part},
	\end{align}
where 
	\begin{align}
		a_n &= \sqrt{2 \pi k} |\hat \phi_k| c_n \cos \left(\frac{2 \pi k \alpha_n}{P} \right) \nonumber \\
		b_n &=\sqrt{2 \pi k} |\hat \phi_k| c_n \sin \left(\frac{2 \pi k \alpha_n}{P} \right) \nonumber
	\end{align}
are the $n$-th components of $\mathbf a$ and $\mathbf b,$ respectively. 

Through the equations \eqref{eq: Rank 2 COOM Leading Eigenvector Real Part} and \eqref{eq: Rank 2 COOM Leading Eigenvector Imaginary Part}, we see $\left( \Re(v_{1,n}(\mathbf A))  \ , \   \Im(v_{1,n}(\mathbf A)) \right)$ is the result of a certain linear transformation evaluated at the point $\left(a_n, b_n\right).$ Since the points $(a_1, b_1) \ , \ \dots \ , \ (a_N,b_N)$ lie on a circle, the components of the leading eigenvector $\mathbf v_1(\mathbf A)$ lie on an ellipse. The cyclic order in which the points $\left \lbrace (a_n, b_n) \right \rbrace_{n=1}^N$ are traversed on the circle is the cyclic order in which the points $\left \lbrace \left( \Re(v_{1,n}(\mathbf A))  \ , \   \Im(v_{1,n}(\mathbf A)) \right) \right \rbrace_{n=1}^N$ are traversed on the ellipse.  Thus, the cyclic order of $\left \lbrace \text{Arg} \left(v_{1,n}(\mathbf A) \right) \right \rbrace_{n=1}^N,$ where $\text{Arg}(z)$ is the principal argument of the complex number $z,$ matches the cyclic order of the offsets in COOM. We refer to the principal arguments of the components of the leading eigenvector as \textit{phases}.

We show an example of how the cyclic order of the leading eigenvector component phases and the cyclic order of the offsets match when the underlying signal $\phi$ has only one harmonic. Reconsider the sinusoidal signal \eqref{eq: Sinusoidal Signal} for each $1 \le n \le N.$ On one hand, we already established the identity permutation in $S_N$ represents the cyclic order of the offsets. On the other hand, in \textbf{Figure \ref{fig: SampleLeadMatrixCOOM}}, for $N=10,$ we plot the heat map of the lead matrix $\mathbf A$ over one period, as well as the complex components, component moduli, and phases of the leading eigenvector $\mathbf v_1(\mathbf A).$ Since the components of $\mathbf v_1(\mathbf A)$ have equal moduli, they lie on a  circle. Observe the permutation $\sigma \in S_{10}$ that sorts the leading eigenvector component phases in ascending order is given by $\sigma(n)=n+1,$ in which $n+1$ is indexed mod $10.$ This is the identity permutation up to a cyclic shift.

\begin{figure}[h!]
	\centering
	\includegraphics[scale=0.6]{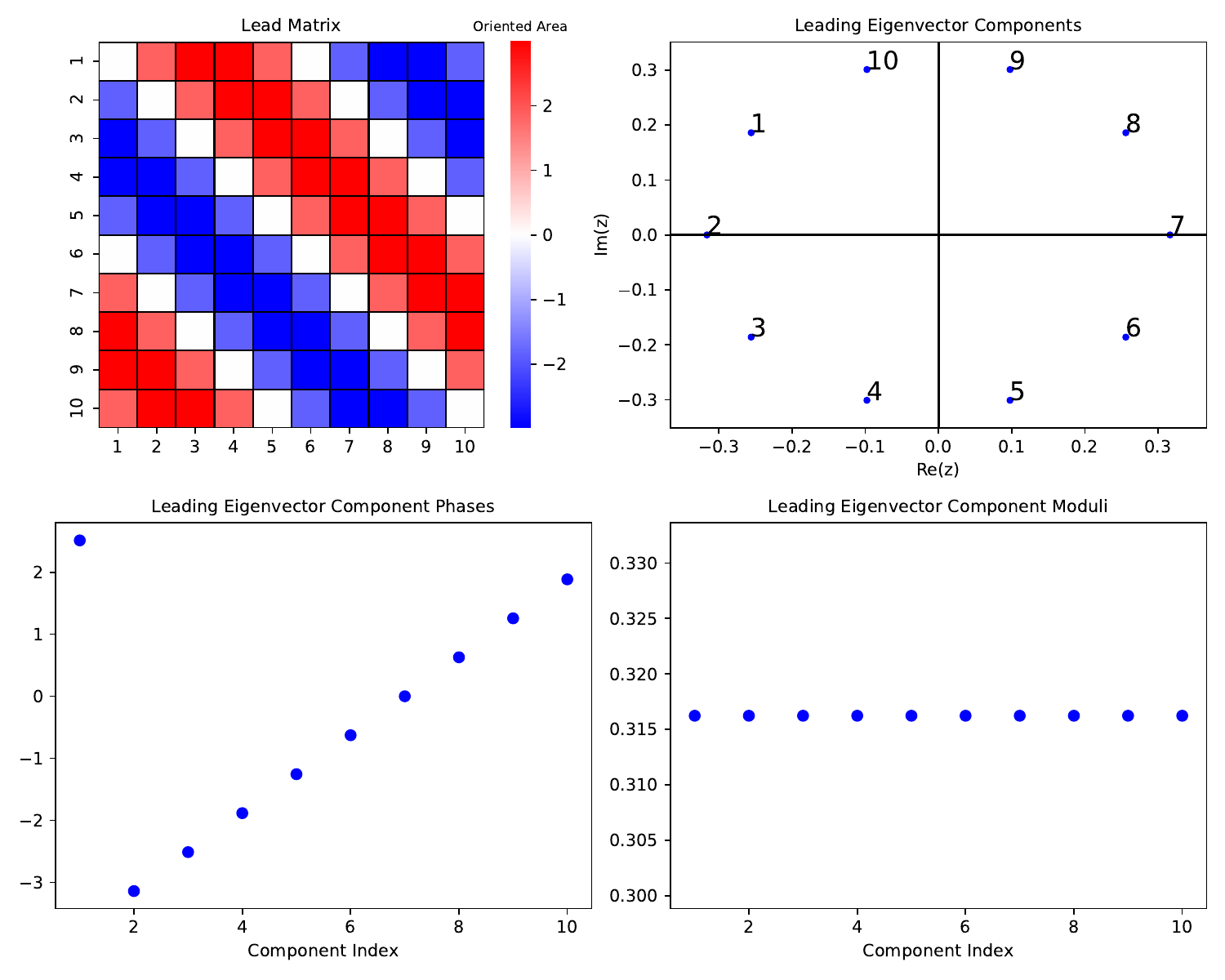}
	\caption{We illustrate the structure of the leading eigenvector of the lead matrix under the rank $2$ COOM. Reconsider the signal $x_n(t)= \sin(2 \pi (t- \frac{n-1}{10}))$ for each $1 \le n \le 10.$  In the first row, we plot the heatmap of the $10 \times 10$ lead matrix corresponding to these signals on the left and plot the $10$ complex components of the leading eigenvector as coordinates in $\R^2$ on the right. We annotate the $n$-th component with the index $n.$ In the second row, we plot the $10$ leading eigenvector component phases on the left, and we plot the $10$ leading eigenvector component moduli on the right. } \label{fig: SampleLeadMatrixCOOM}
\end{figure}

\subsection{Multiple Harmonics}

Suppose the underlying $P$-periodic signal $\phi$ in \eqref{eq: Cyclic COOM} has multiple harmonics, meaning $\phi$ has multiple nonzero Fourier coefficients. Assume, however, one of these Fourier coefficients \textit{dominates}: there exists exactly one $k \in \N$ such that 
	\begin{align}
		|\hat{\phi}_k|^2 \gg \sum_{\ell \ne k} 	|\hat{\phi}_\ell|^2 \nonumber.
	\end{align}
With these assumptions, it is still possible to approximately recover the cyclic order of the offsets $\alpha_1 \ , \ \dots \ , \ \alpha_N$ under COOM via the leading eigenvector of the lead matrix over one period.

Let $\widetilde{\mathbf A}_k$ be the rank two skew-symmetric $N \times N$ matrix, whose $(m,n)$-th entry is equal to \eqref{eq: Rank 2 COOM Lead Matrix Entry} for all $1 \le m,n \le N$.  Then, the $(m,n)$-th entry of $\widetilde{\mathbf A}_k$ will dominate the series expansion in \eqref{eq: Cyclic COOM Lead Matrix}. This means we well approximate the original lead matrix $\mathbf A$ with the lower rank matrix $\widetilde{\mathbf A}_k.$

On the other hand, a rank two approximation of $\mathbf A$ is the matrix $\mathbf P \mathbf A \mathbf P,$ in which $\mathbf P$ is the $N \times N$ matrix corresponding to the projection operator onto the subspace of $\C^N$ spanned by the vectors $\mathbf v_1(\mathbf A)$ and $\mathbf v_2(\mathbf A).$ The smaller the Frobenius norm $\|\mathbf P \mathbf A \mathbf P - \mathbf A\|_F$ is, the better this low rank approximation is of $\mathbf A$.  Moreover, since $\mathbf P$ sends $\mathbf v_1(\mathbf A)$ to itself, we have 
	\begin{align}
		(\mathbf P \mathbf A \mathbf P) \  \mathbf v_1(\mathbf A) &= (\mathbf P \mathbf A) \  \mathbf v_1(\mathbf A) \nonumber \\
		&= \lambda_1(\mathbf A)  \ \mathbf P \  \mathbf v_1(\mathbf A) \nonumber \\
		&= \lambda_1(\mathbf A)  \ \mathbf v_1(\mathbf A) \nonumber,
	\end{align}
which means the leading eigenvector of $\mathbf P \mathbf A \mathbf P$ is $\mathbf v_1(\mathbf A).$ Therefore, assuming the low rank matrix $\mathbf P \mathbf A \mathbf P$ approximates $\mathbf A$ well enough, we can still approximately recover the cyclic order of the offsets $\alpha_1 \ , \ \dots \ , \ \alpha_N$ in COOM via the cyclic order of the component phases of $\mathbf v_1(\mathbf A).$

For practical purposes, one heuristic indicator \cite{BaryshnikovSchlafly2016} that measures the dominance of the low rank $2$ approximation of the lead matrix $\mathbf A$ is the magnitude of $\left |\frac{\lambda_1(\mathbf A)}{\lambda_3(\mathbf A)} \right |,$ the ratio of the first to third largest eigenvalue of $\mathbf A.$ This is because in the rank $2$ situation, this ratio would be infinite.

\section{Ornstein-Uhlenbeck Process} \label{sec: Ornstein-Uhlenbeck Process}
In this section, we fix a probability space $(\Omega, \mathcal A, \P)$ and $N \in \N.$  We consider an \textit{($N$-dimensional) Ornstein Uhlenbeck (OU) process} on this probability space, which is an $N$-dimensional stochastic process $\left \lbrace \mathbf x(t) \right \rbrace_{t \ge 0}$ satisfying the linear stochastic differential equation \cite{OrnsteinUhlenbeck1930}:
	\begin{align}
		d \mathbf x(t) &= - \mathbf B \ \mathbf x(t)  \ dt + \boldsymbol \Sigma \ d \mathbf w(t).
	\end{align}
Here, $\mathbf B : \R^N \rightarrow \R^N$ is a fixed linear operator such that $-\mathbf B$ is \textit{Hurwitz (stable)} i.e. all eigenvalues of $-\mathbf B$ have negative real parts, $\boldsymbol \Sigma: \R^M \rightarrow \R^N$ is a fixed linear operator for some $M \in \N,$ and $\left \lbrace \mathbf w(t) \right \rbrace_{t \ge 0}$ is the standard $M$-dimensional Wiener process independent of the initial random vector $\mathbf x(0)$. We refer to  $\mathbf B$ as the \textit{friction (drift) operator} and $\boldsymbol \Sigma$ as the \textit{volatility operator}, and we collectively refer to $\mathbf B$ and $\boldsymbol \Sigma$ as the \textit{OU model parameters}. We refer to the matrix         
        \begin{align}
		\mathbf D &= \frac{\boldsymbol \Sigma \boldsymbol \Sigma^\text T}{2}
	\end{align} 
as the \textit{diffusion operator}.

The OU process is an ergodic Gaussian diffusion process \cite{GobetShe2016} that prominently appears in Financial Mathematics and Biology \cite{Fasen2013, Meucci2009,BartoszekFuentes-GonzalesMitovPienaarPiwczynskiPuchalkaSalikVoje2022}.  In engineering disciplines, the OU Process is heuristically characterized as the solution to the \textit{Langevin equation} \cite{GodrecheLuck2018}:
\begin{align} \label{eq: OU Langevin Equation}
	\dot{\mathbf x}(t) &= - \mathbf B \ \mathbf x(t) + \boldsymbol \eta(t),
\end{align}
in which $\left \lbrace \boldsymbol \eta(t) \right \rbrace_{t \ge 0}$ is the $N$-dimensional Gaussian ``white noise" process whose autocovariance function satisfies
\begin{align*}
	\boldsymbol \Gamma_{\boldsymbol \eta} (s,t) &= 2 \mathbf D \ \delta(s-t),
\end{align*}
for all $s,t \ge 0.$ Here, $\delta(\cdot)$ is the Dirac measure. The OU process has a unique stationary distribution \cite{GobetShe2016,CourgaeuVeraart2022}, namely the Gaussian distribution with mean $\mathbf 0$ and covariance matrix $\mathbf S,$ in which $\mathbf S$ satisfies the Lyapunov equation \cite{Poznyak2008}:
    \begin{align}
        \mathbf B \mathbf S + \mathbf S \mathbf B^\text T &= 2 \mathbf D.
    \end{align}
For this reason, we refer to $\mathbf S$ as the \textit{stationary covariance operator}. Any realization of the OU process will fluctuate around the origin $\mathbf 0$ in the long term \cite{GodrecheLuck2018}. These long term fluctuations vary according to the stationary Gaussian distribution. Finally, if $\mathbf x(0)$ is the Gaussian distribution with mean $\mathbf 0$ and covariance matrix $\mathbf S,$ then the OU process is stationary.

\section{Problem Statement} \label{sec: Problem Statement}
In this thesis, we will investigate Cyclicity Analysis for the \textit{stationary OU process} with model parameters $\mathbf B$ and $\boldsymbol \Sigma,$ in which $\mathbf x(0)$ is the Gaussian distribution with mean $\mathbf 0$ and covariance matrix $\mathbf S.$ We define the auxiliary \textit{lead process}, which is the $N \times N$ matrix-valued stochastic process $\left \lbrace \mathbf A(t) \right \rbrace_{t \ge 0},$ where 
	\begin{align} 
		\mathbf A(t) &= \frac{1}{2} \int_0^t \mathbf x(s) \ d \mathbf x^\text T(s) - d\mathbf x(s) \ \mathbf x^\text T(s)
	\end{align}
for each $t \ge 0.$ The right hand side is a random matrix whose entries are one-dimensional Ito stochastic integrals \cite{Oksendal2003} with respect to the component OU processes $\left \lbrace x_n(t) \right \rbrace_{t \ge 0}$. In particular, the $(m,n)$-th entry of $\mathbf A(t)$ is a one-dimensional Ito integral of the form
	\begin{align} \nonumber
		A_{m,n}(t) &= \frac{1}{2} \int_0^t x_m(s) \ dx_n(s) - x_n(s) \ dx_m(s).
	\end{align}
By definition, $\mathbf A(t)$ represents the lead matrix corresponding to a \textit{realization} of the OU process over the interval $[0,t].$  

We will prove the lead process obeys the strong law of large numbers identity:
	\begin{align}\label{eq: Lead Matrix Strong Law of Large Numbers Identity}
		\lim_{t \rightarrow \infty} \frac{\mathbf A(t)}{t}&= \frac{\mathbf B \mathbf S - \mathbf S \mathbf B^\text T}{2},
	\end{align}
almost surely. The identity \eqref{eq: Lead Matrix Strong Law of Large Numbers Identity} states we can infer the long term average pairwise leader follower dynamics amongst the $N$ component OU process through a time-averaged lead matrix corresponding to one single realization of the OU process. For this reason, we refer to the matrix
	\begin{align}\label{eq: OU Process Lead Matrix}
		 \mathbf Q &= \frac{\mathbf B \mathbf S - \mathbf S \mathbf B^\text T}{2}
	\end{align}
as the \textit{lead matrix} of the OU process.

Next, we consider a \textit{cyclic OU process}, in which the friction matrix $\mathbf B$ is specifically a \textit{circulant matrix}: if $(b_1 \ , \ \dots \ , \ b_N) \in \R^N$ is the first row of $\mathbf B,$ then the $(m,n)$-th entry of $\mathbf B$ is $b_{n-m+1},$ where $n-m+1$ is indexed mod $N.$  We assume $b_2 \ , \ \dots \ , \ b_N \le 0$ and $b_1>0$ is large enough so that $\mathbf B$ is a valid friction matrix. We consider a signal propagation model governed by this cyclic OU process. In our signal propagation model, we assume that there is a network of $N$ sensors located in space and an underlying one-dimensional signal propagating throughout the network. Similar signal propagation models are considered in \cite{GilsonKouvarisDecoZamora2018,AbrahamShahsavaraniZimmermanHusainBaryshnikov2021,Abraham2022}, for example. Letting $\left \lbrace \mathbf x(t) \right \rbrace_{t \ge 0}$ be the OU process, we interpret the $n$-th component of $\mathbf x(t)$ as a measurement made by the $n$-th sensor of the propagating effect at time $t \ge 0.$ The friction matrix $\mathbf B$ represents the cyclic sensor network structure. We refer to the constant $b_p$ as a \textit{propagation coefficient} for each $p>1,$ and we interpret its magnitude to represent how receptive the $n$-th sensor is to the activity within the neighboring $(n+p-1)$-th sensor for each $1 \le n \le N,$ in which $n+p-1$ is indexed mod $N.$ We refer to the constant $b_1$ as a \textit{suppression coefficient}, and we interpret its magnitude to represent how quickly the signal dissipates as it is propagating throughout the network. Meanwhile, the diffusion matrix $\mathbf D= \frac{\boldsymbol \Sigma \boldsymbol \Sigma^\text T}{2}$ is the covariance matrix whose $(m,n)$-th entry is the covariance of the $m$-th and $n$-th component noises injected into the $m$-th and $n$-th sensors, respectively.

For large dimension $N$ and for different choices of cyclic OU model parameters $\mathbf B$ and $\boldsymbol \Sigma,$ we derive the lead matrix $\mathbf Q$ of the cyclic OU process. Then, we investigate whether Cyclicity Analysis enables us to recover the network structure given by $\mathbf B.$ More specifically, we investigate whether the structure of the leading eigenvector $\mathbf v_1(\mathbf Q)$ recovers the structure of the network. 

For example, let $\mathbf B$ be the $N \times N$ circulant friction matrix such that the suppression coefficient $b_1>1,$ the propagation coefficient $b_2=-1,$ and all other propagation coefficients are $0.$ The cyclic network structure of $\mathbf B$ is such that the $n$-th sensor is only linked to the $(n+1)$-th sensor, the neighboring sensor to its right. Does the structure of the leading eigenvector $\mathbf v_1(\mathbf Q)$ reflect this ground truth network structure in any way ? For example, does the (cyclic) order of the leading eigenvector component phases match the expected (cyclic) order of the sensors receiving the signal as it propagates throughout the network ? In addition, do the leading eigenvector component moduli and the lead matrix eigenvalues indicate anything about the network structure ? 

We will investigate our problem statement under two specific regimes: the first regime is when the suppression coefficient $b_1$ is large i.e. the propagating signal quickly dissipates, and the second regime is when $b_1$ is small i.e. the propagating signal does not quickly dissipate. In the former regime, the circulant friction matrix $\mathbf B$ is perturbed significantly from an unstable matrix, while in the latter regime, $\mathbf B$ is perturbed slightly away from an unstable matrix. We investigate these regimes because they yield different types of asymptotic structures pertaining to the lead matrix $\mathbf Q$ and its leading eigenvector $\mathbf v_1(\mathbf Q).$

\section{Literature Review}
We conclude the introduction by conducting a literature review of the main topics of the thesis.

\subsection{Lead Lag Dynamics}
There are various tools to identify lead lag dynamics between two signals or between two time series. Such tools have been utilized to study financial markets \cite{LiWangSunLiu2022} and fMRI signals \cite{MitraSnyderHackerRaichle2014}. We already discussed one tool, namely the cross-correlation function. Another algorithmic tool is Dynamic Time Warping (DTW) \cite{SakoeChiba1978}. We briefly describe it. Consider two one-dimensional time series of the form $\left \lbrace x_m \right \rbrace_{m=1}^M$ and $\left \lbrace y_m \right \rbrace_{n=1}^N.$ The goal of DTW is to find the optimal alignment of such time-series. To this end, one constructs an $M \times N$ matrix $\mathbf C$ whose $(j,k)$-th entry corresponds to the optimal cost between the restricted time series $\left \lbrace x_m \right \rbrace_{m=1}^j$ and $\left \lbrace y_m \right \rbrace_{n=1}^k.$ The optimal cost is defined via the formula 
    \begin{align} \label{eq: DTW Optimal Cost Update}
        C_{j,k} &= \begin{cases} d(x_j, y_k) & j=k=1 \\
        d(x_j,y_k) + \min \left(C_{j-1,k} \ , \  C_{j-1,k-1} \ , \  C_{j,k-1} \right) & \text{else},
        \end{cases}
    \end{align}
where $d(\cdot,\cdot)$ is a specified distance metric. Then, a \textit{warping path} is any collection of indices $\left \lbrace (j_1 , k_1) \ , \ \dots \ , \ (j_P, k_P) \right \rbrace$ satisfying $(j_1,k_1)=(1,1)$ and $(j_P,k_P)=(M,N)$ and $0 \le j_{p+1}-j_{p} \ , \ k_{p+1}-k_{p}  \le 1$ for all $1 \le p<P,$ where $P \ge 2$ is a fixed length. Finally, the \textit{optimal alignment} between the original time-series $\left \lbrace x_m \right \rbrace_{m=1}^M$ and $\left \lbrace y_m \right \rbrace_{n=1}^N$ is the warping path $\left \lbrace (j_1 , k_1) \ , \ \dots \ , \ (j_P, k_P) \right \rbrace$ minimizing the quantity
    \begin{align}
        \sum_{p=1}^P C_{j_p,k_p}^2 \nonumber.
    \end{align}
DTW, however, has drawbacks. Seeking the warping path to guarantee optimal alignment is a computationally expensive problem. Moreover, DTW is prone to noise.

\subsection{Cyclicity Analysis}
We review the literature on Cyclicity Analysis. As mentioned earlier, Cyclicity Analysis was officially introduced in \cite{BaryshnikovSchlafly2016} as a means to study cyclic signals. Later on, a dissertation \cite{Abraham2022} was published demonstrating how Cyclicity Analysis of fMRI signals yields new insights into understanding the brain processes of patients with a neurological condition known as tinnitus. Moreover, the authors in \cite{AbrahamShahsavaraniZimmermanHusainBaryshnikov2021} published an experimental article comparing the results given by Cyclicity Analysis on fMRI signals to results given by other standard tools.

We emphasize that in this thesis, we are applying Cyclicity Analysis in a completely different setting. We are using it to investigate signals generated via a stochastic process.

\subsection{OU Process}
As mentioned earlier, the OU process is one of the most common stochastic processes with applications in Financial Mathematics and Biology. 

Firstly, many authors have studied the estimation of OU process model parameters. In \cite{SinghGhoshAdhikari2018}, for example, the authors showed using MAP estimation how to recover OU model parameters $\mathbf B$ and $\boldsymbol \Sigma,$ in which $\boldsymbol \Sigma$ is an $N \times N$ invertible matrix, given a time series that approximates a realization of an OU process. In \cite{GobetShe2016}, the authors showed how to approximate the stationary distribution of a generalized OU process with a modified drift coefficient. We emphasize that our problem statement is completely different from model parameter estimation. We are not trying to recover the friction matrix $\mathbf B$ itself from realizations of the OU process; rather, we interpret $\mathbf B$ as a weighted adjacency matrix associated with a network of sensors, and we are investigating whether we can recover the network structure induced by $\mathbf B$ from any realization of the OU process.

Moreover, the lead matrix $\mathbf Q$ of the OU process defined in \eqref{eq: OU Process Lead Matrix} has been studied in the literature as a means to characterize the time-irreversibility of the OU process \cite{GodrecheLuck2018,GilsonTagliazucchiConfre2023,Qian2001}. A stationary stochastic process $\left \lbrace \mathbf y(t) \right \rbrace_{t \ge 0}$ is said to be \textit{time-reversible} \cite{Qian2001} if for all $K \in \N$ for all times $0 \le t_1 \ , \ \dots \ , \ t_K \le \tau,$ the probability distributions of $\left(\mathbf y(t_1) \ , \ \dots \ , \ \mathbf y(t_K)\right)$ and of $\left(\mathbf y(\tau -t_1) \ , \ \dots \ , \ \mathbf y(\tau -t_K)\right)$ are equal. The main result is the stationary OU process is time reversible if and only if $\mathbf B \mathbf S$ is symmetric, where $\mathbf S$ is the stationary covariance matrix of the OU process. Equivalently, the OU process is time reversible if and only if the matrix $\mathbf Q$ vanishes.

To study the degree to which the OU process $\left \lbrace \mathbf x(t) \right \rbrace_{t \ge 0}$ is time irreversible, we briefly outline the approach described in \cite{GilsonTagliazucchiConfre2023}. Let  $P: \R^N \times [0, \infty) \rightarrow \R$ be the probability density function of $\mathbf x(t)$ at time $t \ge 0.$ Since we are considering the OU stationary process, for each $t \ge 0,$ the random vector $\mathbf x(t)$ has a Gaussian distribution with mean $\mathbf 0$ and covariance matrix $\mathbf S.$  Explicitly, we have
	\begin{align}
		P(\mathbf x,t) &= \frac{1}{\sqrt{(2 \pi)^N |\det(\mathbf S)| }} \  e^{-\frac{\mathbf x^\text T \mathbf S^{-1} \mathbf x}{2}}
	\end{align}
for each $t \ge 0.$ The density function $P$ satisfies the \textit{Fokker-Planck (Kolmogorov forward) equation} \cite{GilsonTagliazucchiConfre2023}, which is the parabolic partial differential equation
	\begin{align}\label{eq: Fokker Planck Equation}
		\frac{\partial P}{\partial t}  &= \nabla \cdot \left( \mathbf B \ \mathbf x \  P + \mathbf D \  \nabla P  \right),
	\end{align}
where $\mathbf D$ is the diffusion matrix, $\nabla P$ is the gradient of $P$ with respect to variable $\mathbf x= (x_1 \ , \ \dots \ , \ x_N),$ and $\nabla \cdot \mathbf F$ is the divergence of the vector field $\mathbf F: \R^N \rightarrow \R^N.$ Define the map $\mathbf J: \R^N \times [0, \infty) \rightarrow \R^N$ by the expression within the divergence operator on the right hand side of \eqref{eq: Fokker Planck Equation}, namely
	\begin{align}
		\mathbf J(\mathbf x, t) &= - \mathbf B \  \mathbf x \  P(\mathbf x,t) - \mathbf D \  \nabla P(\mathbf x,t).
	\end{align}
The map $\mathbf J$ is known as the \textit{probability flux (current)} \cite{GodrecheLuck2018, GilsonTagliazucchiConfre2023}. In general, the probability flux describes how probabilities associated with different states in the state space of a stochastic process change over time. Intuitively, when the probability flux is nonzero, the realizations follow paths that lead towards regions with higher probability density i.e. those states with higher probability of occurrence.

For each $t \ge 0,$ upon recognizing that $\nabla P(\mathbf x,t) = -  \mathbf S^{-1} \  \mathbf x \ P(\mathbf x,t),$ we can rewrite the probability flux in the following way:
	\begin{align}
		\mathbf J(\mathbf x, t) &= - (\mathbf B \mathbf x  \ P(\mathbf x,t) + \mathbf D \mathbf S^{-1} \mathbf x  \ P(\mathbf x,t)) \nonumber \\
		&= \left(\mathbf D \mathbf S^{-1} - \mathbf B \right) \mathbf x \  P(\mathbf x,t) \nonumber \\
		&= \left(\frac{\mathbf B \mathbf S + \mathbf S \mathbf B^\text T}{2} \  \mathbf S^{-1} - \mathbf B \right) \mathbf x  \ P(\mathbf x,t) \nonumber  \\
		&=\left( \frac{\mathbf S \mathbf B^\text T \mathbf S^{-1} - \mathbf B}{2}  \right) \mathbf x \  P(\mathbf x,t)  \nonumber \\
		&=\left(\frac{\left(\mathbf B \mathbf S - \mathbf 2 \mathbf Q \right) \mathbf S^{-1} - \mathbf B}{2}  \right) \mathbf x \  P(\mathbf x,t)  \nonumber \\
		&= - \mathbf Q \mathbf S^{-1}  \mathbf x \  P(\mathbf x,t) \nonumber.
	\end{align}
Thus, the stationary OU process is time-reversible if and only if the probability flux $\mathbf J$ vanishes. If the flux does not vanish, then observe $\mathbf Q$ is involved in the measurement of the flux.

\chapter{Preliminaries}\label{chap: Preliminaries}

In this chapter, we will briefly review the needed notations, definitions, and standard results from Linear Algebra, Probability Theory, and Stochastic Processes that will be used throughout this thesis.

\section{Linear Algebra}\label{sec: Linear Algebra}
Throughout this section, we let $\F \in \left \lbrace \R , \C \right \rbrace$ and we fix $M,N \in \N.$  We denote $\Mat_{M,N}(\F)$ as the set of all $M \times N$ matrices with entries in $\F,$ and we tacitly identify the $N$-dimensional vector space $\F^N$ with $\Mat_{N,1}(\F).$ Recall if $\mathbf A \in \Mat_{M,N}(\F),$ then there is a unique linear operator between $\F^N$ and $\F^M$ whose transformation matrix is $\mathbf A.$ Therefore, we will interchangeably refer to $\mathbf A$ as either a matrix or as a linear operator between the appropriate vector spaces.

Let $\mathbf A \in \Mat_{M,N}(\F)$ and $\mathbf x \in \F^N.$ We recall some standard linear algebra notation. For each $1 \le m \le M$ and each $1 \le n \le N,$ we denote $A_{m,n}$ as the $(m,n)$-th entry of $\mathbf A,$ and we denote $x_n$ as the $n$-th component of $\mathbf x.$ We denote $\mathbf A^ \text T$ as the transpose of $\mathbf A$ and denote $\mathbf A^*=\overline{\mathbf A^ \text T}$ as the conjugate-transpose of $\mathbf A.$ We denote $\mathbf I$ as the $N \times N$ identity matrix, in which the value of $N$ will be clear in context. We denote $\mathbf 0$ as either the $M \times N$ zero matrix or the zero vector in $\F^N,$ in which the situation, along with the values of $M$ and $N$, will be clear in context. Finally, we denote $\mathbf e_n \in \R^N$ as the $n$-th canonical basis vector having $n$-th component equal to $1$ and all other components equal to $0.$

Now, let $\mathbf A \in \Mat_{N,N}(\F).$ We denote $\det(\mathbf A), \text{tr}(\mathbf A),$ and $\mathbf A^{-1}$ as the determinant, trace, and inverse of $\mathbf A,$ respectively. Recall the matrix exponential of $\mathbf A$ is the power series
    \begin{align}\label{eq: Matrix Exponential}
	\exp(\mathbf A) &=\sum_{k \ge 0} \frac{\mathbf A^k}{k!}.
    \end{align}

We state the basic properties of the matrix exponential that we will be using later in the thesis. Proofs of these properties can be found in \cite{HornJohnson2012}, for example.

\begin{theorem}[Matrix Exponential Properties]\label{thm: Matrix Exponential Properties}
If $\mathbf A \in \Mat_{N,N}(\R),$ then the following properties hold.
	\begin{enumerate}[label=\roman*)]
		\item $\exp(\mathbf A)$ commutes with $\mathbf A$ under matrix multiplication.
		\item $\exp(\mathbf A)$ is invertible with $(\exp(\mathbf A))^{-1}=\exp(\mathbf -\mathbf A).$
		\item We have 
			\begin{align}
				\left(\exp{\mathbf A} \right)^ \text T=\exp \left(\mathbf A^ \text T \right).
			\end{align}

		\item The matrix-valued map $t \mapsto \exp(t \mathbf A)$ is differentiable on $\R$. In particular, 
			\begin{align}
				(\exp{t \mathbf A})' &= \mathbf A \ \exp(t \mathbf A) = \exp(t \mathbf A) \ \mathbf A.
			\end{align}
	\end{enumerate} 
\end{theorem}

Let $\mathbf A \in \Mat_{N,N}(\F).$ The matrix $\mathbf A$ is \textit{Hurwitz} if all of its eigenvalues have negative real parts. Note if $\mathbf A$ is Hurwitz, then $\lim_{t \rightarrow \infty} \exp(t \mathbf A)= \mathbf 0.$ The matrix $\mathbf A$ is  \textit{Hermitian} if $\mathbf A^* =\mathbf A$ and \textit{skew-Hermitian (anti-Hermitian)} if $\mathbf A^*=-\mathbf A.$ Note that Hermitian real matrices are the same as symmetric real matrices, while skew-Hermitian real matrices are the same as skew-symmetric real matrices. We state the important properties of Hermitian and skew-Hermitian matrices. One may consult \cite{HornJohnson2012} for proofs.

\begin{theorem}[Hermitian and skew-Hermitian Matrix Properties]\label{thm: Hermitian and skew-Hermitian Matrix Properties}
If $\mathbf A \in \Mat_{N,N}(\F)$ is either Hermitian or skew-Hermitian, then the following properties hold
	\begin{enumerate}[label=\roman*)]
		\item $\mathbf A$ is skew-Hermitian if and only if $i \mathbf A$ is Hermitian.
		\item If $\mathbf A$ is Hermitian or skew-Hermitian, then $\mathbf A$ is diagonalizable; there is an orthonormal basis of $\C^N$ consisting of the eigenvectors of $\mathbf A.$ 
		\item If $\mathbf A$ is Hermitian, then all its eigenvalues are real. If $\mathbf A$ is skew-Hermitian, then all eigenvalues of $\mathbf A$ are either $0$ or purely imaginary. 
	\end{enumerate}
\end{theorem} 

Finally, let $\mathbf A \in \Mat_{N,N}(\F).$ The matrix $\mathbf A$ is \textit{circulant} if there exists a sequence of numbers $\left \lbrace a_n \right \rbrace_{1 \le n <N}$ such that $A_{m,n} = a_{m-n+1},$ where $m-n+1$ is indexed mod $N.$ Circulant matrices obey several properties \cite{Gray2006}. Circulant matrices form a subspace of $\Mat_{N,N}(\F).$ They commute under matrix multiplication. The product of circulant matrices is circulant, and the transpose of a circulant matrix is circulant. 

We state and prove the properties pertaining to the eigenvalues and eigenvectors of circulant matrices that are essential for this thesis.
\begin{theorem}[Circulant Matrix Properties]\label{thm: Circulant Matrix Properties}
If $\mathbf A, \mathbf B \in \Mat_{N,N}(\F)$ are circulant matrices, then the following properties hold.
    \begin{enumerate}[label=\roman*)]       
        \item $\mathbf A$ is diagonalizable. In particular, if the first row of $\mathbf A$ is of the form $(a_1 \ , \ \dots \ , \ a_{N}),$ then for each $1 \le n \le N,$ the vector 
            \begin{align}
                \mathbf w_n &= \sqrt {\frac{1}{N}}  \left (1 \ , \  \omega_n \ , \ \dots \ , \ \omega_n^{N-1} \right) \label{eq: Circulant Matrix Eigenvector},
            \end{align}
        is an eigenvector of $\mathbf A,$  associated with the eigenvalue 
            \begin{align}
                \mu_n&=   \sum_{p=1}^{N} a_{p} \  \omega_n^{p-1} \label{eq: Circulant Matrix Eigenvalue},
            \end{align}
        where $\omega_n= e^{\frac{2 \pi i n}{N}}$ is an $N$-th root of unity. 
        Furthermore, the collection of eigenvectors $\left \lbrace \mathbf w_n \right \rbrace_{n=1}^N$ forms an orthonormal basis of $\C^N,$ and we have $\mathbf w_n^*=\mathbf w_{N-n}^T$ for each $1 \le n \le N.$ 

        \item For each $k \ge 0,$ we have 
        	\begin{align}\label{eq: Circulant Matrix Spectral Sum}
        		\mathbf A^k &= \sum_{n=1}^N \mu_n^k \   \mathbf w_n \mathbf w_{N-n}^\text T.
        	\end{align}
        
        \item For each $t \in \mathbb R,$ the matrix exponential of $t \mathbf A$ is explicitly 
            \begin{align}\label{eq: Circulant Matrix Exponential}
                \exp(t \mathbf A) &= \sum_{n=1}^N e^{\mu_n  t} \  \mathbf w_n \mathbf w_{N-n}^ \text T .
            \end{align}
        Moreover, $\exp(t \mathbf A)$ is a circulant matrix.
    \end{enumerate}
\end{theorem}

\begin{proof}
We prove the first statement. For each $1 \le j, n \le N,$ the $j$-th component of $\mathbf A \mathbf w_n$ is of the form 
	\begin{align}
		\sqrt{\frac{1}{N}} \left(\sum_{p=1}^{N}  a_{p-j+1} \ \omega_n^{p-1} \right) &= \sqrt{\frac{1}{N}} \left( \sum_{p=1}^{N-j+1}  a_{p} \ \omega_n^{p+j-2} + \sum_{p=N-j+2}^{N}  a_{p} \ \omega_n^{ p - (N-j+2) } \right) \nonumber \\
		&= \sqrt{\frac{1}{N}} \left(\sum_{p=1}^{N-j+1}  a_{p} \ \omega_n^{p+j-2} + \sum_{p=N-j+2}^{N}  a_{p} \ \omega_n^{p+j-2} \right)  \nonumber  \\
		&=  \sqrt{\frac{1}{N}}  \sum_{p=1}^{N}  a_{p} \ \omega_n^{p+j-2} \nonumber \\
		&=   \left(\sum_{p=1}^{N}  a_{p} \ \omega_n^{p-1} \right ) \  \frac{ \omega_n^{j-1}}{\sqrt N} \nonumber \\
		&=  \mu_n  \  \frac{\omega_n^{j-1}}{\sqrt N} \nonumber,
	\end{align}
which is the $j$-th component of $\mu_n \mathbf w_n.$ This shows $\mu_n$ is an eigenvalue of $\mathbf A$ with associated eigenvector $\mathbf w_n.$ Next, note that 
	\begin{align}
		 \omega_{N-n} &= e^{\frac{2 \pi i (N-n)}{N}} \nonumber \\
		 &= e^{-\frac{2 \pi i n}{N}} = \overline{\omega_n},
	\end{align}
for each $1 \le n \le N,$ which implies $\mathbf w_n^*= \mathbf w^\text T_{N-n}.$ Finally, for all $1 \le m,n \le N$ with $m \ne n,$ we use the finite geometric series identity 
	\begin{align}
		\sum_{j=1}^N r^j &= \frac{r- r^{N+1}}{1-r} \nonumber 
	\end{align}
for all $r \ne 1,$ to see that
	\begin{align}
		 \mathbf w_m^* \mathbf w_n &= \frac{1}{N} \sum_{p=1}^N \omega_{N-m}^{p-1} \omega_n^{p-1} \nonumber  \\
		  &= \frac{1}{N} \sum_{p=1}^N \omega_{N-m+n}^{p-1}  \nonumber \\
		  &=\frac{1}{N} \ \left(\frac{1- \omega_{N-m+n}^N }{1- \omega_{N-m+n}}\right) \nonumber  \\
		  &= 0 \nonumber.
	\end{align}
If $n=m,$ then 
	\begin{align}
		\mathbf w_n^* \mathbf w_n  &= \frac{1}{N} \sum_{p=1}^N \omega_{N-n}^{p-1} \omega_n^{p-1} \nonumber \\
		&= \frac{1}{N} \sum_{p=1}^N \omega_{N}^{p-1} \nonumber \\
		&=\frac{N}N=1 \nonumber.
	\end{align}
Since $\left \lbrace \mathbf w_n \right \rbrace_{n=1}^N$ is a set of $N$ orthonormal vectors in $\C^N,$ we deduce it is a basis for $\C^N.$ This proves the first statement.

We prove the second statement. Let $k \ge 0$ be fixed. Because $\mu_n$ is an eigenvalue of $\mathbf A,$ note that $\mu_n^k$ is an eigenvalue of $\mathbf A^k$ with associated eigenvector $\mathbf w_n.$ For each $1 \le m \le N,$ the orthonormality of the collection of eigenvectors $\left \lbrace \mathbf w_n \right \rbrace_{n=1}^N$ yields
	\begin{align}
	\sum_{n=1}^N  \mu_n^k \ \mathbf w_n \mathbf w_{N-n}^\text T \mathbf w_m &= \sum_{n=1}^N  \mu_n^k \  \mathbf w_n  \ \delta_{m,n} \nonumber \\
	&= \mu_m^k \mathbf w_m = \mathbf A^k \mathbf w_m \nonumber.
	\end{align}
This shows that the linear operator $\mathbf A^k$ and the linear operator $\sum_{n=1}^N  \lambda_n^k \ \mathbf w_n \mathbf w_{N-n}^\text T$ agree on the basis of eigenvectors for each $k \ge 0$. Because every vector in $\C^N$ is a unique linear combination of vectors from this basis, the operators agree on all of $\C^N$ and thus are identical. This proves the second statement. 

Now, we prove the last statement. By the definition of the matrix exponential in \eqref{eq: Matrix Exponential}, we have 
	\begin{align}
		\sum_{k \ge 0} \frac{(t \mathbf A)^k}{k!} &= \sum_{k \ge 0} \frac{t^k \mathbf A^k}{k!} \nonumber \\
		&=  \sum_{k \ge 0} \frac{t^k}{k!} \sum_{n=1}^N \mu_n^k \ \mathbf w_n \mathbf w_{N-n}^\text T \nonumber \\
		&= \sum_{n=1}^N \sum_{k \ge 0} \frac{(\mu_n t)^k}{k!}   \mathbf w_n \mathbf w_{N-n}^\text T \label{eq: Circulant Matrix Exponential 1} \\
		&= \sum_{n=1}^N e^{\mu_n t} \  \mathbf w_n \mathbf w_{N-n}^\text T,  \label{eq: Circulant Matrix Exponential 2}
	\end{align}
in which we obtained
\eqref{eq: Circulant Matrix Exponential 2} by simplifying the inner sum in \eqref{eq: Circulant Matrix Exponential 1} using the Taylor series of the ordinary exponential function:
	\begin{align}
		e^z &= \sum_{k \ge 0} \frac{z^k}{k!}. \nonumber
	\end{align}
The product of circulant matrices is circulant and the scalar multiple of a circulant matrix is also circulant. Therefore, for each $t \in \R,$ the matrix $(t \mathbf A)^k$ is a circulant matrix. Since the sum of circulant matrices is circulant, the partial sum
	\begin{align}
		\sum_{k=0}^K \frac{(t \mathbf A)^k}{k!} \nonumber 
	\end{align}
is circulant for each $K \in \N.$ Thus, upon taking the limit as $K \rightarrow \infty,$ we deduce $\exp(t \mathbf A)$ is a circulant matrix.
\end{proof}

We will adopt the following notation in this thesis. If $c_1 \ , \ \dots \ , \ c_N \in \F,$ then we denote $\text{Circ}(c_1 \ , \ \dots \ ,\  c_N)$ as the $N \times N$ circulant matrix whose first row is $(c_1 \ , \ \dots \ ,\  c_N).$

\section{Probability Theory}\label{sec: Probability Theory}
Throughout this section, we let $\left(\Omega,\mathcal A, \P \right)$ be a probability space, and we also fix $M,N \in \N.$ Recall a random vector on $\Omega$ is a measurable function between $\Omega$ and $\R^N,$ in which $\R^N$ is equipped with its Borel sigma algebra.

Let $\mathbf x$ and $\mathbf y$ be two $M$-dimensional and $N$-dimensional random vectors on $\Omega,$ respectively.  We denote $\E[\mathbf x]$ as the \textit{expectation (mean)} of $\mathbf x$ and recall
        \begin{align}
	\Cov(\mathbf x, \mathbf y) &= \E \left[\mathbf x \mathbf y^\text T \right]  - \E \left[\mathbf x  \right ] \ \left(\E \left[ \mathbf y \right ] \right)^\text T \nonumber
	\end{align}
is the $M \times N$ \textit{cross-covariance matrix} between $\mathbf x$ and $\mathbf y.$ The \textit{covariance matrix} of $\mathbf x$ is the positive semidefinite matrix $\Cov(\mathbf x, \mathbf x).$ 

Let $\mathbf x$ be an $M$-dimensional  random vector on $\Omega.$ Recall  $\mathbf x$ is  \textit{($M$-dimensional) Gaussian} if its probability density function satisfies
	\begin{align} \label{eq: Multivariate Gaussian Probability Density Function}
		f_{\mathbf x}(\mathbf z) &= \frac{1}{\left(2 \pi \right)^{\frac{M}{2}} \ \sqrt{|\det(\mathbf S)|}} \ e^{\frac{-(\mathbf z - \boldsymbol \mu)^\text T \mathbf S^{-1} (\mathbf z - \boldsymbol \mu) }{2}}.
	\end{align} 
for some  $\boldsymbol \mu \in \R^M$ and positive definite matrix $\mathbf S \in \Mat_{M,M}(\R).$ Note that $\mathbf x$ has mean $\boldsymbol \mu$ and covariance matrix $\mathbf S,$ and we write $\mathbf x \sim \mathcal N(\boldsymbol \mu, \mathbf S).$  Recall that any linear combination of independent Gaussian random vectors is also a Gaussian random vector. Moreover,  any linear transformation of a Gaussian random vector is also a Gaussian random vector.

\section{Stochastic Processes}\label{sec: Stochastic Processes}

Throughout this section, we let $\left(\Omega,\mathcal A , \P \right)$ be probability space, $\left \lbrace \mathcal F_t  \right \rbrace_{t \ge 0}$ be a filtration on $\Omega,$ and we fix $N \in \N.$ Recall an event $A \in \mathcal A$ occurs ($\P$)-\textit{almost surely (with probability one)} if $\P(A^c)=0.$ 

Recall an \textit{($N$-dimensional) stochastic process} on $\Omega$ is a map $\mathbf x: [0, \infty) \times \Omega \rightarrow \R^N$ such that for each fixed $t \ge 0,$ the map $\mathbf x(t, \cdot): \Omega \rightarrow \R^N$ is an $N$-dimensional random vector on $\Omega.$ A \textit{realization (sample path/trajectory)} of $\mathbf x$ is a map of the form $\mathbf x(\cdot, \omega): [0,\infty) \rightarrow \R^N,$ where $\omega \in \Omega$ is fixed. From henceforth, we adopt the notation $\left \lbrace \mathbf x(t) \right \rbrace_{t \ge 0},$ in which we write $\mathbf x(t)=\mathbf x(t, \cdot)$ as the random vector at the time $t .$ The process $\left \lbrace \mathbf x(t) \right \rbrace_{t \ge 0}$ is \textit{adapted} to the filtration $\left \lbrace \mathcal F_t \right \rbrace_{t \ge 0}$ if $\mathbf x(t)$ is a random vector on $\Omega$ specifically equipped with the sigma algebra $\mathcal F_t$ for each $t \ge 0.$ In particular, the \textit{standard (natural) filtration} of the process $\left \lbrace \mathbf x(t) \right \rbrace_{t \ge 0}$ is the filtration $\left \lbrace \mathcal S_t \right \rbrace_{t \ge 0},$ such that $\mathcal S_t$ is the sigma algebra generated by the collection of random vectors $\left \lbrace \mathbf x(s) \right \rbrace_{s=0}^t$ for each $t \ge 0.$ Note that $\left \lbrace \mathbf x(t) \right \rbrace_{t \ge 0}$ is always adapted to its own standard filtration.

Let $\left \lbrace \mathbf x(t) \right \rbrace_{t \ge 0}$ be a stochastic process. We recall the following time dependent statistical functions associated with the process. The \textit{mean function}, \textit{autocovariance function}, and the \textit{covariance function} of $\left \lbrace \mathbf x(t) \right \rbrace_{t \ge 0}$ are defined respectively as the maps $\boldsymbol \mu_{\mathbf x} : [0,\infty) \rightarrow \R^N \ , \  \boldsymbol \Gamma_{\mathbf x} : [0,\infty)^2 \rightarrow \Mat_{N,N}(\R),$ and $\mathbf V_\mathbf x: [0,\infty) \rightarrow \Mat_{N,N}(\R),$ where 
    \begin{align}
        \boldsymbol \mu_{\mathbf x}(t)  &= \E[\mathbf x(t)] \label{eq: Stochastic Process Mean Function} \\
        \boldsymbol \Gamma_{\mathbf x}(s,t)  &= \Cov(\mathbf x(s) , \mathbf x(t)) \label{eq: Stochastic Process Autocovariance Function} \\
        \mathbf V_{\mathbf x}(t) &= \boldsymbol \Gamma_{\mathbf x}(t,t) \label{eq: Stochastic Process Covariance Function}.
    \end{align}

The process $\left \lbrace \mathbf x(t) \right \rbrace_{t \ge 0}$ is \textit{(strictly) stationary} if the cumulative distribution functions of $\mathbf x(t)$ and $\mathbf x(t+h)$ are identical for all $t,h \ge 0.$ The process $\left \lbrace \mathbf x(t) \right \rbrace_{t \ge 0}$ is \textit{weakly (wide sense) stationary} if its mean function $\boldsymbol \mu_\mathbf x$ is identically constant and the map $(t,h) \mapsto \boldsymbol \Gamma_\mathbf x(t,t+h)$ between $[0,\infty)^2$ and $\Mat_{N,N}(\R)$ depends only on the variable $h.$ Note that all stationary processes are weakly stationary.

Let $\left \lbrace \mathbf x(t) \right \rbrace_{t \ge 0}$ be a stochastic process adapted to the filtration $\left \lbrace \mathcal F_t \right \rbrace_{t \ge 0}.$ The process $\left \lbrace \mathbf x(t) \right \rbrace_{t \ge 0}$  is \textit{Markov} with respect to the filtration $\left \lbrace \mathcal F_t \right \rbrace_{t \ge 0}$ if for all  $s,t \ge 0$ and any measurable, bounded function $f: \R^N \rightarrow \R,$ we have
	\begin{align}
		\E \left[f(\mathbf x(s+t)) \ |  \ \mathcal F_s \right] &= \E \left[f(\mathbf x(s+t)) \ |  \ \mathbf x(s)  \right] \nonumber
	\end{align}
almost surely. We simply say $\left \lbrace \mathbf x(t) \right \rbrace_{t \ge 0}$  is Markov if it is Markov with respect to its standard filtration.  Moreover, the process $\left \lbrace \mathbf x(t) \right \rbrace_{t \ge 0}$ is a \textit{diffusion process} if it is a Markov process having continuous sample paths with probability one. The process $\left \lbrace \mathbf x(t) \right \rbrace_{t \ge 0}$ is a \textit{martingale} with respect to the filtration $\left \lbrace \mathcal F_t \right \rbrace_{t \ge 0}$ if for all  $0 \le s \le t,$ the expectation $\E \left [ \left |\mathbf x(t) \right | \right ]$ exists and $\E \left[ \mathbf x(t) \ | \ \mathcal F_s \right] = \mathbf x(s).$ We simply say $\left \lbrace \mathbf x(t) \right \rbrace_{t \ge 0}$  is a martingale if it is a martingale with respect to its standard filtration.

Let $\left \lbrace \mathbf x(t) \right \rbrace_{t \ge 0}$ be a stochastic process. The process is said to be \textit{($N$-dimensional) Gaussian} if for all $K \in \N$ and all times $t_1 \ , \ \dots \ ,\  t_K \ge 0$ and all $\alpha_1 \ , \ \dots \ , \ \alpha_K \in \R,$ the linear combination $\sum_{k=1}^K \alpha_k \  \mathbf x(t_k)$ is a Gaussian random vector. For Gaussian processes, the notions of stationarity and weak stationarity are equivalent. The process $\left \lbrace \mathbf x(t) \right \rbrace_{t \ge 0}$ is the \textit{($N$-dimensional) standard Wiener process (Brownian Motion)} if it is a diffusion process such that $\mathbf x(0)= \mathbf 0$ almost surely with independent, Gaussian increments: for all $0 \le s_1 < t_1 \le s_2 < t_2,$ both the increments $\mathbf x(t_2)-\mathbf x(s_2)$ and $\mathbf x(t_1)-\mathbf x(s_1)$ are independent and $\mathbf x(t_1)-\mathbf x(s_1) \sim \mathcal N(\mathbf 0, (t_1-s_1) \  \mathbf I),$ where $\mathbf I$ is the $N \times N$ identity matrix. Note that the Wiener process is a Gaussian process and a martingale. Furthermore, its component processes $\left \lbrace w_m(t) \right \rbrace_{t \ge 0}$ are independent for each $1 \le m \le M.$

Let $\left \lbrace \mathbf x(t) \right \rbrace_{t \ge 0}$ be a Markov process adapted to the filtration $\left \lbrace \mathcal F_t \right \rbrace_{t \ge 0},$ and let the map $(t, \omega) \mapsto \mathbf x(t)(\omega)$ defined between $[0, \infty) \times \Omega$ and  $\R^N$ be $\mathcal B([0,\infty)) \times \mathcal A$ measurable, where $\mathcal B([0,\infty))$ is the Borel sigma algebra on $[0, \infty).$ Let $\left \lbrace \theta_t \right \rbrace_{t \ge 0}$ be a collection of shift operators on $\Omega$ such that $  \mathbf x(t)(\theta_s)=\mathbf x(s+t)$ for all $s,t \ge 0.$ For each $t \ge 0$ and $B \subset \R^N$ and $\mathbf y \in \R^N,$ let $$p_t(\mathbf y,B)=\P(\mathbf x(t) \in B \ | \  \mathbf x(0) = \mathbf y).$$ A probability measure $\pi$ on $\R^N$ equipped with the Borel sigma algebra is \textit{invariant} with respect to $\left \lbrace \mathbf x(t) \right \rbrace_{t \ge 0}$ if 
	\begin{align}
		\int_{\R^N} p_t(\mathbf y, B) \ d \pi (\mathbf y) &= \pi(B) \nonumber
	\end{align}
for each $t \ge 0$ and each $B \subset \R^N.$ A set $A \in \mathcal A$ is \textit{shift invariant} if $\theta_t^{-1}(A)=A$ for all $t \ge 0.$ The Markov process $\left \lbrace \mathbf x(t) \right \rbrace_{t \ge 0}$ is \textit{ergodic} if it admits an invariant probability measure $\pi$ satisfying $\P_\pi(A) \in \lbrace 0,1 \rbrace$ for every shift invariant set $A,$ in which $\P_\pi$ is the probability measure defined by
	\begin{align}
		\P_\pi(B) &= \int_{\R^N}  p_t(\mathbf y, B) \ d \pi(\mathbf y).
	\end{align}
	
The most important theorem pertaining to Markov ergodic processes is Birkhoff's Ergodic Theorem \cite{Sandric2017}.

\begin{theorem}[Birkhoff's Ergodic Theorem]\label{thm: Birkhoff's Ergodic Theorem}
If $\left \lbrace \mathbf x(t) \right \rbrace_{t \ge 0}$ is a Markov ergodic process, then for any $p \in \N$ and for any measurable $f :\R^N \rightarrow \R$ satisfying the condition that $\int_{\R^N} (f(\mathbf y))^p \ d \pi(\mathbf y)$ exists, we have 
	\begin{align}
		\lim_{t \rightarrow \infty} \frac{1}t \int_0^t f(\mathbf x(s)) \ ds &= \E_\pi[f(\mathbf x(0))].
	\end{align}
$\P_\pi$-almost surely, where 
	\begin{align}
		\E_\pi[f(\mathbf y)] &= \int_{\R^N} f(\mathbf y) \ d \pi(\mathbf y).
	\end{align}
\end{theorem}

\section{Stochastic Differential Equations}
As before, let $(\Omega, \mathcal A, \mathbb P)$ be a probability space and fix $M,N \in \N.$ Let $\left \lbrace \mathbf w(t) \right \rbrace_{t \ge 0}$ be the standard $M$-dimensional Wiener process. In this section, we briefly review the theory of stochastic differential equations. 

Let $\mathbf F : [0,\infty) \times \Omega \rightarrow \Mat_{N,M}(\R)$ be a map such that the $MN$-dimensional stochastic process $\left \lbrace \text{vec}(\mathbf F(t , \cdot)) \right \rbrace_{t \ge 0}$ is adapted to the standard filtration of the Wiener process, where $\vec$ is the vectorization operator that maps $N \times M$ matrix to the $NM$-dimensional vector formed by stacking the columns of the matrix one after another. The \textit{time integral} of $\mathbf F$ and \textit{stochastic integral} of $\mathbf F$ with respect to the Wiener process are respectively defined as
    \begin{align} \label{eq: Time Integral}
        \int_0^t \mathbf F(s, \omega) \ ds &= \lim_{K \rightarrow \infty} \sum_{k=1}^{K-1} \mathbf F(s_k , \omega) \ \left( s_{k+1} -s_k \right) \\
        \int_0^t \mathbf F(s, \omega)  \ d \mathbf w(s) &= \lim_{K \rightarrow \infty} \sum_{k=1}^{K-1} \mathbf F(s_k , \omega) \ \left(\mathbf w(s_{k+1}) -\mathbf w(s_{k}) \right)  \label{eq: Ito Integral},
    \end{align}
where $\left \lbrace s_k \right \rbrace_{k=1}^K$ is an increasing sequence in $[0,t]$ satisfying $s_1=0$ and $s_K=t.$ Note the limits in \eqref{eq: Time Integral} and  \eqref{eq: Ito Integral} are taken in the sense of $L^2$ i.e. the mean-squared limit.

An \textit{Ito process} is a process $\left \lbrace \mathbf x(t) \right \rbrace_{t \ge 0}$ adapted to the standard filtration of the Wiener process satisfying
	\begin{align}\label{eq: Ito Process}
		\mathbf x(t) &= \mathbf x(0) + \int_0^t \mathbf b(s, \mathbf x(s)) \ ds+ \int_0^t \boldsymbol \sigma(s,\mathbf x(s)) \ d \mathbf w(s),
	\end{align} 
where $\mathbf b: [0,\infty) \times \R^N \rightarrow \R^N$ and $\boldsymbol \sigma : [0,\infty) \times \R^N \rightarrow \Mat_{N,M}(\R)$ are measurable functions. Informally, we write \eqref{eq: Ito Process} as an equation of the form
    \begin{align} \label{eq: SDE}
        d \mathbf x(t) &= \mathbf b(t, \mathbf x(t)) \ dt + \boldsymbol \sigma(t, \mathbf x(t)) \ d \mathbf w(t),
    \end{align}
which we refer to as a \textit{stochastic differential equation (SDE)}. We interpret any Ito process in \eqref{eq: Ito Process} as a solution to the SDE \eqref{eq: SDE}.

We state the existence and uniqueness criteria \cite[Theorem 5.2.1]{Oksendal2003} pertaining to the solution of the SDE \eqref{eq: SDE}.
\begin{theorem}[SDE Solution Existence and Uniqueness Conditions]\label{thm: SDE Solution Existence and Uniqueness}
Consider the SDE \eqref{eq: SDE}. Suppose there exist constants $C,D$ satisfying 
    \begin{align} 
    	\|\mathbf b(t,\mathbf y) \| + \| \boldsymbol \sigma(t,\mathbf y) \|_F &\le C (1+ \| \mathbf y\|) \label{eq: SDE Existence Uniquenness Condition 1} \\
    	\|\mathbf b(t,\mathbf z) - \mathbf b(t,\mathbf y)\|   &\le D (\| \mathbf z - \mathbf y \|) \label{eq: SDE Existence Uniquenness Condition 2},
    \end{align}
for all $\mathbf y, \mathbf z \in \R^N$ and all $t \ge 0,$  where $\| \cdot \|_F$ is the Frobenius norm and $\|\cdot \|$ is the Euclidean $2$-norm. Suppose further that the initial random vector $\mathbf x(0)$ satisfies $\E[\|\mathbf x(0)\|^2]$ being finite and $\mathbf x(0)$ is independent of the Wiener process. Then, the SDE \eqref{eq: SDE}
has a unique solution for $\mathbf x(t),$ which is a diffusion process adapted to the standard filtration of the Wiener process. Moreover, the solution is square integrable i.e. $\E \left [\int_0^t \| \mathbf x(s)\|^2 \ ds \right]$ is finite for each $t \ge 0.$
\end{theorem}

\chapter{The Ornstein-Uhlenbeck Process} \label{chap: ouprocess}

Let $(\Omega, \mathcal A, \mathbb P)$ be a probability space. Recall in the introduction, we defined an $N$-dimensional OU Process on this probability space, namely the $N$-dimensional stochastic process $\left \lbrace \mathbf x(t) \right \rbrace_{t \ge 0}$ satisfying the SDE
    \begin{align} \label{eq: OU Process SDE} 
        d \mathbf x(t)= - \mathbf B  \  \mathbf x(t) \ dt + \boldsymbol \Sigma \ d \mathbf w(t).
    \end{align}
The linear operator $\mathbf B \in \Mat_{N,N}(\R)$ is called the friction operator which satisfies the condition $-\mathbf B$ is Hurwitz. The linear operator $\boldsymbol \Sigma \in \Mat_{M,N}(\R)$ is called the volatility operator. The process $\left \lbrace \mathbf w(t) \right \rbrace_{t \ge 0}$ is the $M$-dimensional standard Wiener Process. Moreover, we assume that $\mathbf x(0) \sim \mathcal N(\mathbf 0 ,\mathbf S)$ is independent of the Wiener process, in which we recall $\mathbf S \in \Mat_{N,N}(\R)$ is the stationary covariance operator satisfying the Lyapunov equation
    \begin{align} \label{eq: OU Process Stationary Covariance Matrix Lyapunov Equation}
	\mathbf B \mathbf S + \mathbf S \mathbf B^\text T &= 2 \mathbf D,
    \end{align}
and
    \begin{align}\label{eq: OU Process Diffusion Matrix}
        \mathbf D &= \frac{\boldsymbol \Sigma \boldsymbol \Sigma^\text T}{2}
    \end{align}
is the diffusion operator. We view the friction operator $\mathbf B$ and volatility operator $\boldsymbol \Sigma$ as OU model parameters.

In this thesis, we will utilize the matrix-valued function $\mathbf G: \R \rightarrow \Mat_{N,N}(\R)$ defined by
    \begin{align} \label{eq: OU Process Green Function}
        \mathbf G(t) &= \exp(-t \mathbf B),
    \end{align}
which is commonly known as the \textit{Green's function (propagator)} in the literature \cite{GodrecheLuck2018}. We note that $\mathbf G$ is the fundamental solution operator to the linear ordinary differential equation $\dot{\mathbf y} = - \mathbf B \mathbf y.$

The stationary covariance matrix $\mathbf S$ has an explicit formula as a matrix-valued integral involving $\mathbf G(t).$ We will be using this explicit formula throughout this thesis.

\begin{theorem}[OU Process Stationary Covariance Matrix Explicit Formula]\label{thm: OU Process Stationary Covariance Matrix Explicit Formula}
The stationary covariance matrix $\mathbf S$ is explicitly
    \begin{align}\label{eq: OU Process Stationary Covariance Matrix}
        \mathbf S &= 2 \int_{0}^\infty \mathbf G(t) \  \mathbf D \ \mathbf G^\text T(t) \ dt.
    \end{align}
Moreover, $\mathbf S$ is positive semidefinite.
\end{theorem}

\begin{proof}
First, we verify that \eqref{eq: OU Process Stationary Covariance Matrix} is indeed a solution to the Lyapunov equation \eqref{eq: OU Process Stationary Covariance Matrix Lyapunov Equation}. Substituting the right hand side of \eqref{eq: OU Process Stationary Covariance Matrix} for $\mathbf S$ into the left hand side of the equation \eqref{eq: OU Process Stationary Covariance Matrix Lyapunov Equation} and utilizing the property $\mathbf G'=-\mathbf B  \mathbf G$ from \textbf{Theorem \ref{thm: Matrix Exponential Properties}}, we have
    \begin{align}
         \mathbf B \mathbf S + \mathbf S \mathbf B^\text T   &= 2 \mathbf B \left(\int_{0}^\infty \mathbf G(t) \  \mathbf D \ \mathbf G^\text T (t) \ dt\right) +  2 \left(\int_{0}^\infty \mathbf G(t) \  \mathbf D \  \mathbf G^\text T (t) \ dt\right) \mathbf B^ \text T \nonumber \\
         &= 2 \int_{0}^\infty \mathbf B \ \mathbf G(t) \  \mathbf D \ \mathbf G^ \text T(t) \ dt + 2 \int_{0}^\infty  \mathbf G(t) \  \mathbf D \  \left(\mathbf B \ \mathbf G(t) \right)^ \text T dt \nonumber  \\
         &= -2 \int_{0}^\infty \mathbf G'(t) \  \mathbf D \  \mathbf G^ \text T (t) \  dt - 2 \int_{0}^\infty \mathbf G(t) \  \mathbf D \  \left(\mathbf G'(t) \right)^ \text T dt \nonumber  \\
         &= -2 \int_0^\infty \left(\mathbf G(t) \ \mathbf D \ \left(\mathbf G^\text T(t) \right) \right)'  \ dt \nonumber \\
         &= \lim_{t \rightarrow \infty} -2 \  \mathbf G(t) \ \mathbf D \ \mathbf G^\text T(t) + 2 \ \mathbf G(0) \ \mathbf D \ \mathbf G^\text T(0) \nonumber \\
         &=\mathbf 0 + 2 \mathbf D = 2 \mathbf D \nonumber .
    \end{align}
Thus, the right hand side of \eqref{eq: OU Process Stationary Covariance Matrix} is a solution of \eqref{eq: OU Process Stationary Covariance Matrix Lyapunov Equation}.

Next, we verify our solution to the Lyapunov equation \eqref{eq: OU Process Stationary Covariance Matrix Lyapunov Equation} is unique. To this end, we let $\mathbf S_1$ and $\mathbf S_2$ be two solutions to \eqref{eq: OU Process Stationary Covariance Matrix Lyapunov Equation}. Via the product rule, we obtain
    \begin{align}
        \left(\mathbf G(t) \ (\mathbf S_1 - \mathbf S_2) \ \mathbf G^\text T(t) \right)' &=  - \mathbf G(t) \ \mathbf B \  (\mathbf S_1 - \mathbf S_2) \ \mathbf G^\text T(t) \nonumber \\
        &\qquad \qquad - \mathbf G(t) \ (\mathbf S_1 - \mathbf S_2) \  \mathbf B^\text T \  \mathbf G^\text T(t) \nonumber \\
        &= - \mathbf G(t) \ \left(\mathbf B \mathbf S_1 + \mathbf S_1 \mathbf B^\text T - \mathbf B \mathbf S_2 - \mathbf S_2 \mathbf B^\text T\right) \ \mathbf G^\text T(t) \nonumber \\
        &= - \mathbf G(t) \ \left(2 \mathbf D - 2 \mathbf D \right) \ \mathbf G^\text T(t) = \mathbf 0 \nonumber,
    \end{align}
which implies $\mathbf G(t) \ (\mathbf S_1 - \mathbf S_2) \ \mathbf G^\text T(t)$ must be a constant function of $t.$ Since $-\mathbf B$ is Hurwitz, we have
    \begin{align}
        \lim_{t \rightarrow \infty} \mathbf G(t) \ (\mathbf S_1 - \mathbf S_2) \ \mathbf G^\text T(t) &= \mathbf 0 \nonumber, 
    \end{align}
which forces 
    \begin{align}
        \mathbf G(t) \ (\mathbf S_1 - \mathbf S_2) \ \mathbf G^\text T(t) &= \mathbf 0
    \end{align}
for all $t \ge 0.$ Since $\mathbf G(t)$ and $\mathbf G^\text T(t)$ are invertible, we left multiply both sides of the equation by $\mathbf G(-t)$ and right multiply both sides of the equation by $\mathbf G^\text T(-t)$ to obtain the equality $\mathbf S_1-\mathbf S_2= \mathbf 0.$ This shows $\mathbf S_1=\mathbf S_2,$ which means the Lyapunov equation \eqref{eq: OU Process Stationary Covariance Matrix Lyapunov Equation} has only one solution.

Finally, we show $\mathbf S$ is positive semidefinite. Note that 
    \begin{align}
        \mathbf B \mathbf S^\text T + \mathbf S^\text T \mathbf B^\text T &= \left( \mathbf B \mathbf S + \mathbf S \mathbf B^\text T \right)^\text T \nonumber  \\
        &= 2 \mathbf D^\text T = 2 \mathbf D,
    \end{align}
which implies $\mathbf S^\text T$ solves the Lyapunov equation \eqref{eq: OU Process Stationary Covariance Matrix Lyapunov Equation}. But by uniqueness, we have $\mathbf S= \mathbf S^\text T,$ which means $\mathbf S$ is symmetric. Now, for all nonzero $\mathbf y \in \R^N,$ we have 
    \begin{align}
        \mathbf y^\text T   \mathbf S 
  \mathbf y &= 2 \int_0^\infty \mathbf y^\text T \  \mathbf G(t) \ \mathbf D \ \mathbf G^\text T(t) \ \mathbf y \ dt \nonumber  \\
 &= \int_0^\infty \mathbf y^\text T \  \mathbf G(t) \ \boldsymbol \Sigma \boldsymbol \Sigma^\text T \ \mathbf G^\text T(t) \ \mathbf y \ dt \nonumber  \\
 &= \int_0^\infty \left \| \boldsymbol \Sigma^\text T \ \mathbf G^\text T (t)  \ \mathbf y \right \|_2^2 \  dt \nonumber,
    \end{align}
which is non-negative because the integrand $\left \| \boldsymbol \Sigma^\text T \ \mathbf G^\text T (t)  \ \mathbf y \right \|_2^2$ is non-negative for all $t \ge 0.$ Hence, $\mathbf S$ is positive semidefinite.
\end{proof}

\section{Statistical Properties}
In this section, we will recall the important statistical properties of the OU Process. We first derive the explicit representation of the OU process i.e. the solution to the governing stochastic differential equation \eqref{eq: OU Process SDE}. The solution also reveals the essential statistics of the OU process. 

\begin{theorem}[OU Process Explicit Representation]\label{thm: OU Process Explicit Representation Theorem}
The unique solution of the governing OU process stochastic differential equation \eqref{eq: OU Process SDE} is explicitly
    \begin{align} \label{eq: OU Process Explicit Representation}
        \mathbf x(t) &= \mathbf G(t) \ \mathbf x(0) + \int_0^t \mathbf G(t-s) \ \boldsymbol \Sigma \ d \mathbf w(s).
    \end{align}
Moreover, the OU process is an Ito diffusion process that is square integrable.
\end{theorem}

\begin{proof}
First, we show that the right hand side of \eqref{eq: OU Process Explicit Representation} solves the stochastic differential equation  \eqref{eq: OU Process SDE}. Let 
	\begin{align}
		\mathbf y(t) &= \mathbf G(-t)  \ \mathbf x(t) \nonumber.
	\end{align}
Using Ito's product rule and the Ito formalisms \cite[Theorem 4.2.1]{Oksendal2003}, namely
\begin{align*}
	(dt)^2 =0 \ , \ d \mathbf w(t) \  dt =\mathbf 0,
\end{align*}
we obtain
    \begin{align}
        d \mathbf y(t) &= d \left(\mathbf G(-t) \  \mathbf x(t) \right) \nonumber \\
        &=  d\mathbf G(-t) \ \mathbf x(t) +  \mathbf G(-t) \ d \mathbf x(t) + d \mathbf G(-t) \ d \mathbf x(t) \nonumber  \\
        &=\mathbf B  \mathbf G(-t) \  \mathbf x(t) \ dt + \mathbf G(-t)  \left(- \mathbf B \ \mathbf x(t) \ dt + \boldsymbol \Sigma \ d \mathbf w(t) \right) \nonumber
        \\ & \qquad + \left(\mathbf B  \mathbf G(-t) \ dt \right) (- \mathbf B \ \mathbf x(t) \ dt + \boldsymbol \Sigma \ d \mathbf w(t)) \nonumber \\
        &=  \mathbf G(-t) \   \boldsymbol \Sigma \ d \mathbf w(t)  + \left(\mathbf B  \mathbf G(-t) \ dt \right) (- \mathbf B \ \mathbf x(t) \ dt + \boldsymbol \Sigma \ d \mathbf w(t)) \nonumber  \\
        &= \mathbf G(-t) \   \boldsymbol \Sigma \ d \mathbf w(t)  \label{eq: OU Process Transformed SDE},
    \end{align}
Formally, \eqref{eq: OU Process Transformed SDE} is the integral equation
    \begin{align}
        \mathbf y(t) - \mathbf y(0) &= \int_0^t \mathbf G(-s)  \  \boldsymbol \Sigma \ d \mathbf w(s).
    \end{align}
Upon recognizing that $\mathbf y(0)=\mathbf x(0)$ and recalling that $\mathbf G(-t)$ is invertible with inverse $\mathbf G(t)$ by \textbf{Theorem \ref{thm: Matrix Exponential Properties}},
we have 
    \begin{align}
        \mathbf x(t) &= \mathbf G(t) \ \mathbf y(t) \nonumber \\
        &= \mathbf G(t) \left(\mathbf y(0) + \int_0^t \mathbf G(-s) \    \boldsymbol \Sigma \ d \mathbf w(s)  \right) \nonumber \\
        &= \mathbf G(t) \left(\mathbf x(0) + \int_0^t \mathbf G(-s) \    \boldsymbol \Sigma \ d \mathbf w(s)  \right) \nonumber \\
        &= \mathbf G(t) \  \mathbf x(0) + \int_0^t \mathbf G(t-s) \   \boldsymbol \Sigma \ d \mathbf w(s).  \nonumber 
    \end{align}
Thus, the right side of \eqref{eq: OU Process Explicit Representation} is a solution to the governing stochastic differential equation \eqref{eq: OU Process SDE}.

Next, we show that \eqref{eq: OU Process SDE} has a unique solution that is an Ito diffusion process that is square integrable. We need to verify the existence and uniqueness criteria listed in \textbf{Theorem \ref{thm: SDE Solution Existence and Uniqueness}} for the OU process. Define  $\mathbf b:  [0,\infty) \times \R^N \rightarrow \R^N$ and $\boldsymbol \sigma: [0,\infty) \times \R^N \rightarrow \Mat_{N,M}(\R)$ by
	\begin{align}
		\mathbf b(t,\mathbf y) &= - \mathbf B \mathbf y \\
		\boldsymbol \sigma(t,\mathbf y) &= \boldsymbol \Sigma.
	\end{align}
Since $\mathbf b(t, \cdot)$ is a linear operator that is independent of the variable $t,$ it is continuous and is hence measurable. Since $\boldsymbol \sigma$ is constant function in both $\mathbf y$ and $t,$ it is automatically measurable. Furthermore, $\mathbf b$ is bounded; by the definition of boundedness, there exists a constant $D$ directly satisfying \eqref{eq: SDE Existence Uniquenness Condition 2}. Now, let $C=\max(D,\|\boldsymbol \Sigma\|_F).$ Then
\begin{align}
	\|\mathbf b(\mathbf y,t) \| + \| \boldsymbol \sigma(\mathbf y,t) \|_F &= \|\mathbf b(\mathbf y,t) \| + \|\boldsymbol \Sigma\|_F \nonumber \\
	&\le D \| \mathbf y \| + \|\boldsymbol \Sigma\|_F \nonumber \\
	&\le C\| \mathbf y \| + C = C(1 + \| \mathbf y\|). \nonumber 
\end{align}
Hence, our choice of $C$ satisfies \eqref{eq: SDE Existence Uniquenness Condition 1}.  
Moreover, recall by assumption that $\mathbf x(0) \sim \mathcal N(\mathbf 0, \mathbf S)$ is independent of the Wiener process, and $\E[\| \mathbf x(0) \|^2]$ exists by virtue of $\mathbf x(0)$ having a finite covariance matrix. Hence, the SDE  \eqref{eq: OU Process SDE} satisfies the assumptions listed in \textbf{Theorem \ref{thm: SDE Solution Existence and Uniqueness}}. This completes the proof.
\end{proof}

Using the solution of the OU process, we derive its important statistical properties.

\begin{theorem}[OU Process is Gaussian]\label{thm: OU Process is Gaussian}
The OU process is a Gaussian process.
\end{theorem}

\begin{proof}
Consider the stochastic processes $\left \lbrace \mathbf y(t) \right \rbrace_{t \ge 0}$ and $\left \lbrace \mathbf z(t) \right \rbrace_{t \ge 0},$ where
    \begin{align}
        \mathbf y(t) &= \mathbf G(t) \ \mathbf x(0) \nonumber \\
        \mathbf z(t) &= \mathbf G(t) \  \int_0^t \mathbf G(-s) \ \boldsymbol \Sigma \  d \mathbf w(s) \nonumber 
    \end{align}
for each $t \ge 0.$ Note that every finite linear combination of random vectors from the OU process is the sum of a finite linear combination of random vectors from $\left \lbrace \mathbf y(t) \right \rbrace_{t \ge 0}$ and a finite linear combination of random vectors from $\left \lbrace \mathbf z(t) \right \rbrace_{t \ge 0}.$ Furthermore, any linear combination of random vectors from $\left \lbrace \mathbf y(t) \right \rbrace_{t \ge 0}$ is independent of any linear combination of random vectors from $\left \lbrace \mathbf z(t) \right \rbrace_{t \ge 0}$ since $\mathbf x(0)$ is independent of the Wiener process.

We show that the process $\left \lbrace \mathbf y(t) \right \rbrace_{t \ge 0}$ is Gaussian. Every finite linear combination of random vectors from $\left \lbrace \mathbf y(t) \right \rbrace_{t \ge 0}$ is of the form
    \begin{align} \label{eq: OU Process Deterministic Component Linear Combination}
        \sum_{k=1}^K \alpha_k \   \mathbf y(t_k) &= \left(\sum_{k=1}^K \alpha_k \ \mathbf G(t_k) \right) \ \mathbf x(0)
    \end{align}
for some $K \in \N$ and times $t_1 \ , \ \dots \ ,\ t_K \ge 0$ and constants $\alpha_1 \ , \ \dots \ , \ \alpha_K \in \R.$ Observe that the right hand side of \eqref{eq: OU Process Deterministic Component Linear Combination} is a deterministic linear transformation of the Gaussian random vector $\mathbf x(0),$ which means it is also Gaussian. Hence, $\left \lbrace \mathbf y(t) \right \rbrace_{t \ge 0}$ is a Gaussian process.

Next, we show $\left \lbrace \mathbf z(t) \right \rbrace_{t \ge 0}$ is a Gaussian process. Consider the map $(s,\omega) \mapsto \mathbf G(t-s)  \ \boldsymbol \Sigma$ defined between $[0,\infty) \times \Omega$ and $\Mat_{N,M}(\R).$ Since this map does not depend on the variable $\omega,$ it is a deterministic matrix-valued function whose $NM$ entry functions are deterministic real-valued functions. These entry functions are automatically adapted to the standard filtration of the Wiener process and are continuous in $s,$ which means they are square integrable over any closed, finite interval. Hence, by \cite[Theorem 3.2.1]{Oksendal2003} , the process $\left \lbrace \mathbf z(t) \right \rbrace_{t \ge 0}$ is a Gaussian process.


Finally, every finite linear combination of random vectors from the OU process is the sum of two independent Gaussian random vectors, one of which is a linear combination of random vectors from the process $\left \lbrace \mathbf y(t) \right \rbrace_{t \ge 0},$ while the other is a linear combination of random vectors from the process $\left \lbrace \mathbf z(t) \right \rbrace_{t \ge 0}.$ Therefore, every linear combination of random vectors from the OU process is  a Gaussian random vector. This proves the claim.

\end{proof}

\begin{theorem}[OU Process Statistics]\label{thm: OU Process Statistics}
The OU process is a Gaussian process whose mean, auto-covariance functions are
    \begin{align} \label{eq: OU Process Autocovariance Function}
    	\boldsymbol \mu_{\mathbf x}(t) &= \mathbf 0\\
        \boldsymbol \Gamma_{\mathbf x}(s,t) &= \mathbf G(s) \  \mathbf S \ \mathbf G^\text T(t) + 2 \int_0^{\min(s,t)} \mathbf G(s-u) \ \mathbf D \ \mathbf G^\text T(t-u)\ du,
    \end{align}
respectively.
\end{theorem}

\begin{proof}
First, we determine the mean function of the OU process. Let $\left \lbrace \mathbf y(t) \right \rbrace_{t \ge 0}$ and $\left \lbrace \mathbf z(t) \right \rbrace_{t \ge 0}$ be the processes as defined in the proof of the previous theorem. Since $\mathbf x(t)= \mathbf y(t) + \mathbf z(t)$ for each $t \ge 0,$ we use the linearity of the expectation operator to obtain
	\begin{align}
		\boldsymbol \mu_{\mathbf x}(t) &= \boldsymbol \mu_{\mathbf y}(t) + \boldsymbol \mu_{\mathbf z}(t) \nonumber \\
		&=\mathbf G(t) \ \E[\mathbf x(0)] + \E \left[ \int_0^t \mathbf G(t-s) \ \boldsymbol \Sigma \ d \mathbf w(s) \right] \label{eq: OU Process Mean Function 1} \\
		&= \mathbf 0 + \mathbf 0 =\mathbf 0 \nonumber,
	\end{align} 
in which we recognized the expectation of the stochastic integral term in \eqref{eq: OU Process Mean Function 1} vanishes due to the process $\left \lbrace \mathbf z(t) \right \rbrace_{t \ge 0}$ being a martingale.

Next, we calculate the autocovariance function of the OU process.  Using the definition \eqref{eq: Stochastic Process Autocovariance Function} and the bilinearity of the covariance operator, we have
	\begin{align}
		\boldsymbol \Gamma_\mathbf x(s,t) &= \Cov \left(\mathbf x(s) , \mathbf x(t) \right) \nonumber  \\
		&= \Cov \left(\mathbf y(s)  + \mathbf z(s) ,\mathbf y(t)  + \mathbf z(t) \right ) \label{eq: OU Process Autocovariance Function 1} \\
		&= \Cov \left(\mathbf y(s)  , \mathbf y(t)  \right) + \Cov \left( \mathbf z(s), \mathbf z(t) \right ) \label{} \label{eq: OU Process Autocovariance Function 2} \\
		&= \boldsymbol \Gamma_{\mathbf y}(s,t) +  \boldsymbol \Gamma_{\mathbf z}(s,t) \nonumber \\
		&= \mathbf G(s) \ \Cov(\mathbf x(0) , \mathbf x(0)) \ \mathbf G^\text T(t) + \boldsymbol \Gamma_{\mathbf z}(s,t) \nonumber \\
		&= \mathbf G(s) \ \mathbf S \ \mathbf G^\text T(t) + \boldsymbol \Gamma_{\mathbf z}(s,t) \nonumber,
	\end{align}
in which \eqref{eq: OU Process Autocovariance Function 2} is the result of expanding out the covariance expression in  \eqref{eq: OU Process Autocovariance Function 1} using bilinearity and the fact $\mathbf y(s)$ and $\mathbf z(t)$ are independent and hence uncorrelated. We now determine $\boldsymbol \Gamma_{\mathbf z}(s,t).$ For all $s \le t,$ we use the additivity of the stochastic integral and the bilinearity of the covariance operator to obtain   
    \begin{align}
	   \boldsymbol \Gamma_{\mathbf z}(s,t) &=  \Cov \left( \int_0^s \mathbf G(s-u) \   \boldsymbol \Sigma \ d \mathbf w(u) \ , \  \int_0^t \mathbf G(t-v) \   \boldsymbol \Sigma \ d \mathbf w(v) \right) \nonumber  \\
	        &= \Cov \left( \int_0^s \mathbf G(s-u) \   \boldsymbol \Sigma \ d \mathbf w(u) \ , \  \int_0^s \mathbf G(t-v) \   \boldsymbol \Sigma \ d \mathbf w(v)  \right)  \label{eq: OU Process Stochastic Part Autocovariance Function 1} \\
	        &\qquad + \Cov \left ( \int_0^s \mathbf G(s-u) \   \boldsymbol \Sigma \ d \mathbf w(u),  \int_s^t \mathbf G(t-v) \   \boldsymbol \Sigma \ d \mathbf w(v) \right) \nonumber  \\
	       &=  \Cov \left( \int_0^s \mathbf G(s-u) \   \boldsymbol \Sigma \ d \mathbf w(u) \ , \  \int_0^s \mathbf G(t-v) \   \boldsymbol \Sigma \ d \mathbf w(v)  \right) \label{eq: OU Process Stochastic Part Autocovariance Function 2} \\
	       &= \int_0^s \mathbf G(s-u) \ \boldsymbol \Sigma \ \left(\mathbf G(t-u) \ \boldsymbol \Sigma \right)^\text T \ du \label{eq: OU Process Stochastic Part Autocovariance Function 3}\\
	       &=2  \int_0^s \mathbf G(s-u) \ \mathbf D \ \mathbf G^\text T(t-u) \ du \nonumber,
	    \end{align}
in which the second covariance term in \eqref{eq: OU Process Stochastic Part Autocovariance Function 1} vanishes due to the Wiener process having independent increments, and \eqref{eq: OU Process Stochastic Part Autocovariance Function 3} follows from evaluating the covariance term in \eqref{eq: OU Process Stochastic Part Autocovariance Function 2} with Ito's Isometry \cite[Corollary 3.1.7]{Oksendal2003}. Similar reasoning gives 
	\begin{align}
		\boldsymbol \Gamma_\mathbf z(s,t) &= 2  \int_0^t \mathbf G(s-u) \ \mathbf D \ \mathbf G^\text T(t-u) \ du \nonumber
	\end{align}
for $s \ge t.$ Upon combining the results corresponding to the cases $s \le t$ and $s \ge t,$  we have
	\begin{align}
		\boldsymbol \Gamma_\mathbf z(s,t) &= 2  \int_0^{\min(s,t)}\mathbf G(s-u) \ \mathbf D \ \mathbf G^\text T(t-u) \ du \nonumber
	\end{align}
for all $s,t \ge 0.$  Hence, we obtain the OU process autocovariance function.
\end{proof}

Based on the time-dependent statistics, we can deduce the OU process is stationary.
\begin{theorem}[OU Process Stationary]\label{thm: OU Process Stationary}
The OU process is stationary, and its covariance function is 
	\begin{align} \label{eq: OU Process Covariance Function}
		\mathbf V_\mathbf x(t) &= \mathbf S
	\end{align}
for all $t \ge 0.$
\end{theorem}

\begin{proof}
Since the OU Process is a Gaussian process, it suffices to show the process is weakly stationary. In particular, we show the map $(t,h) \mapsto \boldsymbol \Gamma_{\mathbf x}(t+h,t)$ defined between $[0,\infty)^2$ and $\Mat_{N,N}(\R)$ depends only on $h.$ Differentiating the map $\boldsymbol \Gamma_{\mathbf x}(t+h,t)$ with respect to $t,$ we obtain
	\begin{align}
		\frac{\partial  \ \boldsymbol \Gamma_\mathbf x(t+h,t)}{\partial t} &= \frac{\partial}{\partial t} \left(\mathbf G(t+h) \ \mathbf S \ \mathbf G^\text T(t) + 2 \int_0^t \mathbf G(t+h-u) \ \mathbf D \ \mathbf G^\text T(t-u) \ du \right) \label{eq: OU Covariance Function Partial Derivative 1}\\
		&= - \mathbf B \  \mathbf G(t+h)  \ \mathbf S \  \mathbf G^\text T(t) - \mathbf G(t+h) \ \mathbf S \ \left(\mathbf B \ \mathbf G(t) \right)^\text T \nonumber \\
		&\qquad \qquad + 2\  \frac{\partial}{\partial t} \left( \int_0^t \mathbf G(u+h) \ \mathbf D \ \mathbf G^\text T(u) \ du \right)  \label{eq: OU Covariance Function Partial Derivative 2} \\
		&= \mathbf G(t+h)  \ (-\mathbf B \mathbf S - \mathbf S \mathbf B^\text T)  \ \mathbf G^\text T(t)  + \mathbf G(t+h) \  (2 \mathbf D) \ \mathbf G^\text T(t)  \label{eq: OU Covariance Function Partial Derivative 3} \\
		 &= \mathbf G(t+h) \ (-\mathbf B \mathbf S + \mathbf S \mathbf B^\text T + 2 \mathbf D ) \ \mathbf G^\text T(t) = \mathbf 0 \label{eq: OU Covariance Function Partial Derivative 4},
	\end{align}
where the second term of \eqref{eq: OU Covariance Function Partial Derivative 2} follows from applying the change of variables $u \mapsto t-u$ on the integral term within the partial derivative expression of \eqref{eq: OU Covariance Function Partial Derivative 1}, and the first term of \eqref{eq: OU Covariance Function Partial Derivative 3} follows from rewriting the first term of  \eqref{eq: OU Covariance Function Partial Derivative 2}  using the the fact that $\mathbf B$ and $\mathbf G(t)$ commute according to \textbf{Theorem \ref{thm: Matrix Exponential Properties}}. Finally, \eqref{eq: OU Covariance Function Partial Derivative 4}  follows from the Lyapunov equation  \eqref{eq: OU Process Stationary Covariance Matrix Lyapunov Equation}.
Since the partial derivative of $\boldsymbol \Gamma_{\mathbf x}(t+h,t)$ vanishes on the connected set $[0,\infty)^2,$ we conclude that $\boldsymbol \Gamma_{\mathbf x}(t+h,t)$ depends only on $h.$ 

Next, for each $t \ge 0,$ we have
	\begin{align}
		\mathbf V_\mathbf x(t) &= \boldsymbol \Gamma_\mathbf x(t,t) \nonumber \\ &=  \boldsymbol \Gamma_\mathbf x(0,0) \nonumber \\
		&= \mathbf V_\mathbf x(0) = \mathbf S \nonumber.
	\end{align}
\end{proof}
The previous theorem implies $\mathbf x(t) \sim \mathcal N(\mathbf 0, \mathbf S)$ for each $t \ge 0,$ which trivially implies the stationary distribution of the OU process is the Gaussian distribution with mean $\mathbf 0$ and covariance matrix $\mathbf S.$

We provide a sufficient condition for the OU process to be ergodic. 
\begin{theorem}[OU Process Ergodicity]\label{thm: OU Process Ergodicity}
If the diffusion matrix $\mathbf D$ is invertible, then the OU process is ergodic.
\end{theorem}

\begin{proof}
In order to prove ergodicity, we need to show that the stationary distribution of the OU process, namely the Gaussian distribution with mean $\mathbf 0$ and covariance matrix $\mathbf S,$ is unique \cite{Sandric2017}.

Consider a generic (time homogeneous) Ito diffusion process $\left \lbrace \mathbf y(t) \right \rbrace_{t \ge 0}$ satisfying the stochastic differential equation
\begin{align}\label{eq: General Diffusion Process}
	d \mathbf y(t) &= \mathbf b(\mathbf y(t)) \ dt + \boldsymbol \sigma(\mathbf y(t) ) \ d \mathbf w(t)
\end{align}
with $\mathbf b: \R^N  \rightarrow \R^N$ and $\boldsymbol \sigma: \R^N \rightarrow \Mat_{N,M}(\R)$ being Borel measurable functions. A sufficient condition \cite[Theorem 5.3.2]{CapassoBakstein2021} for this diffusion process to have a unique stationary distribution is that there is an open, bounded domain $U \subset \R^N$ with smooth boundary $\partial U$ such that 
\begin{enumerate}[label=\roman*)]
	\item For each $\mathbf x \in U,$ the smallest eigenvalue of the diffusion coefficient $\frac{\boldsymbol \sigma(\mathbf x) \ \boldsymbol \sigma^ \text T(\mathbf x)}{2}$ is bounded away from $0.$ 
	\item The process is recurrent relative to $U$: for each $\mathbf z \in \R^N \setminus U,$ the mean time $\tau$ at which a path issuing from $\mathbf z$ reaches the set $U$ is finite and $\sup_{\mathbf z \in K} \E[\tau | \mathbf z]$ exists for each compact set $K$ in $\R^N.$
\end{enumerate}
Moreover, in \cite[Theorem 3.9]{Khasminskii2012}, it was shown that the condition $\text{ii)}$ is equivalent to the \textit{Foster Lyapunov drift condition}, which states there is a non-negative, twice continuously differentiable function $V : \R^N \rightarrow \R$ for which $\mathcal{L}V$ is negative on $\R^N \setminus U.$ Here, $\mathcal L$ is the infinitesimal generator defined by
\begin{align}
	(\mathcal L f)(\mathbf z) &=   \left( \nabla f \right)^\text T \mathbf b(\mathbf z) + \text{tr} \left(\nabla^2 f \ \frac{\boldsymbol \sigma(\mathbf z) \boldsymbol \sigma^\text T(\mathbf z)}{2}\right).
\end{align}
for any twice continuously-differentiable $f: \R^N \rightarrow \R,$ where $\nabla f$ is the gradient of $f$ and $\nabla^2 f$ is the Hessian matrix of $f.$

First, we verify statement $\text{i)}$ holds for the OU process. Let
\begin{align} \nonumber
	U = \left \lbrace \mathbf z \in \R^N \ : \ \| \mathbf z\| < \sqrt{ 3\text{tr}(\mathbf P) \ \text{tr}(\mathbf D)} \right \rbrace,
\end{align}
where $\mathbf P \in \Mat_{N,N}(\R)$ is the unique positive definite solution to the Lyapunov equation
\begin{align}
	\mathbf B^\text T \mathbf P + \mathbf P \mathbf B &= \mathbf I \label{eq: Ergodic Lyapunov Equation},
\end{align}
where $\mathbf I$ is the $N \times N$ identity matrix. Since $\mathbf P$ and $\mathbf D$ are positive definite, their traces are positive. Hence, $U$ is nonempty. It is an open ball in $\R^N$ centered around $\mathbf 0.$ Note that the diffusion coefficient of the OU process is the constant matrix $\mathbf D.$ Since $\mathbf D$ is positive definite, its eigenvalues are all positive. So automatically, its smallest eigenvalue is bounded away from $0$ on $U.$ This means the OU process satisfies statement $\text i).$

Next, we verify the second statement $\text{ii)}$ by verifying the Foster Lyapunov drift condition holds. For the OU process, the infinitesimal generator is explicitly
    \begin{align}
    	(\mathcal L f)(\mathbf y) &= - (\nabla f)^\text T \ \mathbf B \mathbf y + \text{tr}(\nabla^2 f \  \mathbf D),
    \end{align}
for each twice-continuously differentiable $f: \R^N \rightarrow \R.$ Consider the Lyapunov function $V: \R^N \rightarrow \R$ defined by
    \begin{align}
    	V(\mathbf z)=\mathbf z^\text T \mathbf P \mathbf z. \nonumber 
    \end{align}
Then, we compute
    \begin{align}
    	\nabla V(\mathbf z) &= (\mathbf P + \mathbf P^ \text T) \ \mathbf z= 2 \mathbf P \mathbf z \nonumber \\
    	\nabla^2 V(\mathbf z) &= 2 \mathbf P \nonumber.
    \end{align}  
For all $\mathbf z \in \R^N \setminus U,$ by the Lyapunov equation \eqref{eq: Ergodic Lyapunov Equation}, we have
    \begin{align}
    	(\mathcal LV)(\mathbf z) &= -\left(2 \mathbf P  \mathbf z \right)^\text T\mathbf B \mathbf z + \text{tr}(2  \mathbf P \mathbf D ) \nonumber \\
    	&= -2\mathbf z^\text T  (\mathbf P \mathbf B) \    \mathbf z + 2 \text{tr}(  \mathbf P \mathbf D ) \nonumber  \\
    	&= -\mathbf z^\text T  (\mathbf P \mathbf B + \left(\mathbf P \mathbf B \right)^\text T) \    \mathbf z + 2 \text{tr}(  \mathbf P \mathbf D ) \nonumber  \\
    	&= -\mathbf z^\text T  (\mathbf P \mathbf B +  \mathbf B^\text T \mathbf P) \    \mathbf z + 2 \text{tr}(  \mathbf P \mathbf D )  \nonumber \\
    	&= -\| \mathbf z\|^2 + 2 \text{tr}(  \mathbf P \mathbf D )  \nonumber  \\
    	&\le  -3 \text{tr}(\mathbf P) \  \text{tr}(\mathbf D)+ 2 \text{tr}(\mathbf P) \  \text{tr}(\mathbf D) \label{eq: Trace Triangle Inequality}\\
    	&=-\text{tr}(\mathbf P) \  \text{tr}(\mathbf D)  \nonumber <0,
    \end{align}
in which the second term of \eqref{eq: Trace Triangle Inequality} follows from the Cauchy Schwarz inequality applied to the trace of the product of two positive semidefinite matrices. Hence, the OU process satisfies statement $\text{ii)},$ and we conclude the process is ergodic.
\end{proof}

\section{The Lead Process}
In this section, we consider the OU process generated by model parameters $\mathbf B$ and $\boldsymbol \Sigma,$ and we assume the process is ergodic. We recall the auxiliary \textit{lead process} that we defined in the introduction, namely the $N \times N$ matrix-valued stochastic process $\left \lbrace \mathbf A(t) \right \rbrace_{t \ge 0},$ in which 
    \begin{align} \label{eq: OU Lead Lag Process}
        \mathbf A(t) &= \frac{1}{2} \int_0^t \mathbf x(s) \ d \mathbf x^\text T(s) - d \mathbf x(s)   \ \mathbf x^\text T(s).
    \end{align}
for each $t \ge 0.$ More explicitly, for each $1 \le m,n \le N,$ the $(m,n)$-th entry of $\mathbf A(t)$ is an Ito stochastic integral of the form
	\begin{align}\label{eq: OU Lead Lag Component Process}
		A_{m,n}(t) &= \frac{1}{2} \int_0^t x_m(s) \ dx_n(s) - x_n(s) \ dx_m(s),
	\end{align}
where $\left \lbrace x_n(t) \right \rbrace_{t \ge 0}$ is the $n$-th component OU process. By definition, each $\mathbf A(t)$ is the lead matrix corresponding to a realization of the OU process over the finite time interval $[0,t].$ 

We prove a strong law of large numbers identity pertaining to the lead process.
\begin{theorem}[Lead Process Strong Law of Large Numbers Identity] \label{thm: Lead Process Law of Large Numbers Identity}
The lead process satisfies the following strong law of large numbers identity: 
    \begin{align} \label{eq: Lead Process Limit}
        \lim_{t \rightarrow \infty} \frac{\mathbf A(t)}{t} &= \frac{\mathbf B \mathbf S- \mathbf S \mathbf B^\text T}{2} 
    \end{align}
almost surely.
\end{theorem}

\begin{proof}
Using the governing stochastic differential equation \eqref{eq: OU Process SDE}, we write
	\begin{align}
		\frac{1}{2t} \int_0^t \mathbf x(s) \ d \mathbf x^\text T(s)  &= \frac{1}{2t} \int_0^t \mathbf x(s) \ \left(- \mathbf B \  \mathbf x(s) \ ds + \boldsymbol \Sigma \ d \mathbf w(s) \right)^\text T \nonumber \\
  &= - \frac{1}{2t} \int_0^t \mathbf x(s) \ \mathbf x^\text T(s) \ ds \ \mathbf B^\text T + \frac{1}{2t} \int_0^t \mathbf x(s) \ d \mathbf w^\text T(s) \ \boldsymbol \Sigma^\text T \label{eq: Lead Process Limit 1}.
	\end{align}
We need to determine the almost sure limits of both integral terms in \eqref{eq: Lead Process Limit 1} as $t \rightarrow \infty.$

First, we show that the almost sure limit of the first term of \eqref{eq: Lead Process Limit 1} is $-\frac{\mathbf S \mathbf B^\text T}{2}$ as $t \rightarrow \infty.$ To this end, consider the map $f: \R^N \rightarrow \Mat_{N,N}(\R)$ defined by 
	\begin{align}
		f(\mathbf y) &= \mathbf y  \mathbf y^\text T.
	\end{align}
This map has $N^2$ real-valued entry functions of the form $f_{m,n}: \R^N \rightarrow \R$ defined by $f_{m,n}(\mathbf y)=y_my_n$ for $1 \le m,n \le N.$ Each $f_{m,n}$ is continuous and hence measurable between the spaces $\R^N$ and $\R$ equipped with their respective Borel sigma-algebras. Hence, $f$ itself is measurable and by \textbf{Theorem \ref{thm: Birkhoff's Ergodic Theorem}}, we obtain
	\begin{align}
		\lim_{t \rightarrow \infty} \frac{1}{2t} \int_0^t \mathbf x(s) \ \mathbf x^\text T(s) \ ds \ \mathbf B^\text T &= \lim_{t \rightarrow \infty} \frac{1}{2t} \int_0^t f(\mathbf x(s)) \ ds \ \mathbf B^\text T \nonumber \\
		&=\frac{\E[f(\mathbf x(0))] \ \mathbf B^\text T}2 \nonumber  \\
		&= \frac{\mathbf S \mathbf B^\text T}2 \nonumber 
	\end{align}
almost surely. 

Next, we show that the almost sure limit of the second term of \eqref{eq: Lead Process Limit 2} is $\mathbf 0.$ To this end, consider the $N \times M$ matrix-valued stochastic process $\left \lbrace \mathbf R(t) \right \rbrace_{t \ge 0}$ defined by 
    \begin{align}
        \mathbf R(t) &= \int_0^t \mathbf x(s) \ d \mathbf w^\text T(s) \nonumber.
    \end{align}
for each $t \ge 0.$ We need to show that 
    \begin{align}
        \lim_{t \rightarrow \infty} \frac{\mathbf R(t)}{t}  &= \mathbf 0 \nonumber
    \end{align}
almost surely. Equivalently, we need to show each entry process of $\left \lbrace \mathbf R(t) \right \rbrace_{t \ge 0},$ namely the process $\left \lbrace R_{n,m}(t) \right \rbrace_{t \ge 0},$ in which 
	\begin{align}
		R_{n,m}(t) &= \int_0^t x_n(s) \ d w_m(s) \nonumber
	\end{align}
for each $t \ge 0,$ satisfies
	\begin{align}\label{eq: Lead Process Limit 2}
		\lim_{t \rightarrow \infty} \frac{R_{n,m}(t)}{t} &=0 
	\end{align}
almost surely. Note that $\left \lbrace R_{n,m}(t) \right \rbrace_{t \ge 0}$ is a continuous time square integrable martingale with respect to the standard filtration of the Wiener process because the component OU process $\left \lbrace x_n(t) \right \rbrace_{t \ge 0}$ is square integrable and is adapted to the component Wiener process $\left \lbrace w_m(t) \right \rbrace_{t \ge 0}.$  By Ito's Isometry, for each $t \ge 0,$ we have
    \begin{align}
    	\E \left [ R^2_{n,m}(t)  \right] &= \E \left [ \int_0^t x^2_n(s) \ ds \right] \nonumber \\
    	&= \int_0^t  \E \left [ x^2_n(s) \right] \ ds \nonumber \\
    	&= t S_{n,n}. \nonumber 
    \end{align}
Moreover, the discrete process $\left \lbrace R_{n,m}(K) \right \rbrace_{K \in \N \cup \left \lbrace 0 \right \rbrace}$ is a discrete time square integrable martingale.

We show that the discrete time martingale  $\left \lbrace R_{n,m}(K) \right \rbrace_{K \in \N \cup \left \lbrace 0 \right \rbrace}$ satisfies the strong law of large numbers identity
	\begin{align}\label{eq: Lead Process Limit 3}
		\lim_{K \rightarrow \infty} \frac{R_{n,m}(K)}{K} &=0
	\end{align}
almost surely, in which $K$ runs over the integers. We consider the martingale difference sequence $\left \lbrace Y_{n,m}(K) \right \rbrace_{K \in \N},$ where
	\begin{align}
		Y_{n,m}(K) &=  R_{n,m}(K) - R_{n,m}(K-1) \nonumber
	\end{align}
for each $K \in \N.$ Because $\left \lbrace R_{n,m}(K) \right \rbrace_{K \in \N \cup \lbrace 0 \rbrace}$ is a martingale, we have
	\begin{align}
		\E \left [R_{n,m}(K) \ R_{n,m}(K-1) \right ] &= \E \left [ R_{n,m}(K-1) \ \E \left[ R_{n,m}(K) \ | \ \mathcal S_{K-1} \right] \ \right] \nonumber \\
  &= \E \left [ R^2_{n,m}(K) \right] \nonumber,
	\end{align}
where $\left \lbrace \mathcal S_t \right \rbrace_{t \ge 0}$ is the standard filtration of the Wiener process. Therefore, 
	\begin{align}
		\sum_{k=1}^K \frac{\E \left [ (Y_{n,m}(k)) ^2 \right ]}{k^2} &= \sum_{k=1}^K \frac{\E \left [\left (R_{n,m}(k) - R_{n,m}(k-1) \right )^2 \right]}{k^2} \nonumber  \\
		&=  \sum_{k=1}^K \frac{\E\left[ (R_{n,m}(k))^2 \right] - 2 \E\left[ R_{n,m}(k) R_{n,m}(k-1) \right]+ \E\left[ (R_{n,m}(k-1))^2 \right]}{k^2} \nonumber \\
		&= \sum_{k=1}^K \frac{\E\left[ (R_{n,m}(k))^2 \right] - 2 \E\left[ (R_{n,m}(k-1))^2 \right]+ \E\left[ (R_{n,m}(k-1))^2 \right]}{k^2} \nonumber  \\
		&= \sum_{k=1}^K \frac{\E\left[ (R_{n,m}(k))^2 \right] - \E\left[ (R_{n,m}(k-1))^2 \right]}{k^2} \nonumber   \\
		&=  \sum_{k=1}^K \frac{k S_{n,n} - (k-1)S_{n,n}}{k^2} \nonumber \\
		&=  \sum_{k=1}^K \frac{S_{n,n}}{k^2}, \nonumber
	\end{align}
which converges as $K \rightarrow \infty$. By \cite[Exercise 4.4.10]{Durrett2019}, we deduce $\frac{R_{n,m}(K)}{K}$ converges almost surely to some limit as $K \rightarrow \infty.$ Because the sequence $\left \lbrace k \right \rbrace_{k \in \N}$ is an increasing sequence diverging to infinity, we deduce by Kronecker's Lemma \cite[Theorem 2.5.9]{Durrett2019} that 
	\begin{align}
		\lim_{K \rightarrow \infty} \frac{R_{n,m}(K)}K  &=0 \nonumber
	\end{align}
almost surely, which proves the identity \eqref{eq: Lead Process Limit 3}. Now, by Doob's submartingale inequality \cite[Theorem 4.4.2]{Durrett2019}, we have
	\begin{align}
		\sum_{K \ge 0} \P \left(\sup_{t \in [K, K+1]}  \left | R_{n,m}(t) - R_{n,m}(K)  \right | > K^{\frac{2}{3}}   \right) &\le \sum_{K \ge 0} \frac{\E[ \left | R_{n,m}(K+1) - R_{n,m}(K)  \right |^2  ]}{K^{\frac{4}{3}} } \nonumber  \\
		&\le \sum_{K \ge 0} \frac{S_{n,n}}{K^{\frac{4}{3}}} \nonumber,
	\end{align}
which we note converges as $K \rightarrow \infty$. Therefore, by the Borel Cantelli Lemma, we have 
	\begin{align}
		\P \left( \left \lbrace \sup_{t \in [K, K+1]}  \left | R_{n,m}(t) - R_{n,m}(K)  \right | > K^{\frac{2}{3}}  \right \rbrace  \text { infinitely often}\right) &=0 \nonumber.
	\end{align}
This means for large $t \ge 0,$ there is a unique integer $K(t)$ with $K(t) -1 \le t \le K(t)$ such that
    \begin{align}
		\left | \frac{R_{n,m}(t)}{t}  \right | & \le 	\left | \frac{R_{n,m}(t)-R_{n,m}(K(t))}{t}  \right | + \left | \frac{R_{n,m}(K(t))}{t} \right | \nonumber \\ 
		&= \left | \frac{R_{n,m}(t)-R_{n,m}(K(t))}{t}  \right | + \left | \frac{R_{n,m}(K(t))}{K(t)} \right | \nonumber \\ 
		& \le  \frac{ K(t)^{\frac{2}{3}}}{K(t)}  + \left | \frac{R_{n,m}(K(t))}{K(t)} \right | \nonumber \\
		&= \frac{1}{K(t)^{\frac{1}{3}}}+ \left | \frac{R_{n,m}(K(t))}{K(t)} \right | \label{eq: Lead Process Limit 4},
    \end{align}
almost surely. Therefore, as $t \rightarrow \infty,$ we have $K(t) \rightarrow \infty,$  which means \eqref{eq: Lead Process Limit 4} converges almost surely to $0.$ Hence, we deduce the identity \eqref{eq: Lead Process Limit 2} and conclude the second integral in \eqref{eq: Lead Process Limit 1} almost surely converges to $\mathbf 0 \boldsymbol \Sigma^\text T=\mathbf 0.$

Thus, we conclude
	\begin{align}
		\lim_{t \rightarrow \infty} \frac{\mathbf A(t)}{t} &=- \frac{\mathbf S \mathbf B^\text T}{2} - \left(-\frac{\mathbf S \mathbf B^\text T}{2} \right)^\text T \nonumber \\
		&= \frac{\mathbf B \mathbf S- \mathbf S \mathbf B^\text T}{2} \nonumber.
	\end{align}
almost surely.
\end{proof}

Because of this strong law of large numbers identity, we use the matrix
	\begin{align}\label{eq: OU Process Lead Matrix}
		\mathbf Q &= \frac{\mathbf B \mathbf S - \mathbf S \mathbf B^\text T}2
	\end{align}
as our main tool for the Cyclicity Analysis of the OU process. We refer to $\mathbf Q$ as the \textit{lead matrix} of the OU process. In the next chapter, we will investigate the structure of this matrix in depth.

\chapter{Cyclicity Analysis of the Cyclic Ornstein-Uhlenbeck Process}\label{chap: cyclicouprocess}

In this chapter, we consider the $N$-dimensional cyclic OU process $\left \lbrace \mathbf x(t) \right \rbrace_{t \ge 0}$ with model parameters $\mathbf B$ and $\boldsymbol \Sigma,$ in which $\mathbf B$ is a circulant friction matrix. We write $\mathbf B = \text{Circ}(b_1 \ , \ \dots \ , \ b_N),$ in which $b_p \le 0$ for each $p>1$ and $b_1>0$ is large enough so that $\mathbf B$ is a friction matrix. We let $\mathbf D$ be the $N \times N$ diffusion matrix corresponding to this OU process, as defined in \eqref{eq: OU Process Diffusion Matrix}. Recall by \textbf{Theorem \ref{thm: Circulant Matrix Properties}}, for each $1 \le n \le N,$
	\begin{align}
		\mu_n &= \sum_{p=1}^N b_p \ \omega_n^{p-1} \label{eq: Cyclic OU Process Friction Matrix Eigenvalue} 
	\end{align}
is an eigenvalue of $\mathbf B$ with associated eigenvector
	\begin{align}
		\mathbf w_n &= \frac{1}{\sqrt N} \left(1 \ , \ \omega_n \ , \ \dots \ , \ \omega_n^{N-1} \right) \label{eq: Cyclic OU Process Friction Matrix Eigenvector} ,
	\end{align}
in which $\omega_n=e^{\frac{2 \pi i n}{N}}.$

Using the eigenvalues of $\mathbf B$, we can determine the explicit lower bound of $b_1$ to ensure that $\mathbf B$ is a valid friction matrix.

\begin{theorem}[Cyclic OU Process Friction Matrix Stability Condition]\label{thm: Cyclic OU Process Friction Matrix Stability Condition}
$\mathbf B$ is a valid friction matrix if and only if $b_1 > -\sum_{p=2}^N b_p.$
\end{theorem}

\begin{proof}
In order for $\mathbf B$ to be a valid friction matrix, $-\mathbf B$ must be Hurwitz i.e. $\Re(\mu_n)>0$ for all $1 \le n \le N,$ where $\Re(z)$ is the real part of the complex number $z.$ This is equivalent to the condition that  $\min_{1 \le n \le N} \Re(\mu_n)>0.$

We determine $\min_{1 \le n \le N} \Re(\mu_n)$ explicitly.
For each $1 \le n \le N,$  we extract the real part of $\mu_n$ to obtain
	\begin{align}
			\Re(\mu_n) &= \sum_{p=1}^N b_p \ \cos \left(\frac{2 \pi n (p-1)}{N} \right) \nonumber \\
			&= b_1 +\sum_{p=2}^N b_p \ \cos \left(\frac{2 \pi n (p-1)}{N} \right) \nonumber \\
			&\ge  b_1 +\sum_{p=2}^N b_p  \label{eq: Circulant Friction Matrix Eigenvalue Real Part 1} \\
			&= \Re(\mu_N), \nonumber 
	\end{align}
in which the inequality \eqref{eq: Circulant Friction Matrix Eigenvalue Real Part 1} follows from the fact that $b_p$ is non-positive and $b_p \cos \left(\frac{2 \pi n (p-1)}{N} \right) \ge b_p$ for each $1 <p \le N.$ 
Hence, we see
	\begin{align}
		\min_{1 \le n \le N} \Re(\mu_n) &= \Re(\mu_N), \nonumber
	\end{align}
which means $\Re(\lambda_N) >0$ if and only if 
	\begin{align}
		b_1 +\sum_{p=2}^N b_p >0. \nonumber
	\end{align}
Rearranging the above inequality yields the claim.
\end{proof}

In this chapter, we consider the signal propagation model that we introduced \textbf{Chapter \ref{chap: Introduction}} governed by the aforementioned cyclic OU process. In this model, there is a one-dimensional signal propagating throughout a network consisting of $N$ sensors. Writing out $\mathbf x(t) = \left(x_1(t) \ , \ \dots \ , \ x_N(t) \right)$ for each $t \ge 0,$ we interpret $x_n(t)$ as the measurement made by the $n$-th sensor of the propagating effect at time $t$. The circulant friction matrix $\mathbf B$ represents the sensor network structure. More specifically, for each $p>1,$ the constant $b_p$ is a \textit{propagation coefficient}, whose magnitude quantifies how receptive the $n$-th sensor is to the activity within the $(n+p-1)$-th sensor for each $1 \le n \le N,$ in which $n+p-1$ is indexed mod $N.$ The constant $b_1$ as a \textit{suppression coefficient}, whose magnitude quantifies the rate at which the propagating signal in our model dissipates. Finally, the diffusion matrix $\mathbf D$ is the covariance matrix of the noise injected into the sensors. In particular, $D_{m,n}$ is the covariance between $m$-th component noise injected into the $m$-th sensor and the $n$-th component noise injected into the $n$-th sensor for each $1 \le m,n \le N.$ 

For large dimension $N \in \N$ and different choices of cyclic OU model parameters $\mathbf B$ and $\boldsymbol \Sigma,$ we compute the lead matrix $\mathbf Q$ defined in \eqref{eq: OU Process Lead Matrix}. We investigate whether Cyclicity Analysis enables us to recover the structure of the cyclic network structure induced by $\mathbf B.$ More specifically, we investigate whether the structure of $\mathbf v_1(\mathbf Q),$ the leading eigenvector of $\mathbf Q,$ reflects the cyclic network structure. We investigate this problem under two regimes: the first regime is when $b_1$ is far from $\sum_{p=2}^N b_p$ and the second regime is when $b_1$ is close to $\sum_{p=2}^N b_p.$

\section{Lead Matrix Explicit Representation}

In this section, we derive the explicit formula of the lead matrix $\mathbf Q$ corresponding to the cyclic OU Process generated by the model parameters $\mathbf B$ and $\boldsymbol \Sigma$ defined at the beginning of the chapter.    

\begin{theorem}[Cyclic OU Process Lead Matrix Explicit Formula]\label{thm: Cyclic OU Process Lead Matrix Explicit Formula}
The lead matrix $\mathbf Q$ of the cyclic OU process is explicitly 
    \begin{align}\label{eq: Cyclic OU Process Lead Matrix}
        \mathbf Q &=  \sum_{m,n=1}^N \frac{\sum_{p=2}^N b_p (\omega_m^{p-1} - \omega_n^{p-1})}{2b_1 + \sum_{q=2}^N b_q (\omega_m^{q-1} + \omega_n^{q-1})} \  \mathbf w_m \ \mathbf w_{N-m}^\text T \  \mathbf D \ \mathbf w_{N-n}\ \mathbf w_n^\text T.
    \end{align}
\end{theorem}

\begin{proof}
First, we explicitly compute the stationary covariance matrix $\mathbf S$ of the OU process in \eqref{eq: OU Process Stationary Covariance Matrix}. By \textbf{Theorem \ref{thm: Circulant Matrix Properties}}, the Green's Function $\mathbf G(t),$ defined in \eqref{eq: OU Process Green Function}, is explicitly
	\begin{align}
		\mathbf G(t) &= \sum_{n=1}^N e^{-\mu_n t} \  \mathbf w_n \ \mathbf w_{N-n}^\text T \nonumber.
	\end{align}
Substituting this expression for $\mathbf G(t)$ into the integrand of \eqref{eq: OU Process Stationary Covariance Matrix}, we have 
    \begin{align}
        \mathbf S &= 2 \int_0^\infty \sum_{m,n=1}^N e^{- \mu_m t} \ \mathbf w_m \  \mathbf w_{N-m}^ \text T \  \mathbf D  \ \left(e^{- \mu_n t} \ \mathbf w_n \ \mathbf w_{N-n}^\text T \right)^ \text T \ dt \nonumber \\
        &= 2 \sum_{m,n=1}^N \int_0^\infty e^{-t (\mu_m+\mu_n)} \ \mathbf w_m \ \mathbf w_{N-m}^\text T \  \mathbf D \ \mathbf w_{N-n} \ \mathbf w_n^\text T \ dt \nonumber \\
        &= 2 \sum_{m,n=1}^N \frac{\mathbf w_m \ \mathbf w_{N-m}^\text T \  \mathbf D \ \mathbf w_{N-n}\ \mathbf w_n^\text T}{\mu_m + \mu_n}.  \nonumber
    \end{align}

Now, upon substituting this explicit representation of $\mathbf S$ into the lead matrix definition in \eqref{eq: OU Process Lead Matrix}, we obtain
    \begin{align}
        \mathbf Q &=  \mathbf B \left(\sum_{m,n=1}^N \frac{\mathbf w_m \ \mathbf w_{N-m}^\text T \  \mathbf D \ \mathbf w_{N-n} \ \mathbf w_n^\text T}{\mu_m + \mu_n} \right)  - \left(\sum_{m,n=1}^N \frac{\mathbf w_m \ \mathbf w_{N-m}^\text T \  \mathbf D \ \mathbf w_{N-n}\ \mathbf w_n^\text T}{\mu_m + \mu_n} \right) \mathbf B^ \text T \nonumber \\
         &=  \left(\sum_{m,n=1}^N \frac{ \mu_m \mathbf w_m \ \mathbf w_{N-m}^\text T \  \mathbf D \ \mathbf w_{N-n} \ \mathbf w_n^\text T}{\mu_m + \mu_n} \right)  - \left(\sum_{m,n=1}^N \frac{\mathbf w_m \ \mathbf w_{N-m}^\text T \  \mathbf D \ \mathbf w_{N-n}\ (\mathbf B \  \mathbf w_n)^\text T }{\mu_m + \mu_n} \right)  \nonumber \\
        &=  \sum_{m,n=1}^N \frac{\mu_m - \mu_n }{\mu_m + \mu_n} \ \mathbf w_m \ \mathbf w_{N-m}^\text T \  \mathbf D \ \mathbf w_{N-n} \ \mathbf w_n^\text T \nonumber \\
        &= \sum_{m,n=1}^N \frac{\sum_{p=1}^N b_p (\omega_m^{p-1} - \omega_n^{p-1})}{ \sum_{q=1}^N b_q (\omega_m^{q-1} + \omega_n^{q-1})} \  \mathbf w_m \ \mathbf w_{N-m}^\text T \  \mathbf D \ \mathbf w_{N-n} \ \mathbf w_n^\text T. \nonumber \\
        &= \sum_{m,n=1}^N \frac{\sum_{p=2}^N b_p (\omega_m^{p-1} - \omega_n^{p-1})}{2b_1 +  \sum_{q=2}^N b_q (\omega_m^{q-1} + \omega_n^{q-1})} \  \mathbf w_m \ \mathbf w_{N-m}^\text T \  \mathbf D \ \mathbf w_{N-n} \ \mathbf w_n^\text T \nonumber.
    \end{align}
This completes the proof.
\end{proof}

As mentioned in the beginning of the chapter, we consider two regimes for our signal propagation model: the first regime in which the propagating signal exhibits significant dissipation and the second regime in which the signal does not exhibit significant dissipation. In order to better distinguish these regimes and make it suitable for future calculations, we modify our setup. For each $\epsilon>0,$ we consider the circulant friction matrix
    \begin{align}\label{eq: Cyclic OU Process Friction Matrix}
        \mathbf B(\epsilon)= \text{Circ}\left(\epsilon-\sum_{p=2}^N b_p  \ , \ b_2 \ , \ \dots \ , \ b_N \right)
    \end{align}
in place of the original friction matrix $\mathbf B,$ in which the constant $\epsilon>0$ is a \textit{perturbation}. In this setup, the former regime occurs when $\epsilon$ is large i.e. the friction matrix is significantly perturbed from an unstable matrix, while the latter regime occurs when $\epsilon$ is small i.e. the friction matrix is slightly perturbed from an unstable matrix. So in future sections, we will consider the cyclic OU process with model parameters $\mathbf B(\epsilon)$ and $\boldsymbol \Sigma$ with $\mathbf B(\epsilon)$ defined as in \eqref{eq: Cyclic OU Process Friction Matrix}. Consequently, we let $\mathbf Q(\epsilon)$ be the lead matrix of this cyclic OU process and $\mathbf v_1(\mathbf Q(\epsilon))$ be its leading eigenvector.

In future sections, we will also extend the big-Oh notation to matrices in the following way: if $g: [0, \infty) \rightarrow \R$ is a function, then we write $O(g(\epsilon))$ to represent an $M \times N$ matrix-valued function of $\epsilon$ such that there exists a $K \ge 0$ for which all $MN$ entry functions of $\epsilon$ are bounded above by $K g(\epsilon).$ We will specify in context as to whether $\epsilon$ tends to infinity or tends to $0.$

\section{Only One Nonzero Propagation Coefficient and Injection of Noise into Only One Sensor}\label{sec: Only One Nonzero Propagation Coefficient and Injection of Noise into Only One Sensor}

In this section, for each perturbation $\epsilon>0,$ we consider the signal propagation model governed by the Cyclic OU process with model parameters $\mathbf B(\epsilon)$ and $\boldsymbol \Sigma,$ in which $\mathbf B(\epsilon)$ is defined as in \eqref{eq: Cyclic OU Process Friction Matrix}. Here, we assume exactly one propagation coefficient is nonzero i.e. $b_p$ is nonzero for only one index $1 < p \le N.$ The ground truth network structure is such that the $n$-th sensor is only linked to the $(n+p-1)$-th sensor for each $1 \le n \le N,$ where $n+p-1$ is indexed mod $N.$ We further assume $p-1$ and $N$ are coprime. Note that this condition ensures the signal reaches every sensor in the network. To see this, $\gcd(p-1,N)=1$ implies there exist integers $x,y$ for which $Nx+(p-1)y =1$ i.e. $(p-1) y \equiv 1 \mod N$ has an integer solution $y>0.$ Hence, for each $1 \le  k <N,$ we have 
	\begin{align}
		N - k(p-1)y \equiv k \mod N \nonumber.
	\end{align} 
In other words, if the signal is broadcast from the $N$-th sensor, then it reaches the $k$-th sensor after passing through $ky$ intermediate sensors.

Moreover, we assume that $\boldsymbol \Sigma$ is a volatility matrix such that $\mathbf D$ is a diagonal diffusion matrix whose $s$-th diagonal entry $d_s$ is nonzero for only one index $1 \le s \le N.$ This corresponds to the situation where noise is only injected into only one sensor, namely the $s$-th sensor. 

We derive the lead matrix $\mathbf Q(\epsilon)$ corresponding to this particular cyclic OU process with the assumptions imposed on $\mathbf B(\epsilon)$ and $\boldsymbol \Sigma.$

\begin{theorem}[One Nonzero Propagation Coefficient, Noise in One Sensor Lead Matrix Explicit Formula]\label{thm: Cyclic OU Process One Nonzero Propagation Coefficient and Noise in One Sensor Lead Matrix Explicit Formula}
For each $\epsilon>0,$ with the assumptions imposed on the model parameters $\mathbf B(\epsilon)$ and $\boldsymbol \Sigma,$ the lead matrix $\mathbf Q(\epsilon)$ is explicitly 
	\begin{align}\label{eq: Cyclic OU Process One Nonzero Propagation Coefficient and Noise in One Sensor Lead Matrix}
		\mathbf Q(\epsilon) &= \frac{d_s b_p}{N} \sum_{m,n=1}^N \frac{(\omega_m^{p-1} - \omega_n^{p-1}) \  \omega_{m+n}^{1-s}}{2 \epsilon + b_p (\omega_m^{p-1} + \omega_n^{p-1}-2)} \ \mathbf w_m \mathbf w_n^\text T.
	\end{align}
\end{theorem}

\begin{proof}
Substituting $b_1=\epsilon-b_p$ and $b_q=0$ for all $1<q \le N$ with $q \ne p$ into the general cyclic OU process lead matrix formula \eqref{eq: Cyclic OU Process One Nonzero Propagation Coefficient and Noise in One Sensor Lead Matrix}, we simplify the resulting summand using matrix multiplication:
	\begin{align}
		\mathbf Q(\epsilon) &=  \sum_{m,n=1}^N \frac{b_p (\omega_m^{p-1}- \omega_n^{p-1})}{2\epsilon  + b_p(\omega_m^{p-1}+ \omega_n^{p-1}-2)} \  \mathbf w_m \ \mathbf w_{N-m}^\text T \  \mathbf D \ \mathbf w_{N-n} \ \mathbf w_n^\text T \nonumber \\
		&= \sum_{m,n=1}^N \frac{b_p (\omega_m^{p-1}- \omega_n^{p-1} )}{2\epsilon  + b_p(\omega_m^{p-1} + \omega_n^{p-1}-2)} \  \mathbf w_m \ \mathbf w_{N-m}^\text T \ \left(\frac{d_s \  \omega_{N-n}^{s-1}}{\sqrt N} \ \mathbf e_s \right) \ \mathbf w_n^\text T \nonumber \\
		&= \sum_{m,n=1}^N \frac{b_p (\omega_m^{p-1}- \omega_n^{p-1})}{2\epsilon  + b_p(\omega_m^{p-1} + \omega_n^{p-1}-2)} \  \mathbf w_m \ \left(\frac{d_s  \ \omega_{N-m}^{s-1} \  \omega_{N-n}^{s-1}}{N} \right) \ \mathbf w_n^\text T \nonumber  \\
		&= \sum_{m,n=1}^N \frac{b_p (\omega_m^{p-1}- \omega_n^{p-1})}{2\epsilon  + b_p(\omega_m^{p-1} + \omega_n^{p-1}-2)} \  \mathbf w_m \ \left(\frac{d_s \  \omega_{-m-n}^{s-1}}{N}  \right) \ \mathbf w_n^\text T \nonumber \\ 
		&= \frac{d_sb_p}{N} \sum_{m,n=1}^N \frac{ (\omega_m^{p-1}- \omega_n^{p-1}) \ \omega_{m+n}^{1-s}}{2\epsilon  + b_p(\omega_m^{p-1} + \omega_n^{p-1}-2) } \  \mathbf w_m  \mathbf w_n^\text T\nonumber.
	\end{align} 
This proves the claim.
\end{proof}

Now, under the two regimes, we determine whether the structure of the leading eigenvector of $\mathbf Q(\epsilon)$ enables us to recover the expected network structure induced by $\mathbf B(\epsilon)$.

\subsection{First Regime}
We start with the first regime, in which the perturbation $\epsilon$ is large. We derive the asymptotic expansion of the lead matrix under this regime.

\begin{theorem}[One Nonzero Propagation Coefficient, Noise in One Sensor First Regime Lead Matrix Asymptotic Expansion]\label{thm: Cyclic OU Process One Nonzero Propagation Coefficient and Noise in One Sensor First Regime Lead Matrix Asymptotic Expansion}
If $\mathbf Q(\epsilon)$ is the lead matrix in \eqref{eq: Cyclic OU Process One Nonzero Propagation Coefficient and Noise in One Sensor Lead Matrix}, then as $\epsilon \rightarrow \infty,$ we have
	\begin{align} \label{eq: Cyclic OU Process One Nonzero Propagation Coefficient and Noise in One Sensor First Regime Lead Matrix Asymptotic Formula}
		\mathbf Q(\epsilon) &= \mathbf A(\epsilon) + O \left(\frac{1}{\epsilon^2} \right),
	\end{align}
where $\mathbf A(\epsilon)$ is the $N \times N$ skew-symmetric matrix whose $(j,k)$-th entry is
	\begin{align}\label{eq: Cyclic OU Process One Nonzero Propagation Coefficient and Noise in One Sensor First Regime Lead Matrix Asymptotic Entry}
		A_{j,k}(\epsilon) &= \begin{cases} \frac{d_s b_p}{2 \epsilon} & j \equiv s+1-p \mod N \ ,  \ k=s \\
			-\frac{d_s b_p}{2 \epsilon} & j=s \ , \ k \equiv s+1-p \mod N \\
			0 & \text{else}
		\end{cases}. 
	\end{align}
\end{theorem}

\begin{proof}
Upon dividing the numerator and denominator of the summand in \eqref{eq: Cyclic OU Process One Nonzero Propagation Coefficient and Noise in One Sensor Lead Matrix} by $2 \epsilon,$ we obtain 
	\begin{align} \label{eq: Cyclic OU Process One Nonzero Propagation Coefficient and Noise in One Sensor First Regime Lead Matrix Asymptotic Expansion 1}
		\mathbf Q(\epsilon) &= \frac{d_sb_p}{N} \sum_{m,n=1}^N \frac{\frac{ (\omega_m^{p-1}-\omega_n^{p-1})  \ \omega_{m+n}^{1-s}}{2 \epsilon}}{1+ \frac{b_p (\omega_m^{p-1}+\omega_n^{p-1}-2)}{2 \epsilon}} \ \mathbf w_m \mathbf w_n^\text T.
	\end{align}
Suppose $\epsilon$ satisfies
    \begin{align}
    	\left |  \frac{b_p \ \left(\omega_m^{p-1}+\omega_n^{p-1}-2 \right)}{2 \epsilon}\right | &<1 \nonumber
    \end{align}
for all $1 \le m,n \le N.$ We can then rewrite the summand on the right hand side of  \eqref{eq: Cyclic OU Process One Nonzero Propagation Coefficient and Noise in One Sensor First Regime Lead Matrix Asymptotic Expansion 1} as a geometric series:
	\begin{align}
		\mathbf Q(\epsilon) &= \frac{d_s b_p}{N} \sum_{m,n=1}^N \frac{(\omega_m^{p-1} - \omega_n^{p-1})  \ \omega_{m+n}^{1-s}}{2 \epsilon} \ \sum_{r \ge 0} \left( - \frac{b_p (\omega_m^{p-1}+\omega_n^{p-1}-2)}{2 \epsilon} \right)^r \ \mathbf w_m \mathbf w_n^\text T \nonumber \\
		&= \frac{d_s b_p}{N} \sum_{r \ge 0} \sum_{m,n=1}^N  \frac{  \left(-b_p \left(\omega_m^{p-1}+\omega_n^{p-1}-2 \right) \right)^r (\omega_m^{p-1} - \omega_n^{p-1}) \ \omega_{m+n}^{1-s} }{(2 \epsilon)^{r+1}}\ \mathbf w_m  \mathbf w_n^\text T  \label{eq: Cyclic OU Process One Nonzero Propagation Coefficient and Noise in One Sensor First Regime Lead Matrix Asymptotic Expansion 2} \\
		&= \frac{d_s b_p}{2 \epsilon N } \sum_{m,n=1}^N (\omega_m^{p-1} - \omega_n^{p-1}) \ \omega_{m+n}^{1-s} \ \mathbf w_m  \mathbf w_n^\text T+ O \left(\frac{1}{\epsilon^2} \right) \label{eq: Cyclic OU Process One Nonzero Propagation Coefficient and Noise in One Sensor First Regime Lead Matrix Asymptotic Expansion 3},
	\end{align}
in which \eqref{eq: Cyclic OU Process One Nonzero Propagation Coefficient and Noise in One Sensor First Regime Lead Matrix Asymptotic Expansion 3} is the result of extracting the inner double sum corresponding to the $r=0$ term in \eqref{eq: Cyclic OU Process One Nonzero Propagation Coefficient and Noise in One Sensor First Regime Lead Matrix Asymptotic Expansion 2} and replacing the remaining sum over all $r>0$ with $O \left(\frac{1}{\epsilon^2} \right).$ Extracting the $(j,k)$-th entry of \eqref{eq: Cyclic OU Process One Nonzero Propagation Coefficient and Noise in One Sensor First Regime Lead Matrix Asymptotic Expansion 3}, we get
	\begin{align}
		Q_{j,k}(\epsilon) &= \frac{d_s b_p}{2 \epsilon N^2} \sum_{m,n=1}^N (\omega_m^{p-1} - \omega_n^{p-1}) \ \omega_{m+n}^{1-s} \ \omega_m^{j-1} \ \omega_n^{k-1}+ O \left(\frac{1}{\epsilon^2} \right)  \nonumber \\
		&= \frac{d_s b_p}{2 \epsilon N^2} \sum_{m,n=1}^N \left(\omega_m^{p+j-s-1} \omega_n^{k-s} - \omega_m^{j-s} \omega_n^{p+k-s-1} \right )+ O \left(\frac{1}{\epsilon^2} \right). \label{eq: Cyclic OU Process One Nonzero Propagation Coefficient and Noise in One Sensor First Regime Lead Matrix Asymptotic Expansion 4}
	\end{align}
	
To complete the proof, we evaluate the double sum appearing on the right hand side of \eqref{eq: Cyclic OU Process One Nonzero Propagation Coefficient and Noise in One Sensor First Regime Lead Matrix Asymptotic Expansion 4}. In order to do so, we utilize the following geometric series identity for all $q \in \Z$:
    \begin{align}\label{eq: Finite Geometric Series Identity}
		\frac{1}{N} \sum_{n=1}^N \omega_n^q &= \begin{cases} 1 & q \equiv 0 \mod N \\ 0 & \text{else} \end{cases}.
    \end{align}
Using \eqref{eq: Finite Geometric Series Identity}, we obtain
    \begin{align}
		Q_{j,k}(\epsilon) &=  \begin{cases}
		    \frac{d_s b_p}{2 \epsilon}  & j \equiv s+1-p \mod N \ , \ k \equiv s \mod N \\
      0 & \text{else}
		\end{cases} \nonumber   \\
  & \qquad \qquad - \begin{cases}
		    \frac{d_s b_p}{2 \epsilon}  & k \equiv s+1-p \mod N \ , \ j \equiv s \mod N \\
      0 & \text{else}
		\end{cases} \ + O \left(\frac{1}{\epsilon^2} \right) \label{eq: Cyclic OU Process One Nonzero Propagation Coefficient and Noise in One Sensor First Regime Lead Matrix Asymptotic Expansion 5}
    \end{align}
Now, for $1 \le j,k \le N,$ consider the systems of congruence equalities appearing in \eqref{eq: Cyclic OU Process One Nonzero Propagation Coefficient and Noise in One Sensor First Regime Lead Matrix Asymptotic Expansion 5}, namely
    \begin{align}
        j &\equiv s+1-p \mod N \ , \ k \equiv s \mod N \label{eq: Cyclic OU Process One Nonzero Propagation Coefficient and Noise in One Sensor First Regime Lead Matrix Asymptotic Expansion 6}\\
        k & \equiv s+1-p \mod N \ , \ j \equiv s \mod N \label{eq: Cyclic OU Process One Nonzero Propagation Coefficient and Noise in One Sensor First Regime Lead Matrix Asymptotic Expansion 7}.
    \end{align}
Observe the second congruence equation in \eqref{eq: Cyclic OU Process One Nonzero Propagation Coefficient and Noise in One Sensor First Regime Lead Matrix Asymptotic Expansion 6} holds if and only if $k-s$ divides $N.$ But since $k-s <N,$ the second congruence equation in \eqref{eq: Cyclic OU Process One Nonzero Propagation Coefficient and Noise in One Sensor First Regime Lead Matrix Asymptotic Expansion 6} holds if and only if $k=s.$ Similarly, the second congruence equation in \eqref{eq: Cyclic OU Process One Nonzero Propagation Coefficient and Noise in One Sensor First Regime Lead Matrix Asymptotic Expansion 7} holds if and only if $j=s.$ Moreover, note that \eqref{eq: Cyclic OU Process One Nonzero Propagation Coefficient and Noise in One Sensor First Regime Lead Matrix Asymptotic Expansion 6} and \eqref{eq: Cyclic OU Process One Nonzero Propagation Coefficient and Noise in One Sensor First Regime Lead Matrix Asymptotic Expansion 7} cannot simultaneously hold if $p>2$. Otherwise, we would have both $j \equiv s \mod N$ and $j \equiv s+1-p \mod N,$ which would imply $s \equiv s+1-p \mod N$ i.e. $p \equiv 1 \mod N.$  This would mean $p-1$ divides $N,$ and if $p>2,$ this would contradict the assumption $\gcd(p-1,N)=1.$ Hence, we can simplify \eqref{eq: Cyclic OU Process One Nonzero Propagation Coefficient and Noise in One Sensor First Regime Lead Matrix Asymptotic Expansion 5} in the following way:
    \begin{align} \nonumber
    Q_{j,k}(\epsilon) &= \left( \begin{cases}
    		    \frac{d_s b_p}{2 \epsilon}  & j \equiv s+1-p \mod N \ , \ k=s \\
    		    -\frac{d_s b_p}{2 \epsilon}  & k \equiv s+1-p \mod N \ , \ j=s  \\ 
            0 & \text{else}
          \end{cases} \right) +  \ O \left(\frac{1}{\epsilon^2} \right),
    \end{align}
which proves the claim.
\end{proof}

Knowing the asymptotic structure of the lead matrix under the first regime, we now investigate the asymptotic structure of its leading eigenvector.

\begin{theorem}[One Nonzero Propagation Coefficient, Noise in One Sensor First Regime Leading Eigenvector Asymptotic Expansion]\label{thm: Cyclic OU Process One Nonzero Propagation Coefficient and Noise in One Sensor First Regime Leading Eigenvector Asymptotic Expansion}
The leading eigenvector of $\mathbf Q(\epsilon)$ satisfies
    \begin{align}\label{eq: Cyclic OU Process One Nonzero Propagation Coefficient and Noise in One Sensor First Regime Leading Eigenvector}
        \lim_{\epsilon \rightarrow \infty} \mathbf v_1(\mathbf Q(\epsilon)) &= \frac{ i \ \mathbf e_s + \mathbf e_{s+1-p}}{\sqrt 2},
    \end{align}
where $\mathbf e_s$ is the $s$-th canonical basis vector in $\C^N$ and $s+1-p$ is indexed mod $N$ as before. 
\end{theorem}

\begin{proof}
Consider $N \times N$ matrix $\mathbf A(\epsilon)$ whose $(j,k)$-th entry is defined in \eqref{eq: Cyclic OU Process One Nonzero Propagation Coefficient and Noise in One Sensor First Regime Lead Matrix Asymptotic Entry} for all $1 \le j,k \le N.$ Note that $\mathbf A(\epsilon)$ is a skew-symmetric matrix that has only two nonzero rows, namely
$\frac{d_s b_p}{2 \epsilon} \  \mathbf e_{s}$ and $-\frac{d_s b_p}{2 \epsilon} \  \mathbf e_{s+1-p}.$ Since these rows are orthogonal, we deduce $\mathbf A(\epsilon)$ has rank $2.$ Moreover, the previous theorem states as $\epsilon \rightarrow \infty,$ we have
    \begin{align} \nonumber
        \| \mathbf Q(\epsilon)- \mathbf A(\epsilon) \|_F &= O \left(\frac{1}{\epsilon^2} \right),
    \end{align}
which implies $\mathbf A(\epsilon)$ a low rank approximation of $\mathbf Q(\epsilon)$ as $\epsilon \rightarrow \infty.$

We now determine the leading eigenvector of $\mathbf A(\epsilon).$ Define
	\begin{align}
		\mathbf a(\epsilon) &= \sqrt{ \frac{d_s b_p}{2 \epsilon}}  \ \mathbf e_{s+1-p} \nonumber \\
		\mathbf b(\epsilon) &=  \sqrt{\frac{d_s b_p}{2 \epsilon} }  \ \mathbf e_{s} \nonumber.
	\end{align}
Note that
	\begin{align}
		\mathbf A(\epsilon) &= \mathbf a(\epsilon)  \ \mathbf b^\text T(\epsilon)- \mathbf b(\epsilon) \  \mathbf a^\text T(\epsilon).
	\end{align} 
Because $\mathbf a(\epsilon)$ and $\mathbf b(\epsilon)$ are orthogonal, they are linearly independent, and the angle between them is $\frac{\pi}{2}.$ Hence, by the formula for the eigenvalue and eigenvector of a rank $2$ skew-symmetric matrix stated in  \cite{BaryshnikovSchlafly2016}, the largest eigenvalue of $\mathbf A(\epsilon)$ is
	\begin{align}
		\lambda_1(\mathbf A(\epsilon)) &= i \sin \left(\frac{\pi}{2} \right) \| \mathbf a(\epsilon)\| \|\mathbf b(\epsilon)\| \\
		&=\frac{d_s b_p}{2 \epsilon} \ i \nonumber.
	\end{align}
Its associated eigenvector is the leading eigenvector of $\mathbf A(\epsilon),$ which is explicitly
	\begin{align}
		\mathbf v_1(\mathbf A(\epsilon)) &=\frac{ -e^{- \frac{\pi}{2} i } \frac{\mathbf a(\epsilon)}{\| \mathbf a(\epsilon)\|} +  \frac{\mathbf b(\epsilon)}{\| \mathbf b(\epsilon)\|}}{\left \|  -e^{- \frac{\pi}{2} i } \frac{\mathbf a(\epsilon)}{\| \mathbf a(\epsilon)\|} +  \frac{\mathbf b(\epsilon)}{\| \mathbf b(\epsilon)\|} \right \|} \nonumber \\
		&=\frac{1}{\left \|\frac{ i \  \mathbf e_s - \mathbf e_{s+1-p}}{\sqrt{\frac{d_s b_2}{2 \epsilon}}} \right \|}  \left( \frac{ i \  \mathbf e_s + \mathbf e_{s+1-p}}{\sqrt{\frac{d_s b_2}{2 \epsilon}}} \right)\nonumber \\
		&=\frac{1}{\sqrt{\frac{2}{\frac{d_s b_p}{2 \epsilon}}}} \left(\frac{ i \  \mathbf e_s + \mathbf e_{s+1-p}}{\sqrt{\frac{d_s b_p}{2 \epsilon}}} \right) \nonumber \\
		&=  \frac{ i \  \mathbf e_s + \mathbf e_{s+1-p}}{\sqrt 2} \nonumber.
	\end{align}
\end{proof}

We interpret the result of this previous theorem. In the first regime, the leading eigenvector of the lead matrix $\mathbf Q(\epsilon)$ converges to a vector with exactly two nonzero components with equal moduli, namely the $s$-th and $(s+1-p)$-th components. The $s$-th component lies on the real axis, while the $(s+1-p)$-th component lies on the imaginary axis. All other components are equal to $0.$ Therefore, the structure of the asymptotic leading eigenvector under the first regime suggests that only the $s$-th and $(s+1-p)$-th sensors receive the dissipating signal that is propagating throughout the network. Thus, Cyclicity Analysis does not enable us to fully recover the expected network structure under the first regime. 

\subsection{Second Regime}
We now investigate the second regime, in which the perturbation $\epsilon$ is small. We derive the asymptotic expansion of the lead matrix under this regime.

\begin{theorem}[One Nonzero Propagation Coefficient and Noise in One Sensor Second Regime Lead Matrix Asymptotic Formula]\label{thm: Cyclic OU Process One Nonzero Propagation Coefficient and Noise in One Sensor Second Regime Lead Matrix Asymptotic Formula}
Consider the lead matrix $\mathbf Q(\epsilon)$ in \eqref{eq: Cyclic OU Process One Nonzero Propagation Coefficient and Noise in One Sensor Lead Matrix}.  Then, as $\epsilon \rightarrow 0,$ the matrix $\mathbf Q(\epsilon)$ has an asymptotic expansion of the form

    \begin{align}\label{eq: Cyclic OU Process One Nonzero Propagation Coefficient and Noise in One Sensor Second Regime Lead Matrix Asymptotic Formula} 
        \mathbf Q(\epsilon)  &= \frac{d_s}{N} \sum_{m,n=1}^N \left(\frac{\left(\omega_m^{p-1}-\omega_n^{p-1} \right)  \  \omega_{m+n}^{1-s}} {\omega_m^{p-1}+\omega_n^{p-1}-2} \right) \ \mathbf w_m  \mathbf w_n^\text T + O(\epsilon),
	\end{align}
in which the summand inside the double sum is defined to be $0$ if $m=n=N.$
\end{theorem}

\begin{proof}
We use the same technique that was used to prove \textbf{Theorem \ref{thm: Cyclic OU Process One Nonzero Propagation Coefficient and Noise in One Sensor First Regime Lead Matrix Asymptotic Expansion}}. We divide the numerator and denominator of the summand in \eqref{eq: Cyclic OU Process One Nonzero Propagation Coefficient and Noise in One Sensor Lead Matrix} by the quantity $b_p \left(\omega_m^{p-1} + \omega_n^{p-1} -2 \right)$ to obtain 
	\begin{align}
		\mathbf Q(\epsilon) &= \frac{d_s b_p}{N} \sum_{m,n=1}^N \frac{ \frac{\left(\omega_m^{p-1}-\omega_n^{p-1} \right)  \  \omega_{m+n}^{1-s}} {b_p \left(\omega_m^{p-1} + \omega_n^{p-1} -2 \right)} }{1+ \frac{2 \epsilon}{b_p \left(\omega_m^{p-1} + \omega_n^{p-1} -2 \right)}} \ \mathbf w_m \mathbf w_n^ \text T \nonumber \\
		&= \frac{d_s}{N} \sum_{m,n=1}^N \frac{ \frac{\left(\omega_m^{p-1}-\omega_n^{p-1} \right)  \  \omega_{m+n}^{1-s}} {\omega_m^{p-1}+\omega_n^{p-1}-2} }{1+ \frac{2 \epsilon}{b_p \left(\omega_m^{p-1} + \omega_n^{p-1} -2 \right)}} \ \mathbf w_m \mathbf w_n^ \text T \label{eq: Cyclic OU Process One Nonzero Propagation Coefficient and Noise in One Sensor Lead Matrix 1}.
	\end{align}

Suppose 
	\begin{align}
		\left | \frac{2\epsilon}{b_p  \left(\omega_m^{p-1} + \omega_n^{p-1} -2 \right)} \right| <1 \nonumber,
	\end{align}
for all $1 \le m,n \le N.$ Then, we can rewrite \eqref{eq: Cyclic OU Process One Nonzero Propagation Coefficient and Noise in One Sensor Lead Matrix 1} as a geometric series. We have 
	\begin{align}
		\mathbf Q(\epsilon) &=  \frac{d_s}{N} \sum_{m,n=1}^N \frac{\left(\omega_m^{p-1}-\omega_n^{p-1} \right) \  \omega_{m+n}^{1-s}} {\omega_m^{p-1}+\omega_n^{p-1}-2} \  \sum_{r \ge 0} \left(- \frac{2 \epsilon}{b_p \ (\omega_m^{p-1}+\omega_n^{p-1}-2)} \right)^r \ \mathbf w_m \mathbf w_n^\text T \nonumber  \\
		&= \frac{d_s}{N} \sum_{r \ge 0} \sum_{m,n=1}^N \left(\frac{\left(\omega_m^{p-1}-\omega_n^{p-1} \right)  \  \omega_{m+n}^{1-s}} {\omega_m^{p-1}+\omega_n^{p-1}-2} \right) \  \left(- \frac{2 \epsilon}{b_ p(\omega_m^{p-1}+\omega_n^{p-1}-2)} \right)^r \ \mathbf w_m \mathbf w_n^\text T \label{eq: Cyclic OU Process One Nonzero Propagation Coefficient and Noise in One Sensor Lead Matrix 2} \\
		&= \frac{d_s}{N} \sum_{m,n=1}^N \left(\frac{\left(\omega_m^{p-1}-\omega_n^{p-1} \right)  \  \omega_{m+n}^{1-s}} {\omega_m^{p-1}+\omega_n^{p-1}-2} \right) \ \mathbf w_m  \mathbf w_n^\text T + O(\epsilon), \label{eq: Cyclic OU Process One Nonzero Propagation Coefficient and Noise in One Sensor Lead Matrix 3}
	\end{align}
in which we obtained \eqref{eq: Cyclic OU Process One Nonzero Propagation Coefficient and Noise in One Sensor Lead Matrix 3} by extracting the term $r=0$ in the outer sum of \eqref{eq: Cyclic OU Process One Nonzero Propagation Coefficient and Noise in One Sensor Lead Matrix 2} and replacing the sum over all $r>0$ with $O(\epsilon).$

Now, we need to justify defining the summand in the double sum of \eqref{eq: Cyclic OU Process One Nonzero Propagation Coefficient and Noise in One Sensor Second Regime Lead Matrix Asymptotic Formula} to be $0$ when $m=n=N.$ Suppose $1 \le m,n \le N$ are two indices satisfying
    \begin{align} \label{eq: Cyclic OU Process One Nonzero Propagation Coefficient and Noise in One Sensor Lead Matrix 4}
        \omega_m^{p-1}+ \omega_n^{p-1}=2.
    \end{align}
Upon taking the moduli of both sides and using the triangle inequality, we obtain
    \begin{align}
        2 &= |\omega_m^{p-1}+ \omega_n^{p-1}| \nonumber \\
        &\le |\omega_m^{p-1}| + |\omega_n^{p-1}| \nonumber  \\
        &= 1+1=2 \nonumber.
    \end{align}
Therefore, we have 
    \begin{align} \label{eq: Cyclic OU Process One Nonzero Propagation Coefficient and Noise in One Sensor Lead Matrix 5}
        |\omega_m^{p-1}+ \omega_n^{p-1}| = |\omega_m^{p-1}| + |\omega_n^{p-1}|.
    \end{align}
Recall if $z,w \in \C$ satisfy $|z+w| =|z|+|w|$ i.e. $z,w$ are such that equality is obtained in the triangle inequality for complex numbers, then either one of $z,w$ must be zero or $\text{Arg}(z)=\text{Arg}(w)$ (see \cite{MartinAhlfors1996}). Since $\omega_m^{p-1}, \omega_{n}^{p-1}$ are both nonzero, \eqref{eq: Cyclic OU Process One Nonzero Propagation Coefficient and Noise in One Sensor Lead Matrix 5} implies $\text{Arg}(\omega_m^{p-1})=\text{Arg}(\omega_n^{p-1})$ i.e. $\frac{2 \pi m(p-1)}{N} = \frac{2 \pi n(p-1)}{N},$ and so $m=n.$ Now, substituting $m=n$ into \eqref{eq: Cyclic OU Process One Nonzero Propagation Coefficient and Noise in One Sensor Lead Matrix 4}, we obtain $\omega_m^{p-1}=1,$ which implies $$m(p-1) \equiv 0 \mod N.$$ Because $\gcd(p-1,N)=1$ divides $0,$ we see the congruence equation has a solution for $m.$ In particular, by \cite[Section 2.6]{Hua1982}, we have
    \begin{align}
        m &= N(k+1)
    \end{align}
for $k \in \Z.$ But since $1 \le m \le N,$ this forces $m=N.$ This shows \eqref{eq: Cyclic OU Process One Nonzero Propagation Coefficient and Noise in One Sensor Lead Matrix 1} implies $m=n=N.$ Finally, note that
    \begin{align}
    \frac{\omega_m^{p-1}-\omega_m^{p-1}}{\omega_m^{p-1}+\omega_m^{p-1}-2} &=0 \nonumber 
    \end{align}
if $m<N.$ But taking the limit of the left hand side as $m \rightarrow N,$ we obtain via L'Hospital's rule:
    \begin{align}
        \lim_{m \rightarrow N} \frac{\omega_m^{p-1}-\omega_m^{p-1}}{\omega_m^{p-1}+\omega_m^{p-1}-2} &= \lim_{m \rightarrow N} \frac{\frac{2 \pi i (p-1)}{N} \left(\omega_m^{p-1}-\omega_m^{p-1} \right) }{\frac{2 \pi i (p-1)}{N} \left(\omega_m^{p-1}+\omega_m^{p-1} \right)} =0.
    \end{align}
This completes the proof.
\end{proof}

\section{Second Regime Special Case: First Propagation Coefficient Nonzero and Noise Injected into Last Sensor} 

At the current time of writing, we are not aware of any closed formulas for the eigenvalues and eigenvectors of the matrix defined on the right hand side of \eqref{eq: Cyclic OU Process One Nonzero Propagation Coefficient and Noise in One Sensor Second Regime Lead Matrix Asymptotic Formula}. Nevertheless, in this section, we investigate the second regime in the situation where $p=2$ and $s=N.$ That is, the ground truth network structure is such that $n$-th sensor is linked only to the $(n+1)$-th sensor, and noise is injected only into the $N$-th sensor. To simplify the analysis, we will also assume that $b_2=-1$ and $d_N=1.$

We investigate the asymptotic structure of the lead matrix in the second regime.
\begin{theorem}[First Nonzero Propagation Coefficient, Noise in Last Sensor Second Regime Lead Matrix]\label{thm: Cyclic OU Process First Nonzero Propagation Coefficient and Noise in Last Sensor Second Regime Lead Matrix Asymptotic Expansion}
Consider the lead matrix $\mathbf Q(\epsilon)$ in 
\eqref{eq: Cyclic OU Process One Nonzero Propagation Coefficient and Noise in One Sensor Lead Matrix}. Then, as $\epsilon \rightarrow 0,$ we have an alternate asymptotic expansion 
	\begin{align} \label{eq: Cyclic OU Process First Nonzero Propagation Coefficient and Noise in Last Sensor Second Regime Lead Matrix Asymptotic Expansion}
		\mathbf Q(\epsilon) &= \mathbf A + O(\epsilon),
	\end{align}
where $\mathbf A$ is the $N \times N$ skew-symmetric matrix whose $(j,k)$-th entry is
	\begin{align}\label{eq: Cyclic OU Process First Nonzero Propagation Coefficient and Noise in Last Sensor Second Regime Lead Matrix Asymptotic Binomial Matrix Entry}
		A_{j,k} &= \begin{cases}  \left(\frac{1}{2} \right)^{2N-j-k} \frac{k-j}{2N-j-k} \binom{2N-j-k}{N-k} &j,k \ne N \\ 0 & \text{else} \end{cases},
	\end{align}
and
    \begin{align}
        \binom{m}n &= \begin{cases}
    \frac{m!}{(m-n)! \  n!} & 0 \le n \le m \\
    0 & \text{else} \end{cases} \nonumber 
    \end{align}
is the binomial coefficient.
\end{theorem}

\begin{proof}
Substituting $b_2=-1$ and $d_N=1$ into \eqref{eq: Cyclic OU Process One Nonzero Propagation Coefficient and Noise in One Sensor Second Regime Lead Matrix Asymptotic Formula}, we get
\begin{align} \label{eq: Cyclic OU Process First Nonzero Propagation Coefficient and Noise in Last Sensor Second Regime Lead Matrix Asymptotic Formula}
	\mathbf Q(\epsilon) &= \frac{1}{N} \sum_{m,n=1}^N \left(\frac{(\omega_m - \omega_n) \ \omega_{m+n}}{\omega_m+ \omega_n-2} \right) \ \mathbf w_m  \mathbf w_n^\text T + O(\epsilon).
\end{align}

Extracting the $(j,k)$-th entry of \eqref{eq: Cyclic OU Process First Nonzero Propagation Coefficient and Noise in Last Sensor Second Regime Lead Matrix Asymptotic Formula}, we obtain
\begin{align}
	Q_{j,k}(\epsilon) &=  \frac{1}{N^2} \sum_{m,n=1}^N \frac{(\omega_m-\omega_n) \ \omega_{m+n}} {\omega_m+\omega_n-2} \  \omega_m^{j-1} \omega_n^{k-1}  + O(\epsilon) \nonumber \\
	&=-\frac{1}{2N^2} \sum_{m,n=1}^N \frac{\omega_m^{j+1} \omega_n^k - \omega_m^{j} \omega_n^{k+1}} {1- \frac{\omega_m+\omega_n}{2}}   + O(\epsilon) \label{eq: Cyclic OU Process First Nonzero Propagation Coefficient and Noise in Last Sensor Second Regime Lead Matrix Asymptotic Formula 1} \\
	&= -\frac{1}{2N^2} \sum_{r \ge 0}  \sum_{m,n=1}^N \left(\omega_m^{j+1} \omega_n^k - \omega_m^{j} \omega_n^{k+1} \right) \left(\frac{\omega_m+\omega_n}{2} \right)^r + O(\epsilon)  \label{eq: Cyclic OU Process First Nonzero Propagation Coefficient and Noise in Last Sensor Second Regime Lead Matrix Asymptotic Formula 2}.
\end{align}
Upon expanding the expression $\left( \frac{\omega_m + \omega_n}{2} \right)^r$  within the summand of \eqref{eq: Cyclic OU Process First Nonzero Propagation Coefficient and Noise in Last Sensor Second Regime Lead Matrix Asymptotic Formula 2} with the binomial theorem, we get
\begin{align}
	Q_{j,k}(\epsilon) &= -\frac{1}{2N^2}  \sum_{q,r \ge 0} \sum_{m,n=1}^N \left(\omega_m^{j+1} \omega_n^k - \omega_m^{j} \omega_n^{k+1} \right)  \ \left(\frac{1}2 \right)^r \  \binom{r}q  \  \omega_m^{r-q} \  \omega_n^q + O(\epsilon) \nonumber \\
	&= -\frac{1}{2N^2}  \sum_{q,r \ge 0} \sum_{m,n=1}^N  \ \left(\frac{1}2 \right)^r \  \binom{r}q  \left(\omega_m^{r+j-q+1} \  \omega_n^{k+q} -  \omega_m^{r+j-q} \  \omega_n^{k+q+1} \right)  \label{eq: Cyclic OU Process First Nonzero Propagation Coefficient and Noise in Last Sensor Second Regime Lead Matrix Asymptotic Formula 3} \\
	&\qquad \qquad \qquad  + O(\epsilon) \nonumber  \\
	&= -\frac{1}{2} \sum_{q,r \ge 0}  \begin{cases}
	\left(\frac{1}{2} \right)^r \binom{r}q & r \equiv q-j-1 \mod N \ , \ q \equiv -k \mod N  \\
	-\left(\frac{1}{2} \right)^r \binom{r}q & r \equiv q-j \mod N \ , \ q \equiv -k-1 \mod N  \\
	0 & \text{else}
	 \end{cases} \label{eq: Cyclic OU Process First Nonzero Propagation Coefficient and Noise in Last Sensor Second Regime Lead Matrix Asymptotic Formula 4} \\
	 & \qquad \qquad \qquad \qquad + O(\epsilon) \nonumber,
\end{align}
where \eqref{eq: Cyclic OU Process First Nonzero Propagation Coefficient and Noise in Last Sensor Second Regime Lead Matrix Asymptotic Formula 4} follows from evaluating the inner double sum within  \eqref{eq: Cyclic OU Process First Nonzero Propagation Coefficient and Noise in Last Sensor Second Regime Lead Matrix Asymptotic Formula 3} using the geometric series identity \eqref{eq: Finite Geometric Series Identity}.

Now, consider the system of congruence equations 
	\begin{align}
		r &\equiv q-j-1 \mod N \ , \ q \equiv -k \mod N \label{eq: Cyclic OU Process First Nonzero Propagation Coefficient and Noise in Last Sensor Second Regime Lead Matrix Asymptotic Formula 6} \\ 
            r &\equiv q-j \mod N \ , \ q \equiv -k-1 \mod N \label{eq: Cyclic OU Process First Nonzero Propagation Coefficient and Noise in Last Sensor Second Regime Lead Matrix Asymptotic Formula 7},
	\end{align}
where $r,q \ge 0$ and $1 \le j,k \le N.$ The second congruence equation in \eqref{eq: Cyclic OU Process First Nonzero Propagation Coefficient and Noise in Last Sensor Second Regime Lead Matrix Asymptotic Formula 6} implies $q = vN-k$ for some integer $v \in \N,$ while the first implies $r=uN-j-k-1$ for some integer $u \ge 2.$ Similarly, The congruence equations in \eqref{eq: Cyclic OU Process First Nonzero Propagation Coefficient and Noise in Last Sensor Second Regime Lead Matrix Asymptotic Formula 7} imply
$q= vN-k-1$ and $r= uN-j-k-1$ for some integers $v \in \N$ and $u \ge 2.$ Hence, we can rewrite \eqref{eq: Cyclic OU Process First Nonzero Propagation Coefficient and Noise in Last Sensor Second Regime Lead Matrix Asymptotic Formula 4} as a double sum over all $u \ge 2$ and $v \ge 1$ in the following way:
	\begin{align}
		Q_{j,k}(\epsilon) &= - \sum_{u \ge 2, v \ge 1} \left(\frac{1}{2} \right)^{uN-j-k} \left(\binom{uN-j-k-1}{vN-k-1} - \binom{uN-j-k-1}{vN-k} \right) \label{eq: Cyclic OU Process First Nonzero Propagation Coefficient and Noise in Last Sensor Second Regime Lead Matrix Asymptotic Formula 8} \\
		&\qquad \qquad \qquad + O(\epsilon) \nonumber \\
		&= \left(\frac{1}{2} \right)^{2N-j-k} \left(\binom{2N-j-k-1}{N-k} - \binom{2N-j-k-1}{N-k-1} \right)  + O(\epsilon) \label{eq: Cyclic OU Process First Nonzero Propagation Coefficient and Noise in Last Sensor Second Regime Lead Matrix Asymptotic Formula 9}   \\
		&= \left(\frac{1}{2} \right)^{2N-j-k} \frac{2N-j-k-2(N-k)}{2N-j-k} \binom{2N-j-k}{N-k} + O(\epsilon) \label{eq: Cyclic OU Process First Nonzero Propagation Coefficient and Noise in Last Sensor Second Regime Lead Matrix Asymptotic Formula 10} \\
		 &= \left(\frac{1}{2} \right)^{2N-j-k} \frac{k-j}{2N-j-k} \binom{2N-j-k}{N-k}  + O(\epsilon) \nonumber,
	\end{align}
in which \eqref{eq: Cyclic OU Process First Nonzero Propagation Coefficient and Noise in Last Sensor Second Regime Lead Matrix Asymptotic Formula 9} follows from extracting the term corresponding $u=2,v=1$ in \eqref{eq: Cyclic OU Process First Nonzero Propagation Coefficient and Noise in Last Sensor Second Regime Lead Matrix Asymptotic Formula 8} and replacing the remaining sum over all other $u,v$ with $O(\epsilon)$ and  \eqref{eq: Cyclic OU Process First Nonzero Propagation Coefficient and Noise in Last Sensor Second Regime Lead Matrix Asymptotic Formula 10} follows from rewriting \eqref{eq: Cyclic OU Process First Nonzero Propagation Coefficient and Noise in Last Sensor Second Regime Lead Matrix Asymptotic Formula 9} using the binomial coefficient identity \cite{Knuth1989}:
	\begin{align}
		\binom{n-1}{k} - \binom{n-1}{k-1} &= \frac{n-2k}{n} \binom{n}{k} \nonumber.
	\end{align}
This completes the proof.
\end{proof}

\subsection{Structure of Eigenvalues}
We investigate the spectrum of $\mathbf A,$ whose $(j,k)$-th entry is defined in \eqref{eq: Cyclic OU Process First Nonzero Propagation Coefficient and Noise in Last Sensor Second Regime Lead Matrix Asymptotic Binomial Matrix Entry}, for large $N$. To make things suitable for future calculations, we will consider the Hermitian matrix $i \mathbf A$ in place of the original $\mathbf A.$ Note that the eigenvalues of $i \mathbf A$ are the eigenvalues of $\mathbf A$ multiplied by $i.$ However, the eigenvectors of $i \mathbf A$ are the same as those of $\mathbf A.$ Furthermore, we  parameterize $i \mathbf A$ by its size $N.$ In the end, we define  $\mathbf A_N$ to be the $N \times N$ Hermitian matrix whose $(j,k)$-th entry is the quantity \eqref{eq: Cyclic OU Process First Nonzero Propagation Coefficient and Noise in Last Sensor Second Regime Lead Matrix Asymptotic Binomial Matrix Entry} multiplied by $i$ for all $1 \le j,k \le N.$ Throughout, we let $\lambda_n(\mathbf A_N)$ be the $n$-th largest eigenvalue of $\mathbf A_N$ and $\left \lbrace \mathbf v_n(\mathbf A_N) \right \rbrace_{n=1}^N$ be the orthonormal basis of eigenvectors such that $\mathbf v_n(\mathbf A_N)$ is the associated unit eigenvector of $\mathbf A_N$ with $\lambda_n(\mathbf A_N).$ Finally, we denote $A_{N,j,k}$ as the $(j,k)$-th entry of $\mathbf A_N.$ 

Recall the $j$-th Gershgorin disk of $\mathbf A_N$ is the disk in the complex plane with center $0$ and radius equal to 
\begin{align}\label{eq: Cyclic OU Process First Nonzero Propagation Coefficient and Noise in Last Sensor Second Regime Lead Matrix Asymptotic Binomial Matrix Gershgorin Radii}
	R_{N,j} &= \sum_{k=1}^N |A_{N,j,k}|.
\end{align}

We concentrate on the largest eigenvalue of $\mathbf A_N$ for large $N.$  By Gershgorin's Circle Theorem \cite{Varga2011}, all eigenvalues of $\mathbf A_N$ lie in the union of the Gershgorin disks.  But we conjecture an upper bound on the radii of these disks.

\begin{conjecture}[Gershgorin Radii Upper Bound]\label{conj: Cyclic OU Process First Nonzero Propagation Coefficient and Noise in Last Sensor Second Regime Lead Matrix Asymptotic Binomial Matrix Gershgorin Radii Upper Bound}
The Gershgorin radii of $\mathbf A_N$ satisfy
		\begin{align}
			\lim_{N \rightarrow \infty} R_{N,N-j} &= \begin{cases} \frac{\binom{2j}{j}}{2^{2j-1}} & j>0\\
				1 & j=0 \label{eq: Cyclic OU Process First Nonzero Propagation Coefficient and Noise in Last Sensor Second Regime Lead Matrix Asymptotic Binomial Matrix Gershgorin Radii Upper Bound}
			\end{cases} .
		\end{align}
\end{conjecture}

To see this conjecture holds for $j=1,$ we evaluate
	\begin{align}
		R_{N,N-1} &= \sum_{k=1}^N |A_{N,N-1,k}| \nonumber \\
		&= \sum_{k=1}^N  \left(\frac{1}{2} \right)^{N-k+1} |N-k-1| \label{eq: Cyclic OU Process First Nonzero Propagation Coefficient and Noise in Last Sensor Second Regime Lead Matrix Asymptotic Binomial Matrix Gershgorin Radii Upper Bound 1}\\
		&= \sum_{k=1}^{N-1}  \left(\frac{1}{2} \right)^{N-k+1} (N-k-1) + \frac{1}{2}  \label{eq: Cyclic OU Process First Nonzero Propagation Coefficient and Noise in Last Sensor Second Regime Lead Matrix Asymptotic Binomial Matrix Gershgorin Radii Upper Bound 2} \\ 
		&= \frac{1}{2} - \frac{N}{2^N} + \frac{1}{2}  \label{eq: Cyclic OU Process First Nonzero Propagation Coefficient and Noise in Last Sensor Second Regime Lead Matrix Asymptotic Binomial Matrix Gershgorin Radii Upper Bound 3} \\
		&= 1- \frac{N}{2^N} \label{eq: Cyclic OU Process First Nonzero Propagation Coefficient and Noise in Last Sensor Second Regime Lead Matrix Asymptotic Binomial Matrix Gershgorin Radii Upper Bound 4}
	\end{align}
in which we  obtained \eqref{eq: Cyclic OU Process First Nonzero Propagation Coefficient and Noise in Last Sensor Second Regime Lead Matrix Asymptotic Binomial Matrix Gershgorin Radii Upper Bound 2} by rewriting \eqref{eq: Cyclic OU Process First Nonzero Propagation Coefficient and Noise in Last Sensor Second Regime Lead Matrix Asymptotic Binomial Matrix Gershgorin Radii Upper Bound 1} after extracting the term $k=N$ and used the identity 
	\begin{align} \nonumber
		\sum_{k=1}^{N-1} k r^k &= \frac{N r^{N+1}-r^{N+1}-N r^N+r}{(r-1)^2}
	\end{align}
to obtain \eqref{eq: Cyclic OU Process First Nonzero Propagation Coefficient and Noise in Last Sensor Second Regime Lead Matrix Asymptotic Binomial Matrix Gershgorin Radii Upper Bound 3} from \eqref{eq: Cyclic OU Process First Nonzero Propagation Coefficient and Noise in Last Sensor Second Regime Lead Matrix Asymptotic Binomial Matrix Gershgorin Radii Upper Bound 2}. Taking the limit of \eqref{eq: Cyclic OU Process First Nonzero Propagation Coefficient and Noise in Last Sensor Second Regime Lead Matrix Asymptotic Binomial Matrix Gershgorin Radii Upper Bound 4} as $N \rightarrow \infty,$ we see the conjecture holds true for the case $j=1.$

Assuming \textbf{Conjecture \ref{conj: Cyclic OU Process First Nonzero Propagation Coefficient and Noise in Last Sensor Second Regime Lead Matrix Asymptotic Binomial Matrix Gershgorin Radii Upper Bound}}  holds, we can show the sequence of largest eigenvalues $\left \lbrace \lambda_1(\mathbf A_N) \right \rbrace_{N \in \N}$ converges.

\begin{theorem}[Asymptotic Binomial Matrix Largest Eigenvalue Convergence]\label{thm: Asymptotic Binomial Matrix Largest Eigenvalue Convergence}
	The sequence of largest eigenvalues $\left \lbrace \lambda_1(\mathbf A_N) \right \rbrace_{N \in \N}$ converges.
\end{theorem}

\begin{proof}
To show convergence, we show that the sequence $\left \lbrace \lambda_1( \mathbf A_N) \right \rbrace_{N \in \N}$ is increasing and bounded above in $\R.$

First, we show the sequence is increasing. By induction on $N,$ we can write $\mathbf A_N$ in the block form
	\begin{align}
		\mathbf A_N &= \begin{bmatrix}
			0 & \mathbf a_N^* \\  \mathbf  a_N & \mathbf A_{N-1}
		\end{bmatrix},
	\end{align}
	where $\mathbf a_N \in \C^{N-1}$ is the first column of $\mathbf A_N$ after excluding its first component $0$.  Notice that $\mathbf A_{N-1}$ is the first principal minor of $\mathbf A_N,$ that is, we obtain $\mathbf A_{N-1}$ by deleting the first row and first column of $\mathbf A_N.$  By the Cauchy Interlacing Theorem \cite{TaoDentonZhang2022}, the eigenvalues of $\mathbf A_{N-1}$ interlace those of $\mathbf A_N,$ which specifically means 
		\begin{align}\label{eq: Cauchy Interlacing Inequality}
			\lambda_N(\mathbf A_{N}) \le \lambda_{N-1}(\mathbf A_{N-1}) \le \lambda_{N-1}(\mathbf A_N) \le  \dots \ \le \lambda_{1}(\mathbf A_{N-1}) \le \lambda_1(\mathbf A_N)
		\end{align}
	The above inequality directly implies $\left \lbrace \lambda_1(\mathbf A_N) \right \rbrace_{ N \in \N}$ is an increasing sequence. 
	
Next, we show that $\left \lbrace \lambda_1(\mathbf A_N) \right \rbrace_{N \in \N}$ is a sequence that is bounded above. Consider $R_{N,N-j},$  the $(N-j)$-th Gershgorin disk radius of $\mathbf A_N.$ Assuming \textbf{Conjecture \ref{conj: Cyclic OU Process First Nonzero Propagation Coefficient and Noise in Last Sensor Second Regime Lead Matrix Asymptotic Binomial Matrix Gershgorin Radii Upper Bound}}  holds, for $j>0,$ we have
	\begin{align}
		\lim_{N \rightarrow \infty} R_{N,N-j} &= \frac{\binom{2j}{j}}{2^{2j-1}} \nonumber  \\
		&= \frac{ \frac{2j}{j} \binom{2(j-1)}{j-1}}{2^{2j-1}} \nonumber  \\
		&= \frac{\binom{2(j-1)}{j-1}}{2^{2j-2}} \nonumber \\
		&<\frac{\binom{2(j-1)}{j-1}}{2^{2(j-1)-1}} = \lim_{N \rightarrow \infty} R_{N,N-(j-1)} \nonumber.
	\end{align}
Hence, as $N \rightarrow \infty,$ we see the Gershgorin radii of $\mathbf A_N$ are bounded above by $1.$  This implies the sequence  $\left \lbrace \lambda_1(\mathbf A_N) \right \rbrace_{N \in \N}$ is bounded above by $1.$
\end{proof}

Now, we make a conjecture about the limit of the sequence $\left \lbrace \lambda_n(\mathbf A_N) \right \rbrace_{N \in \N}.$

\begin{conjecture}[Asymptotic Binomial Matrix Largest Eigenvalue Convergence]\label{conj: Cyclic OU Process First Nonzero Propagation Coefficient and Noise in Last Sensor Second Regime Lead Matrix Asymptotic Binomial Matrix Largest Eigenvalue Limit}
For small $n \in \N,$ we have
	\begin{align}
		\lim_{N \rightarrow \infty} \lambda_{2n-1}(\mathbf A_N) &= \frac{2}{(2n-1)\pi} .
	\end{align}
\end{conjecture}

We provide numerical evidence supporting \textbf{Conjecture \ref{conj: Cyclic OU Process First Nonzero Propagation Coefficient and Noise in Last Sensor Second Regime Lead Matrix Asymptotic Binomial Matrix Largest Eigenvalue Limit}}. In \textbf{Table \ref{tbl: Cyclic OU Process First Nonzero Propagation Coefficient and Noise in Last Sensor Second Regime Lead Matrix Asymptotic Binomial Matrix Largest Eigenvalue Limit}}, we tabulate the distance between $\lambda_1(\mathbf A_N)$ and $\frac{2}{\pi}$ for $20 \le n \le 24.$ We observe such distances are close to zero.

\begin{table}[h!]
	\centering
	\begin{tabular}{||c | c||} 
		\hline
		$N$ & $\left|\lambda_1(\mathbf A_N) - \frac{2}{\pi}  \right|$  \\ [0.5ex] 
		\hline\hline
		$20$ & $\approx 0.0000232513$  \\ 
		\hline
		$21$ & $\approx 0.0000158971$ \\
		\hline
		$22$ & $\approx 0.0000109305$ \\
		\hline
		$23$ & $\approx 7.55596 \times 10^{-6}$ \\ 
		\hline 
		$24$ & $\approx 5.24969 \times 10^{-6}$ \\
		\hline
	\end{tabular}

	\caption{We tabulate the distance between the largest eigenvalue of $\mathbf A_N$ and the conjectured limit as $N \rightarrow \infty$, namely $\frac{2}{\pi} ,$ for $20 \le N \le 24.$}
	\label{tbl: Cyclic OU Process First Nonzero Propagation Coefficient and Noise in Last Sensor Second Regime Lead Matrix Asymptotic Binomial Matrix Largest Eigenvalue Limit}
\end{table}

Now, we pose a conjecture about the largest eigenvalues corresponding to the principal minors of $\mathbf A_N.$ Recall the $j$-th principal minor of $\mathbf A_N$ is the $(N-1) \times (N-1)$ matrix $\widetilde{\mathbf A}_{N,j}$ formed by deleting the $j$-th row and $j$-th column of $\mathbf A_N.$

\begin{conjecture}[Asymptotic Binomial Matrix Principal Minor Largest Eigenvalues]\label{conj: Cyclic OU Process First Nonzero Propagation Coefficient and Noise in Last Sensor Second Regime Lead Matrix Asymptotic Binomial Matrix Principal Minor Largest Eigenvalues}
For each $1 \le j<N,$ we have
	\begin{align}
		\lambda_{1}(\widetilde{\mathbf A}_{N,j+1}) < \lambda_{1}(\widetilde{\mathbf A}_{N,j}) < \lambda_1(\mathbf A_N),
	\end{align}
where $\lambda_{n}(\widetilde{\mathbf A}_{N,j})$ is the $n$-th largest eigenvalue of $\widetilde{\mathbf A}_{N,j}.$
\end{conjecture}

We provide some numerical evidence supporting the conjecture. In \textbf{Table \ref{tbl: Cyclic OU Process First Nonzero Propagation Coefficient and Noise in Last Sensor Second Regime Lead Matrix Asymptotic Binomial Matrix Principal Minor Largest Eigenvalues}}, we tabulate the largest eigenvalues of the principal minors of $\mathbf A_{N}$ in the situation $N=10$. We observe as $j$ gets larger, the largest eigenvalue of the $j$-th principal minor precipitously decreases. 
\begin{table}[h!]
	\centering
	\begin{tabular}{||c | c||} 
		\hline 
		$j$ & $\lambda_1(\widetilde{\mathbf A}_{10,j})$  \\ [0.75ex] 
		\hline\hline
		$1$ & $\approx 0.634179$  \\ 
		\hline
		$2$ & $\approx 0.633523$ \\
		\hline 
		$3$ & $\approx 0.632329$ \\
		\hline 
		$4$ & $\approx 0.630125$ \\
		\hline
		$5$ & $\approx 0.626035$ \\ 
		\hline 
		$6$ & $\approx 0.618269$ \\
			\hline
		$7$ & $\approx 0.6018$ \\
			\hline
		$8$ & $\approx 0.555167$ \\
			\hline
		$9$ & $\approx 0.347293$ \\
			\hline
		$10$ & $\approx 0.298889$ \\
		\hline
	\end{tabular}
	
	\caption{We tabulate the largest eigenvalue for each principal minor of $\mathbf A_{10}.$}
	\label{tbl: Cyclic OU Process First Nonzero Propagation Coefficient and Noise in Last Sensor Second Regime Lead Matrix Asymptotic Binomial Matrix Principal Minor Largest Eigenvalues}
\end{table}

\subsection{Structure of Leading Eigenvector}

For large $N,$ we now investigate the structure of $\mathbf v_1(\mathbf A_N),$ the leading eigenvector of $\mathbf A_N,$ assuming all the previous conjectures we stated pertaining to the spectrum of $\mathbf A_N$ hold.

First, we describe the behavior of $|v_{1,n}(\mathbf A_N)|,$ the modulus of the $n$-th component $\mathbf v_1(\mathbf A_N).$

\begin{theorem}[Asymptotic Binomial Matrix Leading Eigenvector Component Moduli Structure]\label{conj: Cyclic OU Process First Nonzero Propagation Coefficient and Noise in Last Sensor Second Regime Lead Matrix Asymptotic Binomial Matrix  Leading Eigenvector Component Moduli Structure}
For each $1 \le n <N,$ we have
	\begin{align}
		|v_{1,n}(\mathbf A_N)| < |v_{1,n+1}(\mathbf A_N)|.
	\end{align}
Moreover, we have
	\begin{align}
		\lim_{N \rightarrow \infty} v_{1,1}(\mathbf A_N) &=0.
	\end{align}
\end{theorem}

\begin{proof}
We first prove the first statement. We utilize the \textit{eigenvalue-eigenvector identity} \cite{TaoDentonZhang2022} which states
	\begin{align}\label{eq: Eigenvalue Eigenvector Identity}
		\left|v_{1,j}(\mathbf A_N) \right|^2 &=  \frac{\prod_{n=1}^{N-1} \left(\lambda_1(\mathbf A_N) -\lambda_n(\widetilde{\mathbf A}_{N,j}) \right) }{\prod_{n=2}^N \left(\lambda_1(\mathbf A_N) - \lambda_n(\mathbf A_N) \right)}.
	\end{align}
 
Assuming \textbf{Conjecture \ref{conj: Cyclic OU Process First Nonzero Propagation Coefficient and Noise in Last Sensor Second Regime Lead Matrix Asymptotic Binomial Matrix Principal Minor Largest Eigenvalues}} holds, we have
	\begin{align}
		\lambda_1(\mathbf A_N) -\lambda_n(\widetilde{\mathbf A}_{N,j}) &< \lambda_1(\mathbf A_N) -\lambda_n(\widetilde{\mathbf A}_{N,j+1}) \nonumber.
	\end{align}
This means we can bound the numerator on the right hand side of \eqref{eq: Eigenvalue Eigenvector Identity} in the following way:
	\begin{align}
	\left|v_{1,j}(\mathbf A_N) \right|^2 &=  \frac{\prod_{n=1}^{N-1} \left(\lambda_1(\mathbf A_N) -\lambda_n(\widetilde{\mathbf A}_{N,j}) \right) }{\prod_{n=2}^N \left(\lambda_1(\mathbf A_N) - \lambda_n(\mathbf A_N) \right)} \nonumber \\
	&< \frac{\prod_{n=1}^{N-1} \left(\lambda_1(\mathbf A_N) -\lambda_n(\widetilde{\mathbf A}_{N,j+1}) \right) }{\prod_{n=2}^N \left(\lambda_1(\mathbf A_N) - \lambda_n(\mathbf A_N) \right)} = \left|v_{1,j+1}(\mathbf A_N) \right|^2 \nonumber,
	\end{align}
which proves the first statement.

Now, we prove the second statement, recall that $\lambda_n(\mathbf A_N) \le \lambda_n(\mathbf A_{N-1})$ by the Cauchy Interlacing inequality. We have
	\begin{align}
		 |v_{1,1}(\mathbf A_N)|^2 &=\frac{\prod_{n=1}^{N-1} \left(\lambda_1(\mathbf A_N) -\lambda_n(\widetilde{\mathbf A}_{N,1}) \right) }{\prod_{n=2}^N \left(\lambda_1(\mathbf A_N) - \lambda_n(\mathbf A_N) \right)} \nonumber \\
		 &=\frac{\prod_{n=1}^{N-1} \left(\lambda_1(\mathbf A_N) -\lambda_n(\mathbf A_{N-1}) \right) }{\prod_{n=2}^N \left(\lambda_1(\mathbf A_N) - \lambda_n(\mathbf A_N) \right)} \nonumber \\
		 &=  \frac{(\lambda_1(\mathbf A_N) - \lambda_1(\mathbf A_{N-1})) \nonumber \ \prod_{n=2}^{N-1} \left(\lambda_1(\mathbf A_N) -\lambda_n(\mathbf A_{N-1}) \right) }{\prod_{n=2}^N \left(\lambda_1(\mathbf A_N) - \lambda_n(\mathbf A_N) \right)} \nonumber \\
		 &\le \frac{(\lambda_1(\mathbf A_N) - \lambda_1(\mathbf A_{N-1})) \ \prod_{n=2}^{N-1} \left(\lambda_1(\mathbf A_N) -\lambda_n(\mathbf A_{N}) \right) }{\prod_{n=2}^N \left(\lambda_1(\mathbf A_N) - \lambda_n(\mathbf A_N) \right)}  \nonumber \\
		 &= \frac{\lambda_1(\mathbf A_N) - \lambda_1(\mathbf A_{N-1})}{\lambda_1(\mathbf A_N) - \lambda_N(\mathbf A_N)} \nonumber \\
		 &=  \frac{\lambda_1(\mathbf A_N) - \lambda_1(\mathbf A_{N-1})}{2\lambda_1(\mathbf A_N)} \nonumber,
	\end{align}
which approaches $0$ as $N \rightarrow \infty$ by \textbf{Theorem \ref{thm: Asymptotic Binomial Matrix Largest Eigenvalue Convergence}}. This proves the second statement.
\end{proof}

We interpret the previous theorem. The previous theorem states for large $N$, the component moduli $\mathbf v_1(\mathbf A_N)$ decrease from right to left. We interpret this in the context of our signal propagation model. As the signal propagates throughout the network, the left most sensors are less receptive to the signal than the right most sensors.

Now, we pose a conjecture about the component phases of $\mathbf v_1(\mathbf A_N).$

\begin{conjecture}[Asymptotic Binomial Matrix Leading Eigenvector Component Phase Structure]\label{conj: Cyclic OU Process First Nonzero Propagation Coefficient and Noise in Last Sensor Second Regime Lead Matrix Asymptotic Binomial Matrix  Leading Eigenvector Component Phase Structure}
For large $N \in \N,$ there exists large $K \le N$ for which
 $$0=\text{Arg}(v_{1,N}(\mathbf A_N)) < \ \dots \ <  \text{Arg}(v_{1,K}(\mathbf A_N))<\pi.$$
\end{conjecture}

In \textbf{Table \ref{tbl: Cyclic OU Process First Nonzero Propagation Coefficient and Noise in Last Sensor Second Regime Lead Matrix Asymptotic Binomial Matrix Component Phases}}, we show numerical evidence of \textbf{Conjecture \ref{conj: Cyclic OU Process First Nonzero Propagation Coefficient and Noise in Last Sensor Second Regime Lead Matrix Asymptotic Binomial Matrix  Leading Eigenvector Component Phase Structure}} when $N=50.$  We observe the last $6$ component phases are in increasing order, supporting the conjecture. But at the current time of writing, we do not have a general conjecture pertaining to the behavior of the other phases.

We interpret the result of the conjecture. Suppose the signal in our propagation model were to be broadcast from the last sensor, namely the $N$-th sensor. Then, \textbf{Conjecture \ref{conj: Cyclic OU Process First Nonzero Propagation Coefficient and Noise in Last Sensor Second Regime Lead Matrix Asymptotic Binomial Matrix  Leading Eigenvector Component Phase Structure}} states the order of the last few eigenvector component phases reflects the expected order of the next few sensors that receive the signal. Therefore, Cyclicity Analysis enables us to partially recover the expected network structure.

\begin{table}[h!]
	\centering
	\begin{tabular}{||c | c||} 
		\hline 
		$n$ & $\text{Arg}(v_{1,n}(\mathbf A_{50}))$  \\ [0.75ex] 
		\hline\hline
		$1$ & $\approx 2.11988$  \\ 
		\hline
		$2$ & $\approx 2.04186$ \\
		\hline 
		$3$ & $\approx 1.96213$ \\
		\hline 
		$44$ & $\approx -3.09417$ \\
		\hline
		$45$ & $\approx 2.93144$ \\ 
		\hline 
		$46$ & $\approx 2.6331$ \\
		\hline
		$47$ & $\approx 2.27014$ \\
		\hline
		$48$ & $\approx 1.83497$ \\
		\hline
		$49$ & $\approx 1.38311$ \\
		\hline
		$50$ & $0$ \\
		\hline
	\end{tabular}
	
	\caption{We tabulate some component phases of $\mathbf v_1(\mathbf A_{50}).$}
	\label{tbl: Cyclic OU Process First Nonzero Propagation Coefficient and Noise in Last Sensor Second Regime Lead Matrix Asymptotic Binomial Matrix Component Phases}
\end{table}

\section{Only One Nonzero Propagation Coefficient and Independent and Identical Noise Injected into Few Sensors}

Next, let $\mathbf B(\epsilon)$ be the friction matrix as defined in the beginning of \textbf{Section \ref{sec: Only One Nonzero Propagation Coefficient and Injection of Noise into Only One Sensor}}. But this time, let $\boldsymbol \Sigma$ be a volatility matrix such that the diffusion matrix $\mathbf D$ is a diagonal diffusion matrix containing multiple nonzero diagonal entries that are equal. In particular, we let the last $L$ diagonal entries of $\mathbf D$ be positive and equal i.e. we assume $D_{N-L+1,N-L+1} = \ \dots \ =  D_{N,N}=d$ for some $d>0,$ while all other diagonal entries are $0.$  This corresponds to the situation we inject independent and identical noise into the last $L$ sensors. 

We explicitly derive the lead matrix with the assumptions imposed on $\mathbf B(\epsilon)$ and $\boldsymbol \Sigma.$

\begin{theorem}[Cyclic OU Process One Nonzero Propagation Coefficient and Independent and Identical Noise in Multiple Sensors Lead Matrix Explicit Formula]\label{thm: Cyclic OU Process One Nonzero Propagation Coefficient and Independent and Identical Noise in Multiple Sensors Lead Matrix Explicit Formula}
For each $\epsilon>0,$ the lead matrix $\mathbf Q(\epsilon)$ is explicitly  
	\begin{align}
			\mathbf Q(\epsilon) &= \frac{b_p d}N \sum_{m,n=1}^N \sum_{\ell=1}^L \frac{(\omega_m^{p-1} - \omega_n^{p-1}) \ \omega_{m+n}^{\ell}}{2 \epsilon + b_p (\omega_m^{p-1}+ \omega_n^{p-1}-2)} \ \mathbf w_m \mathbf w_n^\text T \label{eq: Cyclic OU Process One Nonzero Propagation Coefficient and Independent and Identical Noise in Multiple Sensors Lead Matrix}
	\end{align}
	
\end{theorem}

\begin{proof}
Upon substituting $b_1=\epsilon-b_p$ and $b_q=0$ for all $q \ne 1,p$  into the right hand side of the Cyclic OU lead matrix formula in \eqref{eq:  Cyclic OU Process One Nonzero Propagation Coefficient and Noise in One Sensor Lead Matrix} and letting $\mathbf e_n$ be the $n$-th canonical basis vector in $\C^N,$ we have
	\begin{align}
		\mathbf Q(\epsilon) &=  \sum_{m,n=1}^N \frac{b_p (\omega_m^{p-1}- \omega_n^{p-1})}{2\epsilon  + b_p(\omega_m^{p-1}+ \omega_n^{p-1}-2)} \  \mathbf w_m \ \mathbf w_{N-m}^\text T \  \mathbf D \ \mathbf w_{N-n} \ \mathbf w_n^\text T \nonumber \\
		&=  \sum_{m,n=1}^N \frac{b_p (\omega_m^{p-1}- \omega_n^{p-1})}{2\epsilon  + b_p(\omega_m^{p-1}+ \omega_n^{p-1}-2)} \  \mathbf w_m \ \mathbf w_{N-m}^\text T \ \left (\frac{d \sum_{\ell=1}^L  \omega_{N-n}^{(N-\ell+1)-1} \   \mathbf e_{s_\ell}}{\sqrt N} \right) \ \mathbf w_n^\text T \nonumber \\
		&=\frac{b_p d}{N}  \sum_{m,n=1}^N \frac{b_p (\omega_m^{p-1}- \omega_n^{p-1})}{2\epsilon  + b_p(\omega_m^{p-1}+ \omega_n^{p-1}-2)} \  \mathbf w_m \  \ \left ( \sum_{\ell=1}^L \omega_{N-m-n}^{-\ell}  \right) \ \mathbf w_n^\text T \nonumber \\
		&= \frac{b_p d}N \sum_{m,n=1}^N \sum_{\ell=1}^L \frac{(\omega_m^{p-1} - \omega_n^{p-1}) \ \omega_{m+n}^{\ell}}{2 \epsilon + b_p (\omega_m^{p-1}+ \omega_n^{p-1}-2)} \ \mathbf w_m \mathbf w_n^\text T \nonumber.
	\end{align}
This proves the claim.
\end{proof}

We now investigate the structure of the leading eigenvector of $\mathbf Q(\epsilon)$ under the two regimes. 

\subsection{First Regime}
We begin with the first regime when $\epsilon>0$ is large.

\begin{theorem}[One Nonzero Propagation Coefficient and Independent and Identical Noise in Multiple Sensors First Regime Lead Matrix Asymptotic Expansion]\label{thm: Cyclic OU Process One Nonzero Propagation Coefficient and Independent and Identical Noise in Multiple Sensors First Regime Lead Matrix Asymptotic Formula}
Consider the lead matrix $\mathbf Q(\epsilon)$ defined in \eqref{eq: Cyclic OU Process One Nonzero Propagation Coefficient and Independent and Identical Noise in Multiple Sensors Lead Matrix}. Then, as $\epsilon \rightarrow \infty,$  we have the following asymptotic expansion: 
	\begin{align} \label{eq: Cyclic OU Process One Nonzero Propagation Coefficient and Independent and Identical Noise in Multiple Sensors First Regime Lead Matrix Asymptotic Formula}
		\mathbf Q(\epsilon) &= \mathbf A(\epsilon) + O \left(\frac{1}{\epsilon^2} \right),
	\end{align}
where $\mathbf A(\epsilon)$ is the $N \times N$ skew symmetric matrix whose $(j,k)$-th entry is of the form
	\begin{align} 
		A_{j,k}(\epsilon) &= \sum_{\ell=1}^L  \begin{cases} \frac{b_p d}{2 \epsilon} & j \equiv \ell + 1- p \mod N \ , \ k = \ell \\
		-\frac{b_p d}{2 \epsilon} &   j = \ell  \ , \ k \equiv \ell + 1- p \mod N  \\
		0& \text{else}
		 \end{cases}.  \label{eq: Cyclic OU Process One Nonzero Propagation Coefficient and Independent and Identical Noise in Multiple Sensors First Regime Lead Matrix Entry Asymptotic Formula}
	\end{align}
\end{theorem}

\begin{proof}
Note that $\mathbf Q(\epsilon)$ is the sum of lead matrices of the form \eqref{eq: Cyclic OU Process One Nonzero Propagation Coefficient and Noise in One Sensor Lead Matrix} with $s=N-\ell+1$ for each $1 \le \ell \le L.$ Therefore, we substitute $s=N-\ell+1$ into the asymptotic formula \eqref{eq: Cyclic OU Process One Nonzero Propagation Coefficient and Noise in One Sensor First Regime Lead Matrix Asymptotic Formula}, and the result immediately follows upon summing the $L$ total asymptotic formulas corresponding to each $1 \le \ell \le L.$
\end{proof}

Under the first regime, we investigate the leading eigenvector $\mathbf v_1(\mathbf Q(\epsilon))$ as $\epsilon \rightarrow \infty.$

At the current time of writing, we are not aware of any explicit formulas for the eigenvalues and eigenvectors of
\eqref{eq: Cyclic OU Process One Nonzero Propagation Coefficient and Independent and Identical Noise in Multiple Sensors First Regime Lead Matrix Asymptotic Formula} for general $p.$ So we  concentrate on the case $p=2$ i.e. only the propagation coefficient $b_2$ is nonzero. For simplicity, we let $b_2=-1$ and $d=1.$

\begin{theorem}[First Nonzero Propagation Coefficient and Independent and Identical Noise in Multiple Sensors First Regime Leading Eigenvector Asymptotic Structure]\label{thm: First Nonzero Propagation Coefficient and Independent and Identical Noise in Multiple Sensors First Regime Leading Eigenvector Asymptotic Structure}
If $b_2=-1$ and $d=1,$ then the lead matrix $\mathbf Q(\epsilon)$ in \eqref{eq: Cyclic OU Process One Nonzero Propagation Coefficient and Independent and Identical Noise in Multiple Sensors Lead Matrix} satisfies 
	\begin{align}
		\lim_{\epsilon \rightarrow \infty} \mathbf v_1(\mathbf Q(\epsilon)) &= \frac{\left( \underbrace{0 \ , \ \dots \  ,  \ 0}_{\text{$N-L-1$ times}} \ , \  \left( -1 \right)^{\frac{1}{2}} \sin \left(\frac{(1)  \pi}{L+2}\right) \ , \ \dots \ , \  \left( -1 \right)^{\frac{L+1}{2}} \sin \left(\frac{(L+1) \pi}{L+2}\right)    \right)}{\sqrt{\sum_{\ell=1}^{L+1} \sin^2 \left( \frac{\ell \pi}{L+2} \right) }}
	\end{align}
\end{theorem}
 
\begin{proof}
	
Let $\mathbf A(\epsilon)$ be the $N \times N$ skew-symmetric matrix whose $(j,k)$-th entry is defined in \eqref{eq: Cyclic OU Process One Nonzero Propagation Coefficient and Independent and Identical Noise in Multiple Sensors First Regime Lead Matrix Entry Asymptotic Formula} for each $1 \le j, k \le N.$ Since 
	\begin{align}
		\| \mathbf A(\epsilon) - \mathbf Q(\epsilon) \|_F = O \left(\frac{1}{\epsilon^2} \right),
	\end{align}
we see that $\mathbf A(\epsilon)$ is a low rank approximation of $\mathbf Q(\epsilon)$ as $\epsilon \rightarrow \infty.$

We seek the leading eigenvector of $\mathbf A(\epsilon).$ Note that $\mathbf A(\epsilon)$ can be written in the block form
	\begin{align}
		\mathbf A(\epsilon) &= \begin{bmatrix}
			\mathbf 0 & \mathbf 0 \\
			\mathbf 0 & \mathbf T(\epsilon)
		\end{bmatrix},
	\end{align}
in which $\mathbf T(\epsilon)$ is the $(L+1) \times (L+1)$ matrix such that $T_{\ell, \ell+1}(\epsilon)  = -T_{\ell+1,\ell}=- \frac{b_2 d}{2 \epsilon}$ for each $1 \le \ell \le L$ and all other entries are $0.$ 

We note that $\mathbf T(\epsilon)$ is a tridiagonal Toeplitz matrix, whose eigenvalues and eigenvectors are known \cite{NoschesePasquiniReichel2013}. In particular, if $\tau=T_{\ell,\ell+1}, \sigma= T_{\ell+1,\ell}$ for each $1 \le \ell \le L$  and $\delta=T_{\ell,\ell}$ for all $1 \le \ell \le L+1.$ Then, 
\begin{align}
		\widetilde{\lambda}_\ell &= \delta +2 \sqrt{\sigma \tau } \cos \left(\frac{\ell \pi}{L+2} \right) 
\end{align}
is an eigenvalue of $\mathbf T(\epsilon)$ with associated eigenvector 
\begin{align}
		\widetilde{\mathbf v}_\ell &= \left(\left(\frac{\sigma}{\tau} \right)^{\frac{1}{2}}  \sin \left(\frac{(1)\ell \pi}{L+2} \right) \ , \ \dots \ , \ \left(\frac{\sigma}{\tau} \right)^{\frac{L}{2}}  \sin \left(\frac{(L) \ell \pi}{L+2} \right) \right).
\end{align} 

Therefore, upon substituting $\tau=\frac{1}{2 \epsilon}$ and $\sigma=-\tau$ and $\delta=0,$ the eigenvalues of the bottom right $(L+1) \times (L+1)$ submatrix $\mathbf T(\epsilon)$ are of the form
	\begin{align}
		\widetilde{\lambda}_\ell &= -\frac{1}{\epsilon}  \cos\left(\frac{\ell \pi}{L+2} \right) \ i
	\end{align} 
for each $1 \le \ell \le L+1$ with corresponding eigenvector
	\begin{align}
		\widetilde{\mathbf v}_\ell &= \left( \left( -1 \right)^{\frac{1}{2}} \sin \left(\frac{(1) \ell \pi}{L+2}\right) \ , \ \dots \ , \  \left( -1 \right)^{\frac{L+1}{2}} \sin \left(\frac{(L+1) \ell \pi}{L+2}\right)   \right).
	\end{align}
Therefore, the eigenvalues of $\mathbf A(\epsilon)$ are the eigenvalues of $\mathbf T(\epsilon)$ and $0.$ In particular, the largest eigenvalue of $\mathbf A(\epsilon)$ is $\widetilde{\lambda}_{1}$ with associated leading eigenvector being the vector formed by prepending $\widetilde{\mathbf v}_1$ with $N-L-1$ zeros. Hence, upon normalization, we have 
	\begin{align}
		\mathbf v_1(\mathbf A(\epsilon)) &= \frac{\left(\underbrace{0 \ , \ \dots \  ,  \ 0}_{\text{$N-L-1$ times}}  \ , \  \left( -1 \right)^{\frac{1}{2}} \sin \left(\frac{(1)  \pi}{L+2}\right) \ , \ \dots \ , \  \left( -1 \right)^{\frac{L+1}{2}} \sin \left(\frac{(L+1) \pi}{L+2}\right)  \right)}{\sqrt{\sum_{\ell=1}^{L+1} \sin^2 \left( \frac{\ell \pi}{L+2} \right) }},
	\end{align}
which proves the claim.
\end{proof}

We interpret the result given by \textbf{Theorem \ref{thm: First Nonzero Propagation Coefficient and Independent and Identical Noise in Multiple Sensors First Regime Leading Eigenvector Asymptotic Structure}}. The last $L+1$ leading eigenvector components are nonzero, suggesting the propagating signal in our model is received by the $L+1$ sensors. However, because the last $L+1$ components either lie on the real or imaginary axes the leading eigenvector component phases are either $0, \pi,$ or $\pm \frac{\pi}{2}.$ So we observe the cyclic order of the component phases does not match the expected cyclic order of sensors receiving the signal. Hence, Cyclicity Analysis does not enable us to fully recover the network structure.

However, we notice the last $L+1$ leading eigenvector component moduli form a sinusoidal pattern, which is interesting. This suggests if the signal were to be broadcast from the $N$-th sensor, then one of the later sensors is most receptive to the signal.

\subsection{Second Regime}

We investigate the second regime, in which the perturbation  $\epsilon>0$ is small.

\begin{theorem}[One Nonzero Propagation Coefficient and Independent and Identical Noise in Multiple Sensors Second Regime Lead Matrix Asymptotic Expansion]\label{thm: Cyclic One Nonzero Propagation Coefficient and Independent and Identical Noise in Multiple Sensors Second Regime Lead Matrix Asymptotic Expansion}
Consider the lead matrix $\mathbf Q(\epsilon)$ in \eqref{eq: Cyclic OU Process One Nonzero Propagation Coefficient and Independent and Identical Noise in Multiple Sensors Lead Matrix}.  Then, as $\epsilon \rightarrow 0,$ the matrix $\mathbf Q(\epsilon)$ has an asymptotic expansion of the form
	\begin{align}\label{eq: Cyclic OU Process One Nonzero Propagation Coefficient and Independent and Identical Noise in Multiple Sensors Second Regime Lead Matrix Asymptotic Formula} 
		\mathbf Q(\epsilon)  &= \frac{d}{N} \sum_{m,n=1}^N \sum_{\ell =1}^L \left(\frac{\left(\omega_m^{p-1}-\omega_n^{p-1} \right)  \  \omega_{m+n}^{1-\ell}} {\omega_m^{p-1}+\omega_n^{p-1}-2} \right) \ \mathbf w_m  \mathbf w_n^\text T + O(\epsilon),
	\end{align}
	in which the summand inside the double sum is defined to be $0$ if $m=n=N.$
\end{theorem}

We omit the proof, as it involves the exact same reasoning that was used to prove  \textbf{Theorem \ref{thm: Cyclic OU Process One Nonzero Propagation Coefficient and Noise in One Sensor Second Regime Lead Matrix Asymptotic Formula}}. At the current time of writing, we are not aware of any closed formulas for the eigenvalues and eigenvectors of the matrix on the right hand side of \eqref{eq: Cyclic OU Process One Nonzero Propagation Coefficient and Independent and Identical Noise in Multiple Sensors Second Regime Lead Matrix Asymptotic Formula}. Nevertheless, we pose a conjecture in the situation where $p=2$ and $d=1.$

\begin{conjecture}[First Nonzero Propagation Coefficient and Independent and Identical Noise in Multiple Sensors Second Regime Asymptotic Leading Eigenvector Structure]\label{conj: One Nonzero Propagation Coefficient and Independent and Identical Noise in Multiple Sensors Second Regime Leading Eigenvector Structure}
Let $\mathbf A$ be the $N \times N$ matrix on the right hand side of   \eqref{eq: Cyclic OU Process One Nonzero Propagation Coefficient and Independent and Identical Noise in Multiple Sensors Second Regime Lead Matrix Asymptotic Formula}. If $\mathbf v_1(\mathbf A)$ is the leading eigenvector of $\mathbf A,$ then we have 
	\begin{align}
		\text{Arg}(v_{1,N}(\mathbf A) ) < & \ \dots \ < \text{Arg}(v_{1,N-L+1}(\mathbf A) )=0.
	\end{align}
and $|v_{1,\ell}(\mathbf A)| < |v_{1,\ell+1}(\mathbf A)|$ for each $1 \le \ell \le N-L$ and $|v_{1,\ell}(\mathbf A)| > |v_{1,\ell+1}(\mathbf A)|$ for each $N-L \le \ell <N-1.$
\end{conjecture}

We interpret this conjecture. In the situation where we inject independent and identical noise into $L$ nodes, the order of the few $L$ leading eigenvector component phases matches the expected order of those last few $L$ sensors receiving the signal in our propagation model. In addition, even though we inject noise into the last $L$ sensors, the $(N-L)$-th sensor is most receptive to the propagating signal, which is interesting.

In \textbf{Table \ref{tbl: One Nonzero Propagation Coefficient and Independent and Identical Noise in Multiple Sensors Second Regime Leading Eigenvector Structure}}, we provide evidence supporting \textbf{Conjecture \ref{conj: One Nonzero Propagation Coefficient and Independent and Identical Noise in Multiple Sensors Second Regime Leading Eigenvector Structure}} in the situation where $N=10$ and $L=4.$ In particular, we observe the $7$-th component has the largest modulus. Moreover, we observe the last $3$ component phases are increasing.

\begin{table}[h!]
	\centering
	\begin{tabular}{||c | c | c ||} 
		\hline 
		$n$ & $|v_{1,n}(\mathbf A)|$ & $\text{Arg}(v_{1,n}(\mathbf A))$  \\ [0.75ex] 
		\hline\hline
		$1$ & $\approx 0.084589$ &$\approx 2.97575$  \\ 
			\hline
		$2$ & $\approx 0.115083$ &$\approx 2.61363$  \\ 
			\hline
		$3$ & $\approx 0.153119$ &$\approx 2.22315$  \\ 
			\hline
		$4$ & $\approx 0.205092$ &$\approx 1.80388$  \\ 
			\hline
		$5$ & $\approx 0.283689$ &$\approx 1.3701$  \\ 
			\hline
		$6$ & $\approx 0.409898$ &$\approx 0.94586$ \\
			\hline 
		$7$ & $\approx 0.52438$ &$0$ \\
			\hline
		$8$ & $\approx 0.466637$ &$\approx -0.859917$ \\
			\hline
		$9$ & $\approx 0.354551$ &$\approx -1.75523$ \\
			\hline
		$10$ & $\approx 0.217176$ &$\approx -2.73113$ \\
		\hline 
	\end{tabular}
	\caption{Leading eigenvector component moduli and phases for the matrix $\mathbf A$ defined in \textbf{Conjecture \ref{conj: One Nonzero Propagation Coefficient and Independent and Identical Noise in Multiple Sensors Second Regime Leading Eigenvector Structure}} in the situation where $N=10$ and $L=4.$ }
	\label{tbl: One Nonzero Propagation Coefficient and Independent and Identical Noise in Multiple Sensors Second Regime Leading Eigenvector Structure}
\end{table}

\section{Independent and Identical Noise Injected into All Sensors}

We now investigate the signal propagation model under the assumptions the propagation coefficient $b_p$ is nonzero for exactly one index $1<p \le N$ with $\gcd(p-1,N)=1$ and $\boldsymbol \Sigma$ is a volatility matrix such that the diffusion matrix $\mathbf D$ is proportional to the $N \times N$ identity matrix. This choice of $\boldsymbol \Sigma$ corresponds to the situation where we inject independent and identical noise into all $N$ sensors. Since the diagonal entries of $\mathbf D$ are all the same, we may assume without loss of generality that $\mathbf D$ is the $N \times N$ identity matrix itself.

We first derive the lead matrix with the imposed assumptions on these cyclic OU model parameters.

\begin{theorem}[Cyclic OU Process One Nonzero Propagation Coefficient and Independent and Identical Noise in All Sensors Lead Matrix Explicit Formula]\label{thm: Cyclic OU Process One Nonzero Propagation Coefficient and Independent and Identical Noise in All Sensors Lead Matrix Explicit Formula}
For each $\epsilon>0,$ with the assumptions imposed on $\mathbf B(\epsilon)$ and $\boldsymbol \Sigma,$ the lead matrix $\mathbf Q(\epsilon)$ is explicitly 
	\begin{align}\label{eq: Cyclic OU Process One Nonzero Propagation Coefficient and Independent and Identical Noise in All Sensors Lead Matrix}
		\mathbf Q(\epsilon) &= -i \  \sum_{n=1}^N \frac{b_p \sin \left(\frac{2 \pi n (p-1)}{N} \right)}{ \epsilon -2 b_p \sin^2 \left(\frac{\pi n (p-1)}{N} \right)} \ \mathbf w_{N-n} \mathbf w_n^\text T.
	\end{align}
Furthermore, $\mathbf Q(\epsilon)$ is a circulant matrix.
\end{theorem}

\begin{proof}
We substitute the $N \times N$ identity matrix for $\mathbf D$ and $b_1=\epsilon -b_p$ and $b_q=0$ for all $q \ne 1,p$ into the general Cyclic OU lead matrix formula \eqref{eq: Cyclic OU Process Lead Matrix}. Recalling $\left \lbrace \mathbf w_n \right \rbrace_{n=1}^N$ is an orthonormal basis of $\C^N,$ we can simplify the summand in \eqref{eq: Cyclic OU Process Lead Matrix} in the following way:
	\begin{align}
		\mathbf Q(\epsilon)  &=  \sum_{m,n=1}^N \frac{b_p(\omega_{m}^{p-1}  - \omega_n^{p-1} )}{2 \epsilon + b_p(\omega_{m}^{p-1}  + \omega_n^{p-1} -2)} \  \mathbf w_m \ \mathbf w_{N-m}^\text T \mathbf w_{N-n} \mathbf w_n^\text T \nonumber  \\
		&=   \sum_{m,n=1}^N \frac{b_p(\omega_{m}^{p-1}  - \omega_n^{p-1} )}{2 \epsilon + b_p(\omega_{m}^{p-1}  + \omega_n^{p-1} -2)} \  \mathbf w_m \ \delta_{m, N-n} \mathbf w_n^\text T  \nonumber   \\
		&=   \sum_{n=1}^N \frac{b_p(\omega_{N-n}^{p-1}  - \omega_n^{p-1} )}{2 \epsilon + b_p(\omega_{N-n}^{p-1}  + \omega_n^{p-1} -2)} \  \mathbf w_{N-n} \mathbf w_n^\text T \nonumber \\
		&=    \sum_{n=1}^N \frac{b_p(\omega_{-n}^{p-1}  - \omega_n^{p-1} )}{2 \epsilon + b_p(\omega_{-n}^{p-1}  + \omega_n^{p-1} -2)} \  \mathbf w_{N-n} \mathbf w_n^\text T \label{eq: Cyclic OU Process One Nonzero Propagation Coefficient and Independent and Identical Noise in All Sensors Lead Matrix 1}.
	\end{align}

Now, recall the Euler identities
\begin{align}
	\sin(\theta) &= \frac{e^{i \theta} - e^{- i \theta}}{2i} \nonumber \\
	\cos(\theta) &= \frac{e^{i \theta} + e^{- i \theta}}{2} \nonumber.
\end{align}
Using the Euler identities, we rewrite the numerator and denominator of the summand in \eqref{eq: Cyclic OU Process One Nonzero Propagation Coefficient and Independent and Identical Noise in All Sensors Lead Matrix 1} to obtain
	\begin{align}
		\mathbf Q(\epsilon) &=  \sum_{n=1}^N \frac{-2 i b_p \sin\left(\frac{2 \pi n (p-1)}{N} \right)}{2 \epsilon + b_p \left(2\cos \left(\frac{2 \pi n (p-1)}{N} \right)-2 \right)} \  \mathbf w_{N-n} \mathbf w_n^\text T  \nonumber   \\
		&=  \sum_{n=1}^N \frac{-2 i b_p \sin\left(\frac{2 \pi n (p-1)}{N} \right)}{2 \epsilon -4 b_p \sin^2 \left( \frac{\pi n (p-1)}{N}\right)} \  \mathbf w_{N-n} \mathbf w_n^\text T \nonumber \\
		&= -i \   \sum_{n=1}^N \frac{ b_p \sin\left(\frac{2 \pi n (p-1)}{N} \right)}{\epsilon -2 b_p \sin^2 \left( \frac{\pi n (p-1)}{N}\right)} \  \mathbf w_{N-n} \mathbf w_n^\text T \nonumber.
	\end{align}

Next, since $\mathbf B(\epsilon)$ is circulant, recall $\exp(-t \mathbf B(\epsilon))$ is circulant for each $t \ge 0$ by \textbf{Theorem \ref{thm: Circulant Matrix Properties}}. Examining the explicit formula for the stationary covariance matrix $\mathbf S$ in \eqref{eq: OU Process Stationary Covariance Matrix}, we see that the integrand is a product of the circulant matrices $\exp(-t \mathbf B(\epsilon))$ and its transpose, which implies $\mathbf S$ is circulant.  Therefore, $\mathbf B(\epsilon) 
 \ \mathbf S$ and $\mathbf S \ \mathbf B^\text T(\epsilon)$ are both circulant, and since $\mathbf Q(\epsilon)$ is proportional to the difference  between $\mathbf B(\epsilon) \  \mathbf S$ and  $\mathbf S \ \mathbf B^\text T(\epsilon),$ we conclude $\mathbf Q(\epsilon)$ is circulant. 
\end{proof}

\subsection{Second Regime}
We now consider the second regime in which the perturbation constant $\epsilon>0$ is small.
\begin{theorem}[Cyclic OU Process One Nonzero Propagation Coefficient and Independent and Identical Noise in All Sensors Second Regime Leading Eigenvector Explicit Formula]\label{thm: Cyclic OU Process One Nonzero Propagation Coefficient and Independent and Identical Noise in All Sensors Second Regime Leading Eigenvector Explicit Formula}
Consider the lead matrix $\mathbf Q(\epsilon)$ defined in \eqref{eq: Cyclic OU Process One Nonzero Propagation Coefficient and Independent and Identical Noise in All Sensors Lead Matrix}.  Then, as $\epsilon \rightarrow 0,$ the leading eigenvector of $\mathbf Q(\epsilon)$ is $\mathbf w_{q},$ where $1 \le q \le N$ satisfies $q(p-1) \equiv 1 \mod N.$
\end{theorem}

\begin{proof}
First, we need to determine the eigenvalues of $\mathbf Q(\epsilon).$ In order to do this, we need to compute the first row of $\mathbf Q(\epsilon).$ Since $\mathbf Q(\epsilon)$ is skew-symmetric, its first row is the opposite of its first column. To obtain its first column, we compute the product of $\mathbf Q(\epsilon)$ and $\mathbf e_1,$ where $\mathbf e_1$ is the first standard basis vector in $\C^N.$ This means the first row of $\mathbf Q(\epsilon)$ is 
	\begin{align}
		- \mathbf Q(\epsilon) \ \mathbf e_1 &=  i \  \sum_{n=1}^N \frac{b_p \sin \left(\frac{2 \pi n (p-1)}{N} \right)}{ \epsilon -2 b_p \sin^2 \left(\frac{\pi n (p-1)}{N} \right)} \ \mathbf w_{N-n} \mathbf w_n^\text T \ \mathbf e_1 \nonumber \\
		 &= \frac{i}{\sqrt N}  \sum_{n=1}^N \frac{b_p \sin \left(\frac{2 \pi n (p-1)}{N} \right)}{ \epsilon -2 b_p \sin^2 \left(\frac{\pi n (p-1)}{N} \right)} \ \mathbf w_{N-n}.
	\end{align}
As a result, for each $1 \le q \le N,$ each eigenvalue of $\mathbf Q(\epsilon)$ is of the form
	\begin{align}
		\widetilde{\lambda_q}(\epsilon) &= \sqrt{N} \ \mathbf w^\text T_q  \left(  \frac{i}{\sqrt N}  \sum_{n=1}^N \frac{b_p \sin \left(\frac{2 \pi n (p-1)}{N} \right)}{ \epsilon -2 b_p \sin^2 \left(\frac{\pi n (p-1)}{N} \right)} \ \mathbf w_{N-n} \right) \nonumber \\
		&=   i \  \sum_{n=1}^N \frac{b_p \sin \left(\frac{2 \pi n (p-1)}{N} \right)}{ \epsilon -2 b_p \sin^2 \left(\frac{\pi n (p-1)}{N} \right)} \mathbf w_q^\text T  \ \mathbf w_{N-n} \nonumber \\
		&= i \   \sum_{n=1}^N \frac{b_p \sin \left(\frac{2 \pi n (p-1)}{N} \right)}{ \epsilon -2 b_p \sin^2 \left(\frac{\pi n (p-1)}{N} \right)} \delta_{N-q, N-n} \nonumber \\
		&=    \frac{b_p \sin \left(\frac{2 \pi q (p-1)}{N} \right)}{ \epsilon -2 b_p \sin^2 \left(\frac{\pi q (p-1)}{N} \right)} \ i \nonumber.
	\end{align}
	
Note that
	\begin{align}
		\lim_{\epsilon \rightarrow 0} 	\left|\widetilde{\lambda_q}(\epsilon) \right |&= \left| \frac{b_p \sin \left(\frac{2 \pi q (p-1)}{N} \right)}{  -2 b_p \sin^2 \left(\frac{\pi q (p-1)}{N} \right)} \right | \nonumber \\
		&=   \left| \frac{\sin \left(\frac{\pi q (p-1)}{N} \right) \cos \left(\frac{\pi q (p-1)}{N} \right)}{  \sin^2 \left(\frac{\pi q (p-1)}{N} \right)} \right | \nonumber  \\
		&=   \left|\cot \left(\frac{\pi q (p-1)}{N} \right) \right | \label{label}.
	\end{align}
We seek the index $q$ maximizing $ \left|\cot \left(\frac{\pi q (p-1)}{N} \right) \right |.$ Note that the map $\theta \mapsto |\cot(\theta)|$ is strictly decreasing on $(0,\frac{\pi}{2})$ and strictly increasing on $\left(\frac{\pi}{2}, \pi \right).$ Since $\cot(\theta)=\cot\left(\frac{\pi}{2} - \theta \right)$ for all $\theta \in (0,\frac{\pi}{2}),$ the map $\theta \mapsto |\cot(\theta)|$ restricted to the interval $(0,\pi)$ is symmetric about $\frac{\pi}{2}.$ Finally, note that $\theta \mapsto |\cot(\theta)|$ is $\pi$-periodic. The integers $u  \in \Z$ maximizing $\left | \cot(\frac{u \pi}{N}) \right|$ specifically in the domain $(0, \pi)$ are $u=1$ and $u=N-1.$ By periodicity, the integer $u \in \Z$ maximizing $ \left|\cot(\frac{u \pi}{N}) \right |$ on its entire domain satisfies $u \equiv \pm 1 \mod N.$

This means an index $1 \le q \le N$ that maximizes the eigenvalue limit \eqref{eq: Cyclic OU Process One Nonzero Propagation Coefficient and Independent and Identical Noise in All Sensors Lead Matrix} must satisfy either the congruence equation $q(p-1) \equiv 1 \mod N$ or the congruence equation $q(p-1) \equiv -1 \mod N.$ Both congruence equations have integer solutions because $\gcd(p-1,N),$ which is equal to $1$ by assumption, divides $\pm 1.$  

Let $1 \le q \le N$ be such that $q(p-1) \equiv 1 \mod N.$ Then, the associated eigenvector of the eigenvalue $\widetilde{\lambda}_q(\epsilon)$ is $\mathbf w_q,$ which is the leading eigenvector of $\mathbf Q(\epsilon).$ 
\end{proof}

Knowing the leading eigenvector of $\mathbf Q(\epsilon)$ under the second regime, we investigate whether the structure of the leading eigenvector recovers the structure of the cyclic sensor network.

\begin{theorem}[Cyclic OU Process One Nonzero Propagation Coefficient and Independent and Identical Noise in All Sensors Second Regime Leading Eigenvector Structure]\label{thm: Cyclic OU Process One Nonzero Propagation Coefficient and Independent and Identical Noise in All Sensors Second Regime Leading Eigenvector Structure}
Let $\mathbf w_q$ be the leading eigenvector of $\mathbf Q(\epsilon).$ Then, the permutation $\sigma \in S_N$ representing the cyclic order of the component phases $\mathbf w_q$ is defined by $\sigma(n) \equiv (N-(n-1)(p-1) )\mod N.$
\end{theorem}

\begin{proof}
First, we note that $\gcd(q,N)=1.$ To see this, recall the previous theorem states that $q(p-1) \equiv 1 \mod N.$ By definition, this means $$q(p-1)-Nk=1$$ for some integer $k.$ Hence, we can write $1$ as an integer linear combination of the numbers $q$ and $N,$ which implies $\gcd(q,N)=1.$

Let $\sigma \in S_N$ be the permutation such that 
 \begin{align}\label{eq: Congruence Equation}
	q(\sigma(n)-1) &= n-1 \mod N.
\end{align}
for each $1 \le n \le N.$ Firstly, we note $\sigma$ is well-defined. For each $1 \le n \le N,$ because $\gcd(q,N)=1$ divides $n-1,$ there is a unique solution to the congruence equation \eqref{eq: Congruence Equation} for $\sigma(n).$ 

For each $1 \le n <N,$ we have 
	\begin{align}
		q(\sigma(n+1)-\sigma(n)) &\equiv  (n+1-n) \mod N \nonumber \\
		& \equiv 1 \mod N \nonumber.
	\end{align}
Therefore, we deduce 
	\begin{align}
		q(\sigma(n+1)-\sigma(n)) \equiv q(p-1) \mod N \nonumber.
	\end{align}
Because $\gcd(q,N)=1,$ we may divide both sides of the congruence equation by $q$ to deduce that 
	\begin{align}
		\sigma(n+1) - \sigma(n) & \equiv p-1 \mod N.
	\end{align}

Finally, we claim $\sigma$ is a cyclic order permutation corresponding to the component principal arguments of the eigenvector $\mathbf w_q.$ Recall the $n$-th component of $\mathbf w_q$ is explicitly 
	\begin{align}
		w_{q,n} &= \frac{\omega_q^{n-1}}{\sqrt N} \nonumber  \\
		&= \frac{\omega_1^{q(n-1)}}{\sqrt N} \nonumber .
	\end{align}
By \eqref{eq: Congruence Equation}, we have 
	\begin{align}
		w_{q,\sigma(n)} &= \frac{\omega_1^{q(\sigma(n)-1)}}{\sqrt N} \nonumber  \\
		&= \frac{\omega_1^{n-1}}{\sqrt N} \nonumber ,
	\end{align}
which has principal argument $\frac{2 \pi (n-1)}N \in [0, 2 \pi).$ Therefore, our permutation $\sigma$ ensures the components of $\mathbf w_q$ are sorted in increasing order by their corresponding principal arguments. Thus, $\sigma$ is a valid cyclic order permutation.
\end{proof}

We interpret the result of the previous theorem. If we inject independent and identical noise into all sensors, then under the second regime, the structure of the leading eigenvector of $\mathbf Q(\epsilon)$ enables us to recover the expected network structure induced by the friction matrix $\mathbf B(\epsilon).$  In particular, the cyclic order of the phases of the leading eigenvector matches the expected cyclic order of the sensors receiving the signal in our propagation model.

\chapter{Experimental Results}\label{chap: experimentalresults}

In this chapter, we discuss the experimental results pertaining to the main topic of our thesis. Throughout, we consider an $N$-dimensional OU Process $\left \lbrace \mathbf x(t) \right \rbrace_{t \ge 0}$ with fixed model parameters $\mathbf B \in \Mat_{N,N}(\R)$ and $\boldsymbol \Sigma \in \Mat_{M,N}(\R).$ We let $\mathbf D$ be the diffusion matrix and $\mathbf Q$ be the lead matrix of the OU process. For the purposes of this chapter, we specifically refer to $\mathbf Q$ as the \textit{theoretical lead matrix} of the OU process.

\section{OU Process Realization Time Series Simulation}
In this section, we describe how to generate a time series approximating a realization of the OU process over a time interval of the form $[0,T],$ where $T>0$ is a fixed, large constant. We employ the \textit{Euler-Maruyama} method \cite{KloedenPlaten1992}, which we describe in detail. 

First, we fix a large integer $K,$ representing the number of iterations, and a small constant $\Delta>0,$ representing the time step size between consecutive times.  In the original governing SDE \eqref{eq: OU Process SDE} of the OU process, we replace the time differential $dt$ with $\Delta,$ the stochastic differential $d \mathbf x(t)$ with the difference $\mathbf x(t+\Delta)-\mathbf x(t),$ and the Wiener differential $d \mathbf w(t)$ with the difference $\mathbf w(t+\Delta)-\mathbf w(t).$ 

Recall the increments of the standard Wiener process are independent and Gaussian; in particular, $\mathbf w(t+\Delta)-\mathbf w(t)$ is an $M$-dimensional Gaussian random vector with mean $\mathbf 0$ and covariance matrix $\Delta \mathbf I,$ where $\mathbf I$ is the $M \times M$ identity matrix. So we generate a sequence of vectors $\left \lbrace \boldsymbol \xi_k \right \rbrace_{k=0}^{K-1}$ representing the Wiener increments,  in which all vectors are randomly and independently chosen according to the $M$-dimensional Gaussian distribution with mean $\mathbf 0$ and covariance matrix $\Delta \mathbf I.$  We let $\mathbf x_0 \in \R^N$ be an initial starting vector randomly chosen according to the Gaussian distribution with mean $\mathbf 0$ and covariance matrix $\mathbf S,$ where $\mathbf S$ solves the Lyapunov equation \eqref{eq: OU Process Stationary Covariance Matrix Lyapunov Equation}.

We now construct the time series $\left \lbrace \mathbf x_k \right \rbrace_{k =0}^{K-1},$ where $\mathbf x_k=\mathbf x(k \Delta)$ for each $0 \le k< K.$ Explicitly, we have
	\begin{align}\label{eq: OU Process Realization Time Series}
		\mathbf x_{k+1} &= \mathbf x_k - \Delta \mathbf B \ \mathbf x_k + \boldsymbol \Sigma \ \boldsymbol \xi_k
	\end{align}
for each $0 \le k < K-1.$ This constructed time series $\left \lbrace \mathbf x_k \right \rbrace_{k =0}^{K-1}$ is what we use to approximate a realization of the OU process over the specific time interval $\left[0 \ , \ (K-1) \  \Delta \right].$

In \textbf{Figure \ref{fig: SampleOURealization}}, we plot a sample time series approximating a realization of a $2$-dimensional (cyclic) OU process over the finite time interval $[0,100]$ with model parameters $\mathbf B= \begin{bmatrix}
	1.2 & -0.2 \\ -0.2 & 1.2
\end{bmatrix}$ and $\boldsymbol \Sigma$ equal to the $2 \times 2$ identity matrix. To generate the time series, we used the scheme \eqref{eq: OU Process Realization Time Series}  with $K=10001$ iterations and time step size $\Delta =0.01.$

\begin{figure}[h!]
	\centering
	\includegraphics[width=\textwidth]{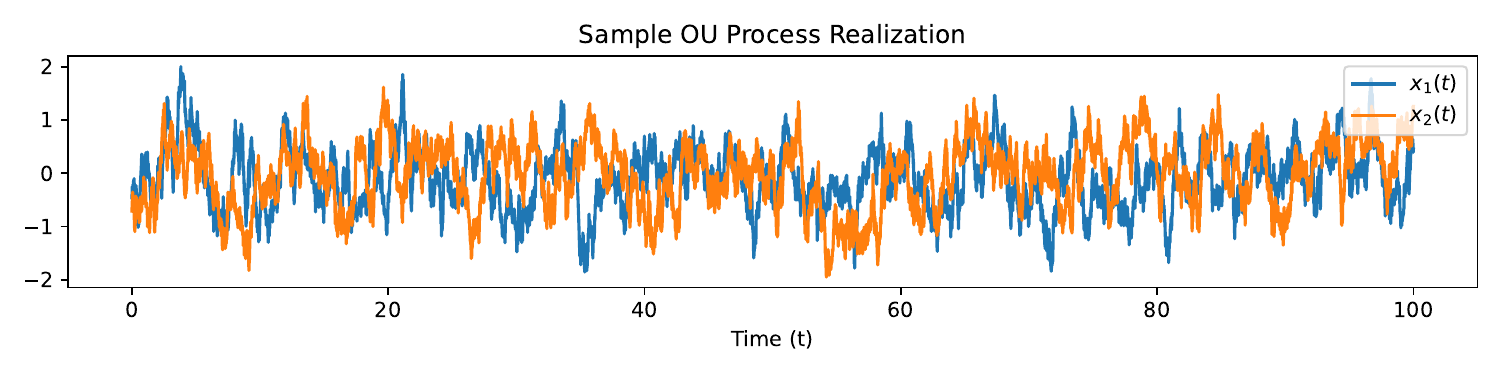}
	\caption{A sample two-dimensional time series approximating a realization of a two-dimensional OU process with model parameters $\mathbf B= \begin{bmatrix} 1.2 & -0.2 \\ -0.2 & 1.2 \end{bmatrix}$ and $\boldsymbol \Sigma$ equal to the $2 \times 2$ identity matrix. We plot the first component time series $x_1$ in blue and the second time series $x_2$ in orange.}\label{fig: SampleOURealization}
\end{figure}

\section{OU Process Empirical Lead Matrix Generation}

Throughout, we fix $K \in \N$ and $\Delta>0.$ Consider a time series  $\left \lbrace \mathbf x_k \right \rbrace_{k =0}^{K-1}$ produced according to our scheme in \eqref{eq: OU Process Realization Time Series}. We compute its \textit{empirical lead matrix}, denoted $\mathbf A(K,\Delta),$ via the shoelace formula mentioned in \eqref{eq: Discrete Lead Matrix}.

As a consequence of the Strong Law of Large numbers identity stated in \textbf{Theorem \ref{thm: Lead Process Law of Large Numbers Identity}}, for each fixed $\Delta>0,$ we have
	\begin{align} \label{eq: Discrete Ergodic Identity}
		\lim_{K \rightarrow \infty} \frac{\mathbf A(K,\Delta)}{(K-1) \Delta} &= \mathbf Q,
	\end{align}
for almost all realizations of the OU process, where $\mathbf Q$ is the theoretical lead matrix.

We show numerical evidence of the discrete law of large numbers identity \eqref{eq: Discrete Ergodic Identity} via an example. Consider the $5$-dimensional OU process with model parameters $\mathbf B=\text{Circ}(2.1, -0.2,-0.4,-0.6,-0.8)$  and $\boldsymbol \Sigma$ equal to the $5 \times 5$ identity matrix. In \textbf{Figure \ref{fig: SampleOURealizationLeadMatrix}}, for each $K \in \left \lbrace 10^2+1  \ , \ \dots \ , \ 10^6+1 \right \rbrace$ and each $\Delta \in \left \lbrace 10^{-1} \ , \ 10^{-2} \right \rbrace,$ we display a scatterplot, in which the first coordinate is the $(m,n)$-th entry of the time-averaged empirical lead matrix $\frac{\mathbf A(K,\Delta)}{(K-1) \Delta}$ and the second coordinate is the $(m,n)$-th entry of the theoretical lead matrix $\mathbf Q.$ We observe as $K$ increases, the scatterplot approaches the straight line $y=x,$ which means the entries of the time-averaged empirical lead matrix $\frac{\mathbf A(K,\Delta)}{(K-1) \Delta}$ approach the corresponding entries of the theoretical matrix $\mathbf Q$. Furthermore, our example shows that even when decreasing the time step size from $10^{-1}$ to $10^{-2}$, we still would need a large number of iterations $K$ if we want an accurate estimate of $\mathbf Q$ from a time-averaged empirical lead matrix.

\begin{figure}[h!]
	\centering
	\includegraphics[width=\textwidth]{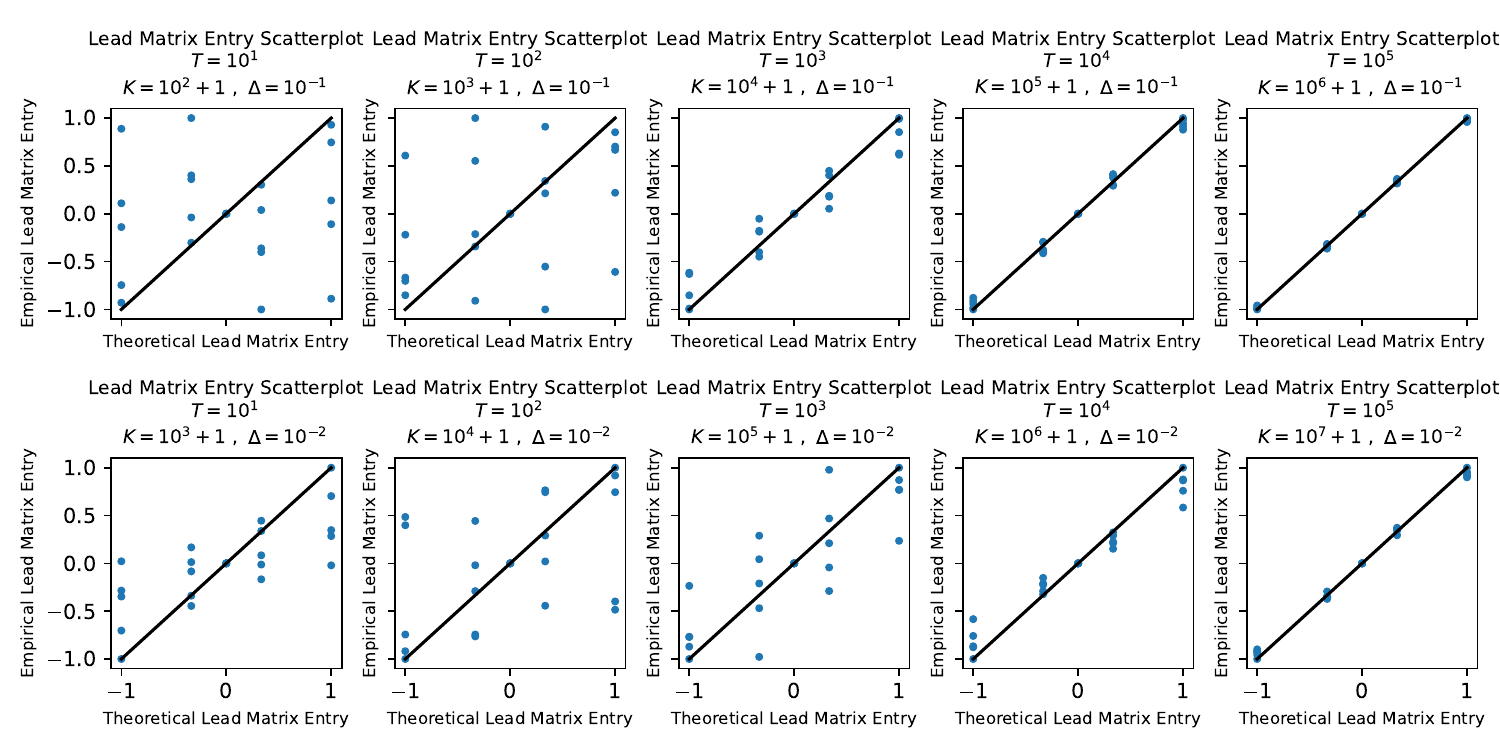}
	\caption{Various lead matrix entry scatterplots for a $5$-dimensional cyclic OU process with model parameters $\mathbf B=\text{Circ}(2.1, -0.2,-0.4,-0.6,-0.8)$  and $\boldsymbol \Sigma$ being the $5 \times 5$ identity matrix. In each scatterplot, we fix values of $K$ and $\Delta$ and plot the coordinates $\left(\frac{A_{m,n}(K,\Delta)}{(K-1) \Delta}, Q_{m,n} \right),$ where $\frac{A_{m,n}(K,\Delta)}{(K-1) \Delta}$ is the $(m,n)$-th entry of the time-averaged empirical matrix $\frac{\mathbf A(K,\Delta)}{(K-1) \Delta}$ and $Q_{m,n}$ is the $(m,n)$-th entry of the theoretical lead matrix $\mathbf Q.$  }\label{fig: SampleOURealizationLeadMatrix}
\end{figure}

\section{Cyclicity Analysis of the Cyclic OU Process}
We revisit the main problem statement of our thesis, in which we consider signal propagation model governed by the cyclic OU process with model parameters $\mathbf B(\epsilon)$ and $\boldsymbol \Sigma$ for each  $\epsilon>0,$ where
the circulant friction matrix $\mathbf B(\epsilon)$ is of the form
	\begin{align*}
		\mathbf B(\epsilon) &= \text{Circ}\left(\epsilon - \sum_{p=2}^N b_p \ , \ b_2 \ , \ \dots \ , \ b_N \right).
	\end{align*}
for some fixed propagation coefficients $b_2 \  , \ \dots \ , \ b_N \le 0.$ 

Throughout this section, we fix the dimension $N=100$ assume the first propagation coefficient $b_2$ is equal to $-1$ while all other propagation coefficients are $0.$  This means the expected network structure is such that the $n$-th sensor is receptive only to activity within the $(n+1)$-th sensor, its immediate neighbor to the right, in which $n+1$ is indexed mod $N$. We also assume $\boldsymbol \Sigma$ is a volatility matrix such that the diffusion matrix $\mathbf D$ is a diagonal matrix whose positive entries are equal. We let $d_s \in \lbrace 0, 1 \rbrace$ be the $s$-th diagonal entry of $\mathbf D.$ Recall $d_s=1$ corresponds to the situation where noise is injected into the $s$-th sensor. For different choices of $\boldsymbol \Sigma$ and perturbation  $\epsilon>0,$ in which $\epsilon$ is some power of $10,$ we  produce a time-averaged empirical lead matrix $\frac{\mathbf A(K,\Delta)}{(K-1) \Delta}$ corresponding to a realization of the $100$-dimensional cyclic OU process with model parameters $\mathbf B(\epsilon)$ and $\boldsymbol \Sigma.$ Here, we fix $K=10^6+1$ iterations and $\Delta=0.01$ as the time step size. We present numerical results as to whether the leading eigenvector of the empirical lead matrix enables us to recover the structure of the cyclic network structure induced by $\mathbf B(\epsilon).$

\subsection{Injection of Noise into Only One Sensor}

Suppose we inject noise into only one sensor.  Here, we consider the situation where we inject noise into the last sensor i.e. $d_{100}=1$ and $d_s=0$ for $1 \le s<100.$  We let $\epsilon \in  \left \lbrace 10^{-11}, 10^{-10}, 10^{-1},10^0,10^3,10^4 \right \rbrace.$ We chose these particular values of $\epsilon$ for a specific reason. For $\epsilon<10^{-11},$  the friction matrix $\mathbf B(\epsilon)$ is considered unstable due to precision issues. For $\epsilon>10^4,$ there was no change to the structure of the leading eigenvector. For the other listed values of $\epsilon,$ we noticed intermediary changes within the structure of the leading eigenvector.

In \textbf{Figure \ref{fig: SampleCyclicOUProcessOnlyOneNonzeroPropagationCoefficientNoiseInjectedinLastNodeLeadingEigenvectors}}, for each such value of $\epsilon,$ we plot the heatmap of a time-averaged empirical lead matrix $\frac{\mathbf A(K,\Delta)}{(K-1)\Delta}(\epsilon)$  and the logarithms of its eigenvalue moduli in descending order. We also plot the component phases and moduli of the leading eigenvector. In \textbf{Table \ref{tbl: SampleCyclicOUProcessOnlyOneNonzeroPropagationCoefficientNoiseInjectedinLastNodeLeadingEigenvectors}}, we record the ratio of the largest to third largest eigenvalue for each such empirical lead matrix.

\begin{figure}[h!]
	\centering
	\includegraphics[width=\textwidth]{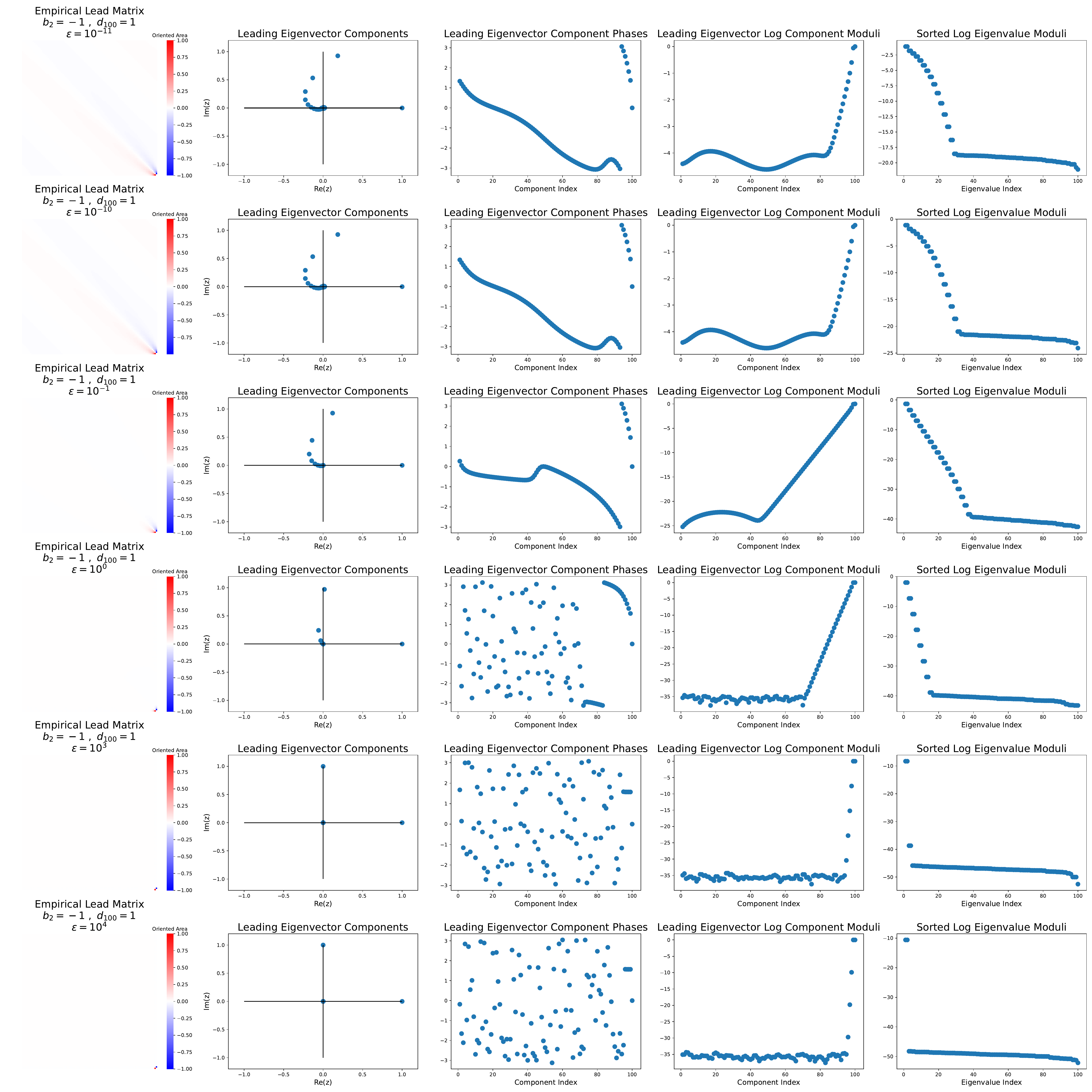}
	\caption{Time-averaged empirical lead matrices $\frac{\mathbf A(K,\Delta)}{(K-1) \Delta} (\epsilon)$ with their leading eigenvectors corresponding to a $100$-dimensional cyclic OU process with  $b_2=-1$ and $d_{100}=1,$ and $\epsilon \in \left \lbrace 10^{-11}, 10^{-10}, 10^{-1},1,10^3,10^4 \right \rbrace.$ Here, we chose $K=10^6+1$ iterations and step size $\Delta =10^{-2}$ in producing the time series representing the OU process realizations.} \label{fig: SampleCyclicOUProcessOnlyOneNonzeroPropagationCoefficientNoiseInjectedinLastNodeLeadingEigenvectors}
\end{figure}

\begin{table}[h!]
	\centering
	\begin{tabular}{||c | c||} 
		\hline
		$\epsilon$ & $\left|\lambda_1/\lambda_3 \right|$  \\ [0.5ex] 
		\hline\hline
		$10^{-11}$ & $\approx 2.00022$  \\ 
		\hline
		$10^{-10}$ & $\approx 2.00022$ \\
		\hline
		$10^{-1}$ & $\approx 8.00132$ \\
		\hline
		$1$ & $\approx 201.8365$ \\ 
		\hline 
		$10^3$ & $\approx 16064344115142$ \\
		\hline
		$10^4$ & $\approx 2.24629 \times 10^{16}$ \\
		\hline
	\end{tabular}
	
	\caption{The ratio of the first to third largest eigenvalue of each time-averaged empirical lead matrix  in \textbf{Figure \ref{fig: SampleCyclicOUProcessOnlyOneNonzeroPropagationCoefficientNoiseInjectedinLastNodeLeadingEigenvectors}}. }
	\label{tbl: SampleCyclicOUProcessOnlyOneNonzeroPropagationCoefficientNoiseInjectedinLastNodeLeadingEigenvectors}
\end{table}

\subsection{Observations}

For small $\epsilon,$ we observe the components of the leading eigenvector spiral around the origin. We observe that the last few component phases are decreasing. 
This means if the signal in our propagation model is broadcast from the last sensor, then the leading eigenvector correctly detects the next few sensors that receive the signal as it propagates.  However, we do not see overall decreasing behavior within the component phases corresponding to other indices. This means the global order of the component phases does not reflect the expected order of sensors receiving the propagating signal. Therefore, the phases alone suggest we cannot fully recover the overall network structure. Moreover, we observe all but the last few leading eigenvector component moduli are close to $0,$ which suggests as the signal is propagating, all but the last few sensors are receptive to the signal. Finally, we observe there is no clear separation between the largest eigenvalue and third largest eigenvalue of the empirical lead matrix, which suggests the largest eigenvalue does not actually dominate the spectrum of the lead matrix.
            
As $\epsilon$ gets larger, the spiraling pattern among the leading eigenvector components begins to disappear. The last leading eigenvector component stays put on the real axis. But the second to last leading eigenvector component gets closer to the imaginary axis, while all other leading eigenvector components get closer to the origin. We see chaotic behavior in all but the last two component phases. This is because while the last two leading eigenvector components approach the origin, they do so from different directions. This suggests that only the last two sensors are receptive to the propagating signal, which means we cannot fully recover the expected network structure. Moreover, there is a noticeable gap between the largest eigenvalue and the third largest eigenvalue of the empirical lead matrix, which means the empirical lead matrix can be well-approximated via a low rank matrix.

\subsection{Independent, Identically Distributed Noise Injected into Few Sensors}

Now, suppose we inject independent and identical noise into only a few sensors.  Consider the situation where we inject noise into the last $11$ sensors i.e. $d_{s}=1$ for all $90 \le s \le 100$ and $d_s=0$ for all other $s.$ We let $\epsilon \in  \left \lbrace 10^{-10}, 10^{-9}, 10^{-3},10^{-2},1,10^2,10^3 \right \rbrace.$ In \textbf{Figure \ref{fig: SampleCyclicOUProcessOnlyOneNonzeroPropagationCoefficientNoiseInjectedinLastElevenNodesLeadingEigenvectors}}, we plot the heatmap of a time-averaged empirical lead matrix $\frac{\mathbf A(K,\Delta)}{(K-1)\Delta}(\epsilon)$   and the logarithms of its eigenvalue moduli in descending order. We also plot the component phases and moduli of the leading eigenvector. In \textbf{Table \ref{tbl: SampleCyclicOUProcessOnlyOneNonzeroPropagationCoefficientNoiseInjectedinLastElevenNodesLeadingEigenvectors}}, we record the ratio of the largest to third largest eigenvalue of each empirical lead matrix.

\begin{figure}[h!]
	\centering
	\includegraphics[width=\textwidth]{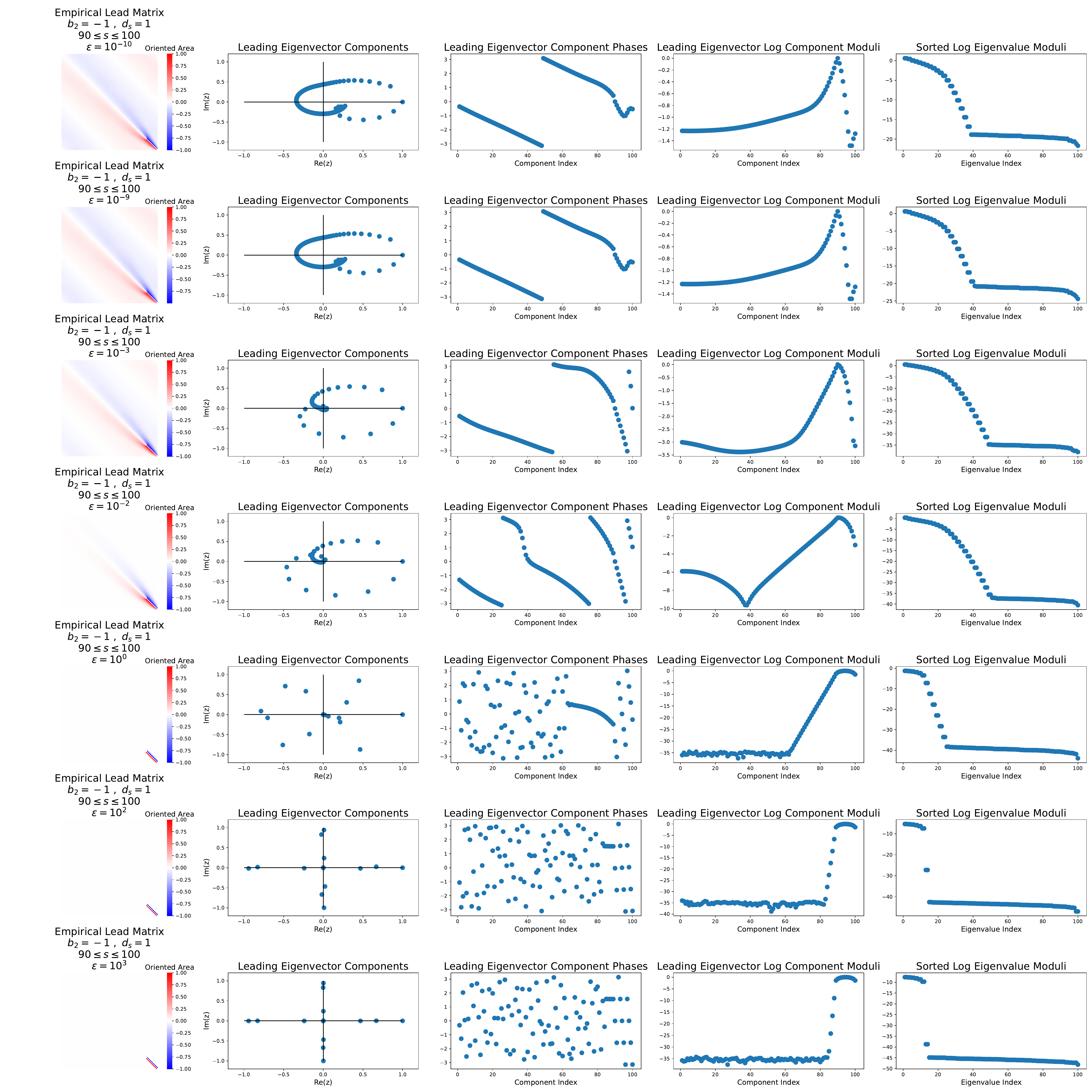}
	\caption{Time-averaged empirical lead matrices $\frac{\mathbf A(K,\Delta)}{(K-1) \Delta} (\epsilon)$ with their leading eigenvectors corresponding to a $100$-dimensional cyclic OU process with  $b_2=-1$ and $d_{s}=1$ with $90 \le s \le 100,$ and $\epsilon \in \left \lbrace 10^{-10}, 10^{-9}, 10^{-3},10^{-2},1,10,10^2,10^3 \right \rbrace.$ Here, we chose $K=10^6+1$ iterations and step size $\Delta =10^{-2}$ in producing the time series representing the OU process realizations. } \label{fig: SampleCyclicOUProcessOnlyOneNonzeroPropagationCoefficientNoiseInjectedinLastElevenNodesLeadingEigenvectors}
\end{figure}

\begin{table}[h!]
	\centering
	\begin{tabular}{||c | c||} 
		\hline
		$\epsilon$ & $\left|\lambda_1/\lambda_3 \right|$  \\ [0.5ex] 
		\hline\hline
		$10^{-10}$ & $\approx 1.2063$  \\ 
		\hline
		$10^{-9}$ & $\approx 1.2063$ \\
		\hline
		$10^{-3}$ & $\approx 1.3219$ \\
		\hline
		$10^{-2}$ & $\approx 1.5622$ \\ 
		\hline 
		$1$ & $\approx 1.1248$ \\
		\hline
		$10^2$ & $\approx 1.0965$ \\
		\hline
		$10^3$ & $\approx 0.9119$ \\
		\hline
	\end{tabular}
	
	\caption{We tabulate the ratios between the first and third largest eigenvalues corresponding to each time-averaged empirical lead matrix  in \textbf{Figure \ref{fig: SampleCyclicOUProcessOnlyOneNonzeroPropagationCoefficientNoiseInjectedinLastElevenNodesLeadingEigenvectors}}. }
	\label{tbl: SampleCyclicOUProcessOnlyOneNonzeroPropagationCoefficientNoiseInjectedinLastElevenNodesLeadingEigenvectors}
\end{table}

\subsection{Observations}

For small $\epsilon,$ the components of the leading eigenvector exhibit a visible spiral pattern. We notice, however, the last few leading eigenvector component phases are not decreasing. This suggests if the signal in our propagation model is broadcast from the $100$-th sensor, then we are not able to recover the correct order of the next few sensors receiving the signal. Moreover, we observe that the $90$-th leading eigenvector component modulus is the largest, which suggests the $90$-th sensor is most receptive to the propagating signal. Recall that the index $90$ is the smallest index corresponding to the sensor we inject noise. Despite the fact we inject noise into the $s$-th sensor for $90 \le s \le 100,$ the corresponding leading eigenvector component moduli suggest the other sensors receiving noise are not as receptive to the propagating signal. Finally, we observe there is no clear separation between the largest eigenvalue of the empirical lead matrix and third largest eigenvalue. 

As $\epsilon$ gets larger, we see the spiraling pattern within the leading eigenvector components disappears. The leading eigenvector components eventually lie on either the real axis, imaginary axis, or the origin.

\subsection{Independent, Identically Distributed Noise Injected into a Large Number of Sensors}

Suppose we inject noise into a larger number of sensors.  In particular, suppose we inject noise into the last $51$ sensors i.e. $d_{s}=1$ for all $50 \le s \le 100$ and $d_s=0$ for all other $s.$ We let $\epsilon \in  \left \lbrace 10^{-9}, 10^{-8}, 10^{-3},10^{-2},1,10^2,10^3,10^4 \right \rbrace.$ In \textbf{Figure \ref{fig: SampleCyclicOUProcessOnlyOneNonzeroPropagationCoefficientNoiseInjectedinLastFiftyNodesLeadingEigenvectors}}, we plot the heatmap of a time-averaged empirical lead matrix $\frac{\mathbf A(K,\Delta)}{(K-1)\Delta}(\epsilon)$   and the logarithms of its eigenvalue moduli in descending order. We also plot the component phases and moduli of the leading eigenvector. In \textbf{Table \ref{tbl: SampleCyclicOUProcessOnlyOneNonzeroPropagationCoefficientNoiseInjectedinLastFiftyNodesLeadingEigenvectors}}, we record the ratio of the largest to third largest eigenvalue for each such empirical lead matrix.

\begin{figure}[h!]
	\centering
	\includegraphics[width=\textwidth]{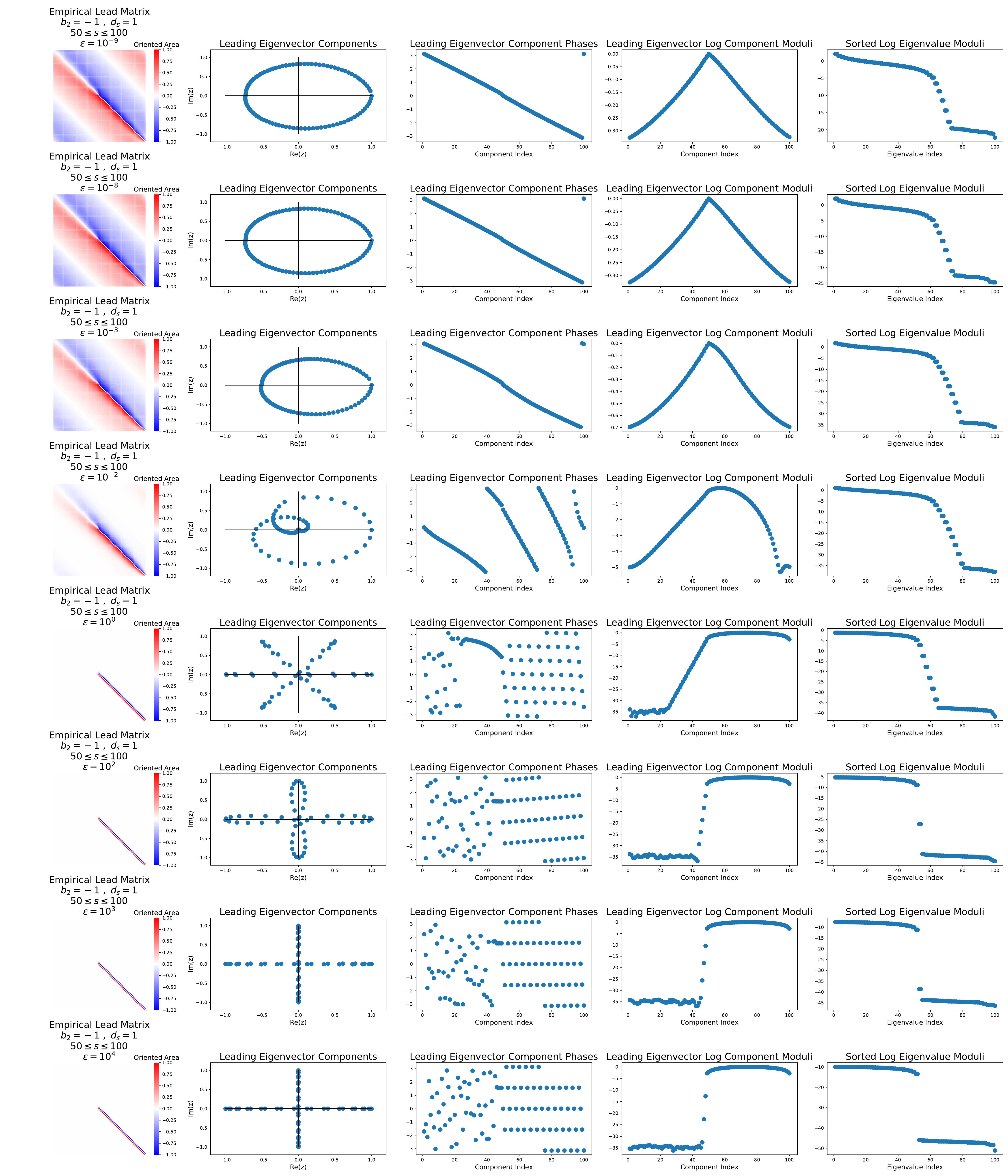}
	\caption{Time-averaged empirical lead matrices $\frac{\mathbf A(K,\Delta)}{(K-1) \Delta} (\epsilon)$ with their leading eigenvectors corresponding to a $100$-dimensional cyclic OU process with $b_2=-1$ and $d_{s}=1$ for $50 \le s \le 100$ and $\epsilon \in \left \lbrace 10^{-9}, 10^{-8}, 10^{-3},10^{-2},1,10^2,10^3,10^4 \right \rbrace.$ Here, we chose $K=10^6+1$ iterations and step size $\Delta =10^{-2}$ in producing the time series representing the OU process realizations. } \label{fig: SampleCyclicOUProcessOnlyOneNonzeroPropagationCoefficientNoiseInjectedinLastFiftyNodesLeadingEigenvectors}
\end{figure}

\begin{table}[h!]
	\centering
	\begin{tabular}{||c | c||} 
		\hline
		$\epsilon$ & $\left|\lambda_1/\lambda_3 \right|$  \\ [0.5ex] 
		\hline\hline
		$10^{-9}$ & $\approx 1.8548$  \\ 
		\hline
		$10^{-8}$ & $\approx 1.8548$ \\
		\hline
		$10^{-3}$ & $\approx 1.387$ \\
		\hline
		$10^{-2}$ & $\approx 1.194$ \\ 
		\hline 
		$1$ & $\approx 1.007$ \\
		\hline
		$10^2$ & $\approx 1.005$ \\
		\hline
		$10^3$ & $\approx 1.005$ \\
		\hline 
		$10^4$ & $\approx 1.005$ \\
		\hline
	\end{tabular}
	
	\caption{We tabulate the ratios between the first and third largest eigenvalues corresponding to each time-averaged empirical lead matrix  in \textbf{Figure \ref{fig: SampleCyclicOUProcessOnlyOneNonzeroPropagationCoefficientNoiseInjectedinLastFiftyNodesLeadingEigenvectors}}. }
	\label{tbl: SampleCyclicOUProcessOnlyOneNonzeroPropagationCoefficientNoiseInjectedinLastFiftyNodesLeadingEigenvectors}
\end{table}

\subsection{Observations}
For small $\epsilon,$ we observe the components of the leading eigenvector appear to form an elliptical pattern. The component principal arguments exhibit linear monotonically decreasing behavior.  This suggests the leading eigenvector phases do approximately reflect the expected structure of the network. Moreover, the component moduli are nonzero and appear to exhibit a symmetric pattern. Finally, we observe there is no clear separation between the first and third largest eigenvalues of the lead matrix.

When $\epsilon$ gets larger,  the elliptical pattern disappears, and the leading eigenvector components approach either the real axis, imaginary axis, or the origin. The last $50$ leading eigenvector component moduli are large, while the others are close to $0.$ Therefore, Cyclicity Analysis suggests we cannot fully recover the structure of the network.

\subsection{Independent, Identically Distributed Noise Injected into All Sensors}

Now, suppose we inject independent and identical noise into all sensors. This corresponds to the situation where $\boldsymbol \Sigma$ is a volatility matrix such that the diffusion matrix $\mathbf D$ is the identity matrix.

In \textbf{Figure \ref{fig: SampleCyclicOUProcessOnlyOneNonzeroPropagationCoefficientNoiseInjectedinAllNodesLeadingEigenvectors}}, we plot the heatmap of a time-averaged empirical lead matrix $\frac{\mathbf A(K,\Delta)}{(K-1)\Delta}(\epsilon)$  for each $\epsilon \in \left \lbrace 10^{-9}, 10^{-8}, 10^{-2}, 1, 10,10^3,10^4 \right \rbrace $  and the logarithms of its eigenvalue moduli in descending order. We also plot the component phases and moduli of the leading eigenvector. In \textbf{Table \ref{tbl: SampleCyclicOUProcessOnlyOneNonzeroPropagationCoefficientNoiseInjectedinLastFiftyNodesLeadingEigenvectors}}, we record the ratio of the largest to third largest eigenvalue for each such empirical lead matrix.

\begin{figure}[h!]
	\centering
	\includegraphics[width=\textwidth]{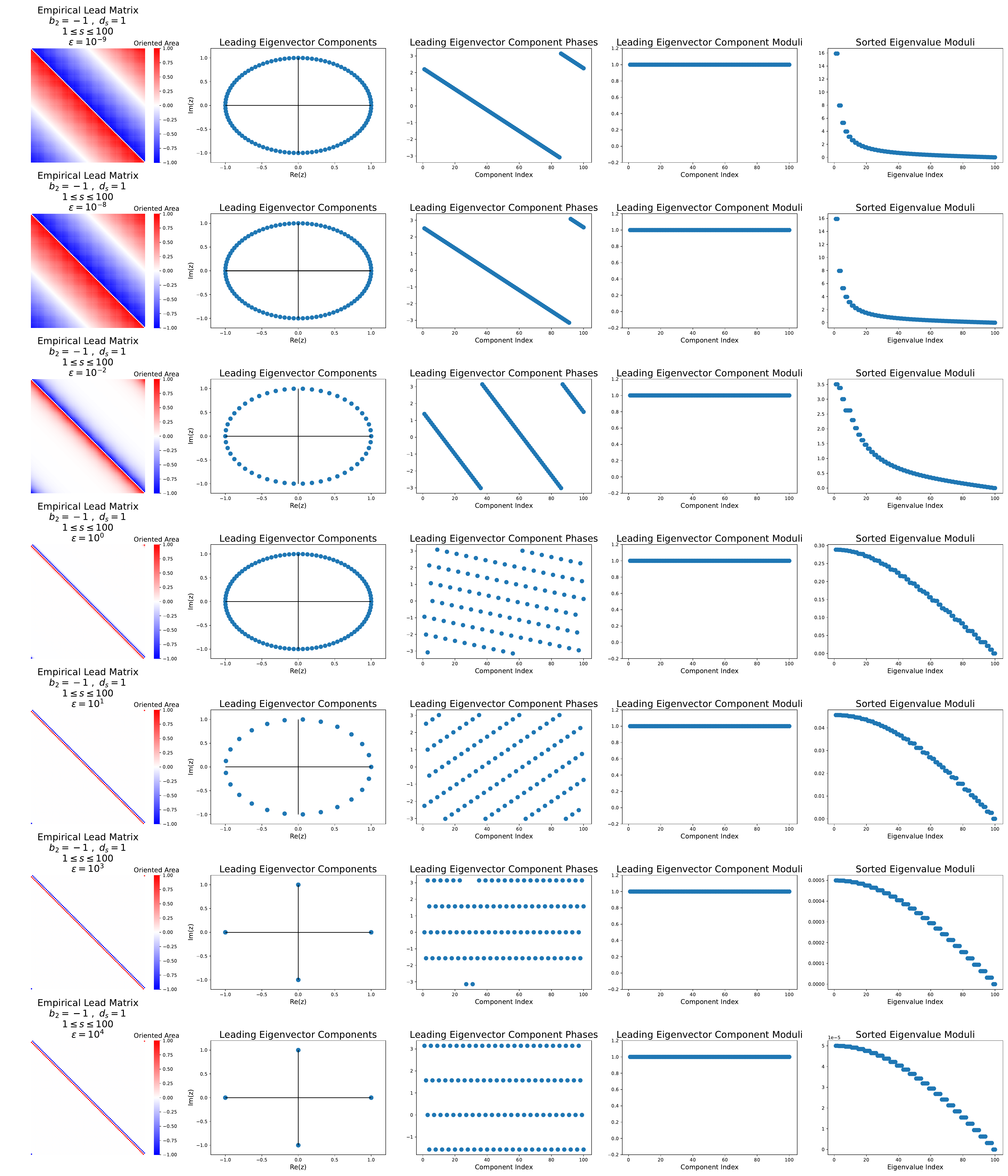}
	\caption{Time-averaged empirical lead matrices $\frac{\mathbf A(K,\Delta)}{(K-1) \Delta} (\epsilon)$ with their leading eigenvectors corresponding to a $100$-dimensional cyclic OU process with $b_2=-1$ and $d_{s}=1$ for $1 \le s \le 100$ and $\epsilon \in\left \lbrace 10^{-9}, 10^{-8}, 10^{-2}, 1, 10,10^3,10^4 \right \rbrace.$ Here, we chose $K=10^6+1$ iterations and step size $\Delta =10^{-2}$ in producing the time series representing the OU process realizations. } \label{fig: SampleCyclicOUProcessOnlyOneNonzeroPropagationCoefficientNoiseInjectedinAllNodesLeadingEigenvectors}
\end{figure}

\begin{table}[h!]
	\centering
	\begin{tabular}{||c | c||} 
		\hline
		$\epsilon$ & $\left|\lambda_1/\lambda_3 \right|$  \\ [0.5ex] 
		\hline\hline
		$10^{-9}$ & $\approx 2.001$  \\ 
		\hline
		$10^{-8}$ & $\approx 2.001$ \\
		\hline
		$10^{-2}$ & $\approx 1.036$ \\
		\hline
		$1$ & $\approx 1.0009$ \\ 
		\hline 
		$10$ & $\approx 1.0002$ \\
		\hline
		$10^3$ & $\approx 1.001$ \\
		\hline
		$10^4$ & $\approx 1.001$ \\
		\hline
	\end{tabular}
	
	\caption{We tabulate the ratios between the first and third largest eigenvalues corresponding to each time-averaged empirical lead matrix  in \textbf{Figure \ref{fig: SampleCyclicOUProcessOnlyOneNonzeroPropagationCoefficientNoiseInjectedinAllNodesLeadingEigenvectors}}. }
	\label{tbl: SampleCyclicOUProcessOnlyOneNonzeroPropagationCoefficientNoiseInjectedinAllNodesLeadingEigenvectors}
\end{table}

\subsection{Observations}

For small $\epsilon,$ the components of the leading eigenvector lie on a circle, as indicated by their moduli being equal. Observe the cyclic order of the leading eigenvector component phases correctly matches the order of sensors receiving the signal in our model. Therefore, Cyclicity Analysis enables us to fully recover the network structure. 

For large $\epsilon,$ the components of the leading eigenvector still lie on a circle, but they lie on either the positive real, positive imaginary, negative real, or the negative imaginary axis. Therefore, the leading eigenvector component phases do not enable us to recover the network structure. 

\chapter{Conclusion}

In this thesis, we posed very general questions pertaining to the lead lag dynamics amongst the components of an $N$-dimensional signal. Firstly, if two signals evolve similarly throughout time, which signal leads and which signal follows ? Secondly, if all signals evolve similarly throughout time, what is the order in which such signals evolve ? 

We discussed Cyclicity Analysis, an established way of answering the aforementioned two questions for cyclic signals, which are repeating, yet aperiodic, signals. Cyclicity Analysis is different from many standard tools in that it is time reparameterization invariant; the results of the analysis do not depend on how one observes the signals throughout time. It involves two parts. The first part is to construct an $N \times N$ lead matrix, in which the sign of the $(m,n)$-th entry heuristically captures the pairwise leader follower relationship between the $m$-th and $n$-th component signals. It is a specific construction based on iterated path integrals, which are the functionals of the signal $\mathbf x$ that are time reparameterization invariant. The second part of Cyclicity Analysis is to assume a model in which the component signals of $\mathbf x$ trace some periodic signal up to time reparameterization, scaling constants, and time shifts. Under this model, one determines the order in which the component signals evolve throughout time, which is the cyclic order of those time shifts. Under certain baseline situations, the leading eigenvector of the lead matrix can approximately recover the order in which the component signals of $\mathbf x$ evolve throughout time. More specifically, the cyclic order of the eigenvector component phases reflects the cyclic order of the offsets. 

In this thesis, we investigated Cyclicity Analysis in a novel setting. We considered an $N$-dimensional Ornstein-Uhlenbeck (OU) stochastic process with model parameters $\mathbf B$ and $\boldsymbol \Sigma,$ in which $\mathbf B$ is a friction matrix and $\boldsymbol \Sigma$ is a volatility matrix. We considered a signal propagation model governed by the cyclic OU process, in which the model parameter $\mathbf B$ is a circulant friction matrix. In this model, there is a signal propagating throughout a network of $N$ sensors. The matrix $\mathbf B$ represents the cyclic network structure of the sensors, while the diffusion matrix $\mathbf D = \boldsymbol \Sigma  \boldsymbol \Sigma^\text T/2$ is the covariance matrix whose entries are pairwise covariances of noises injected into any two sensors. 

To make Cyclicity Analysis suitable for analyzing the OU process, we defined an auxiliary $N \times N$ lead process, a matrix-valued stochastic process in which each random matrix in the collection represents a lead matrix corresponding to a realization of the OU process. The first main result of the thesis is that the lead process obeys a strong law of large numbers identity: its time average converges almost surely to a skew-symmetric matrix $\mathbf Q$.  

Then, with different choices of cyclic OU model parameters $\mathbf B$ and $\boldsymbol \Sigma,$ we investigated whether Cyclicity Analysis could enable us to recover the network structure induced by $\mathbf B$ under our signal propagation model. More specifically, we analyzed whether the structure of the leading eigenvector of $\mathbf Q$ could enable us to recover the network structure induced by $\mathbf B.$ 

For our analysis, we assumed the ground truth network structure is such that every sensor receives the propagating signal from exactly one sensor to its right.  In addition, we assumed independent and identical noise is injected into multiple sensors, which corresponds to the situation where $\mathbf D$ is a diagonal diffusion matrix whose nonzero diagonal entries are equal. 

With the imposed assumptions on $\mathbf B$ and $\boldsymbol \Sigma,$ our experimental results showed that Cyclicity Analysis was in fact able to partially recover the structure of the network. For example, in the situation where we inject noise only into the last node, the cyclic order of the component phases of the leading eigenvector of $\mathbf Q$ partially reflected the expected cyclic order of the sensors receiving the signal in our model. As we increase the number of sensors receiving noise, we observed the cyclic order of the component phases approximately match the expected cyclic order of the sensors. In the ultimate case identical noise is injected into all sensors, the cyclic orders perfectly match.

There is future work that can be done. For example, what if we remove the assumption that $\mathbf D$ has equal positive diagonal entries ? This would correspond to a situation where independent, but not identical, noise is injected into multiple sensors. What if $\mathbf B$ represents a network structure such that each sensor does not receive the signal from only one other sensor, but multiple sensors ? Does the leading eigenvector of the lead matrix $\mathbf Q$ capture those phenomena ?

%

\backmatter

%
\bibliographystyle{unsrt}
\bibliography{thesisrefs}

\begin{thebibliography}{10}

\bibitem{Stein2011}
Elias~M Stein and Rami Shakarchi.
\newblock {\em Fourier analysis: an introduction}, volume~1.
\newblock Princeton University Press, 2011.

\bibitem{Papoulis1967}
Athanasios Papoulis.
\newblock {\em The Fourier integral and its applications.}
\newblock McGraw-Hill New York, 1967.

\bibitem{BaryshnikovSchlafly2016}
Yuliy Baryshnikov and Emily Schlafly.
\newblock Cyclicity in multivariate time series and applications to functional
  mri data.
\newblock In {\em 2016 IEEE 55th conference on decision and control (CDC)},
  pages 1625--1630. IEEE, 2016.

\bibitem{AbrahamShahsavaraniZimmermanHusainBaryshnikov2021}
Ivan Abraham, Somayeh Shahsavarani, Benjamin Zimmerman, Fatima Husain, and
  Yuliy Baryshnikov.
\newblock Slow cortical waves through cyclicity analysis.
\newblock {\em bioRxiv}, 2021.

\bibitem{Chen1958}
Kuo-Tsai Chen.
\newblock Integration of paths--a faithful representation of paths by
  noncommutative formal power series.
\newblock {\em Transactions of the American Mathematical Society},
  89(2):395--407, 1958.

\bibitem{Chen1971}
Kuo-Tsai Chen.
\newblock Algebras of iterated path integrals and fundamental groups.
\newblock {\em Transactions of the American Mathematical Society},
  156:359--379, 1971.

\bibitem{Chen1973}
Kuo-Tsai Chen.
\newblock Iterated integrals of differential forms and loop space homology.
\newblock {\em Annals of Mathematics}, 97(2):217--246, 1973.

\bibitem{Chen1977}
Kuo-Tsai Chen.
\newblock Iterated path integrals.
\newblock {\em Bulletin of the American Mathematical Society}, 83(5):831--879,
  1977.

\bibitem{Lyons1998}
Terry~J Lyons.
\newblock Differential equations driven by rough signals.
\newblock {\em Revista Matem{\'a}tica Iberoamericana}, 14(2):215--310, 1998.

\bibitem{Lyons2007}
Terry~J Lyons, Michael Caruana, and Thierry L{\'e}vy.
\newblock {\em Differential equations driven by rough paths}.
\newblock Springer, 2007.

\bibitem{Axler2024}
Sheldon Axler.
\newblock {\em Linear algebra done right}.
\newblock Springer Nature, 2024.

\bibitem{AllgowerSchmidt1986}
Eugene~L Allgower and Phillip~H Schmidt.
\newblock Computing volumes of polyhedra.
\newblock {\em mathematics of computation}, 46(173):171--174, 1986.

\bibitem{OrnsteinUhlenbeck1930}
George~E Uhlenbeck and Leonard~S Ornstein.
\newblock On the theory of the brownian motion.
\newblock {\em Physical review}, 36(5):823, 1930.

\bibitem{GobetShe2016}
Emmanuel Gobet and Qihao She.
\newblock Perturbation of ornstein-uhlenbeck stationary distributions:
  expansion and simulation.
\newblock 2016.

\bibitem{Fasen2013}
Vicky Fasen.
\newblock Statistical estimation of multivariate ornstein--uhlenbeck processes
  and applications to co-integration.
\newblock {\em Journal of Econometrics}, 172(2):325--337, 2013.

\bibitem{Meucci2009}
Attilio Meucci.
\newblock Review of statistical arbitrage, cointegration, and multivariate
  ornstein-uhlenbeck.
\newblock 2009.

\bibitem{BartoszekFuentes-GonzalesMitovPienaarPiwczynskiPuchalkaSalikVoje2022}
Krzysztof Bartoszek, Jesualdo Fuentes-González, Venelin Mitov, Jason Pienaar,
  Marcin Piwczyński, Radosław Puchałka, Krzysztof Spalik, and Kjetil~Lysne
  Voje.
\newblock {Model Selection Performance in Phylogenetic Comparative Methods
  Under Multivariate Ornstein–Uhlenbeck Models of Trait Evolution}.
\newblock {\em Systematic Biology}, 72(2):275--293, 12 2022.

\bibitem{GodrecheLuck2018}
Claude Godr{\`e}che and Jean-Marc Luck.
\newblock Characterising the nonequilibrium stationary states of
  ornstein--uhlenbeck processes.
\newblock {\em Journal of Physics A: Mathematical and Theoretical},
  52(3):035002, 2018.

\bibitem{CourgaeuVeraart2022}
Valentin Courgeau and Almut~ED Veraart.
\newblock Likelihood theory for the graph ornstein-uhlenbeck process.
\newblock {\em Statistical Inference for Stochastic Processes}, pages 1--34,
  2022.

\bibitem{Poznyak2008}
Alexander~S. Poznyak.
\newblock Chapter 9 - stable matrices and polynomials.
\newblock In Alexander~S. Poznyak, editor, {\em Advanced Mathematical Tools for
  Automatic Control Engineers: Deterministic Techniques}, pages 139--174.
  Elsevier, Oxford, 2008.

\bibitem{Oksendal2003}
Bernt {\O}ksendal.
\newblock {\em Stochastic Differential Equations}, pages 65--84.
\newblock Springer Berlin Heidelberg, Berlin, Heidelberg, 2003.

\bibitem{GilsonKouvarisDecoZamora2018}
M.~Gilson, N.~E. Kouvaris, G.~Deco, and G.~Zamora-L\'opez.
\newblock Framework based on communicability and flow to analyze complex
  network dynamics.
\newblock {\em Phys. Rev. E}, 97:052301, May 2018.

\bibitem{Abraham2022}
Ivan~T Abraham.
\newblock {\em On geometric \& topological methods for analysis of biophysical
  time series data}.
\newblock PhD thesis, 2022.

\bibitem{LiWangSunLiu2022}
Yongli Li, Tianchen Wang, Baiqing Sun, and Chao Liu.
\newblock Detecting the lead--lag effect in stock markets: definition,
  patterns, and investment strategies.
\newblock {\em Financial Innovation}, 8(1):51, 2022.

\bibitem{MitraSnyderHackerRaichle2014}
Anish Mitra, Abraham~Z Snyder, Carl~D Hacker, and Marcus~E Raichle.
\newblock Lag structure in resting-state fmri.
\newblock {\em Journal of neurophysiology}, 111(11):2374--2391, 2014.

\bibitem{SakoeChiba1978}
Hiroaki Sakoe and Seibi Chiba.
\newblock Dynamic programming algorithm optimization for spoken word
  recognition.
\newblock {\em IEEE transactions on acoustics, speech, and signal processing},
  26(1):43--49, 1978.

\bibitem{SinghGhoshAdhikari2018}
Rajesh Singh, Dipanjan Ghosh, and R~Adhikari.
\newblock Fast bayesian inference of the multivariate ornstein-uhlenbeck
  process.
\newblock {\em Physical Review E}, 98(1):012136, 2018.

\bibitem{GilsonTagliazucchiConfre2023}
Matthieu Gilson, Enzo Tagliazucchi, and Rodrigo Cofr{\'e}.
\newblock Entropy production of multivariate ornstein-uhlenbeck processes
  correlates with consciousness levels in the human brain.
\newblock {\em Physical Review E}, 107(2):024121, 2023.

\bibitem{Qian2001}
Hong Qian.
\newblock Mathematical formalism for isothermal linear irreversibility.
\newblock {\em Proceedings of the Royal Society of London. Series A:
  Mathematical, Physical and Engineering Sciences}, 457(2011):1645--1655, 2001.

\bibitem{HornJohnson2012}
Roger~A Horn and Charles~R Johnson.
\newblock {\em Matrix analysis}.
\newblock Cambridge university press, 2012.

\bibitem{Gray2006}
Robert~M Gray et~al.
\newblock Toeplitz and circulant matrices: A review.
\newblock {\em Foundations and Trends{\textregistered} in Communications and
  Information Theory}, 2(3):155--239, 2006.

\bibitem{Sandric2017}
Nikola Sandri{\'c}.
\newblock A note on the birkhoff ergodic theorem.
\newblock {\em Results in mathematics}, 72:715--730, 2017.

\bibitem{CapassoBakstein2021}
Vincenzo Capasso and Devaid Bakstein.
\newblock {\em Introduction to Continuous-Time Stochastic Processes}.
\newblock Springer, 2021.

\bibitem{Khasminskii2012}
Rafail Khasminskii.
\newblock {\em Stochastic stability of differential equations}.
\newblock Springer, 2012.

\bibitem{Durrett2019}
Rick Durrett.
\newblock {\em Probability: theory and examples}, volume~49.
\newblock Cambridge university press, 2019.

\bibitem{MartinAhlfors1996}
D~Martin and LV~Ahlfors.
\newblock {\em Complex analysis}.
\newblock New York: McGraw-Hill, 1966.

\bibitem{Hua1982}
L-K Hua.
\newblock {\em Introduction to number theory}.
\newblock Springer Science \& Business Media, 1982.

\bibitem{Knuth1989}
Ronald~L. Graham, Donald~E. Knuth, and Oren Patashnik.
\newblock {\em Concrete Mathematics: A Foundation for Computer Science}.
\newblock Addison-Wesley, Reading, 1989.

\bibitem{Varga2011}
Richard~S Varga.
\newblock {\em Ger{\v{s}}gorin and his circles}, volume~36.
\newblock Springer Science \& Business Media, 2011.

\bibitem{TaoDentonZhang2022}
Peter Denton, Stephen Parke, Terence Tao, and Xining Zhang.
\newblock Eigenvectors from eigenvalues: A survey of a basic identity in linear
  algebra.
\newblock {\em Bulletin of the American Mathematical Society}, 59(1):31--58,
  2022.

\bibitem{NoschesePasquiniReichel2013}
Silvia Noschese, Lionello Pasquini, and Lothar Reichel.
\newblock Tridiagonal toeplitz matrices: properties and novel applications.
\newblock {\em Numerical linear algebra with applications}, 20(2):302--326,
  2013.

\bibitem{KloedenPlaten1992}
Peter~E. Kloeden and Eckhard Platen.
\newblock {\em Introduction to Stochastic Time Discrete Approximation}.
\newblock Springer Berlin Heidelberg, Berlin, Heidelberg, 1992.

\end{thebibliography}


\end{document}